\documentclass[12pt]{article}
\usepackage{amsmath,epsfig,amsfonts,bm,epic,graphics,pictex,
}

\DeclareMathOperator{\Dm}{Dm}
\begin{document}

\newtheorem{lem}{Lemma}[section]
\newtheorem{prop}{Proposition}
\newtheorem{con}{Construction}[section]
\newtheorem{defi}{Definition}[section]
\newtheorem{coro}{Corollary}[section]
\newcommand{\hf}{\hat{f}}
\newtheorem{fact}{Fact}[section]
\newtheorem{theo}{Theorem}
\newtheorem{conjec}{Conjecture}
\newcommand{\Br}{\Poin}
\newcommand{\Cr}{{\bf Cr}}
\newcommand{\dist}{{\rm dist}}
\newcommand{\diam}{\mathrm{diam}}
\newcommand{\compose}{\circ}
\newcommand{\dbar}{\bar{\partial}}
\newcommand{\Def}[1]{{{\em #1}}}
\newcommand{\dx}[1]{\frac{\partial #1}{\partial x}}
\newcommand{\dy}[1]{\frac{\partial #1}{\partial y}}
\newcommand{\Res}[2]{{#1}\raisebox{-.4ex}{$\left|\,_{#2}\right.$}}
\newcommand{\sgn}{{\rm sgn}}

\newcommand{\CC}{\mathbb{C}}
\newcommand{\D}{{\bf D}}
\newcommand{\RR}{\mathbb{R}}
\newcommand{\NN}{\mathbb{N}}
\newcommand{\HH}{\mathbb{H}}
\newcommand{\ZZ}{\mathbb{Z}}

\newcommand{\tr}{\mbox{Tr}\,}
\newcommand{\R}{{\bf R}}
\newcommand{\C}{{\bf C}}
\newcommand{\Leb}{\text{Leb}}

\newenvironment{nproof}[1]{\trivlist\item[\hskip \labelsep{\bf Proof{#1}.}]}
{\begin{flushright} $\square$\end{flushright}\endtrivlist}
\newenvironment{proof}{\begin{nproof}{}}{\end{nproof}}

\newenvironment{block}[1]{\trivlist\item[\hskip \labelsep{{#1}.}]}{\endtrivlist}
\newenvironment{definition}{\begin{block}{\bf Definition}}{\end{block}}

\newtheorem{com}{Comment}
\font\mathfonta=msam10 at 11pt
\font\mathfontb=msbm10 at 11pt
\def\Bbb#1{\mbox{\mathfontb #1}}
\def\lesssim{\mbox{\mathfonta.}}
\def\suppset{\mbox{\mathfonta{c}}}
\def\subbset{\mbox{\mathfonta{b}}}
\def\grtsim{\mbox{\mathfonta\&}}
\def\gtrsim{\mbox{\mathfonta\&}}

\newcommand{\tsing}{T_{\text{sing}}}
\newcommand{\dsing}{D_{\text{sing}}}
\newcommand{\tpar}{T_{\text{par}}}
\newcommand{\dpar}{D_{\text{par}}}
\newcommand{\thyp}{T_{\text{hyp}}}
\newcommand{\vhyp}{V_{\text{hyp}}}
\newcommand{\pr}{\mathbf{Pr}}

\title{Limit Drift for Complex Feigenbaum Mappings}

\author{Genadi Levin
\thanks{Supported in part by an ISF grant 1378/13}\\
\small{Einstein Institute of Mathematics}\\
\small{Hebrew University}\\
\small{Givat Ram 91904, Jerusalem, ISRAEL}\\
\small{\tt levin@math.huji.ac.il}\\
\and
Grzegorz \'{S}wia\c\negthinspace tek
\thanks{Supported in part by a grant 2015/17/B/ST1/00091 funded by Narodowe Centrum Nauki.}\\
\small{Department. of Mathematics and Information Science}\\
\small{Politechnika Warszawska}\\
\small{Koszykowa 75}\\
\small{00-662 Warszawa, POLAND}\\
\small{\tt g.swiatek@mini.pw.edu.pl}
}
\normalsize
\maketitle

\abstract{We study the dynamics of towers defined by fixed points of
  renormalization for Feigenbaum polynomials in the complex plane with
  varying order $\ell$ of the critical point. It is known that the measure of
  the Julia set of the Feigenbaum polynomial is positive if and only
  if almost every point tends to $0$ under the dynamics of the tower
  for corresponding $\ell$. That in turn depends on the sign of a quantity
  called the
  {\em drift}. We prove the existence and key properties of absolutely
  continuous invariant measures for tower dynamics as well as their
  convergence when $\ell$ tends to
  $\infty$. We also prove the convergence of the drifts to a finite
  limit which can be expressed purely in terms of the limiting tower
  which corresponds to a Feigenbaum map with a flat critical point.
}

\section{Introduction}

Our main object of interest will be the Feigenbaum functions which are the solutions of the
Feigenbaum-Coullet-Tresser fixed point
equation \cite{CT}, \cite{feig0}, \cite{feig1}:
\begin{equation}\label{equ:1hp,1}
  \tau H^2(x) = H(\tau x).
  \end{equation}
$H$ is assumed to be a smooth unimodal map on some interval which contains $0$ with the critical point of order $\ell$ and normalized (following \cite{EW}, \cite{leswi:feig}) so that the critical value is at $0$ and its image at $1$. It is well known (and very non-trivial, see e.g. \cite{leswi:feig} for a historical account) that for each $\ell$ even and positive
a unique solution $(H_{\ell},\tau_{\ell})$ exists and has the form $H_\ell(x)=E_\ell(x)^\ell$
where $E_\ell$ is a real-analytic mapping with strictly negative derivative on $[0,1]$ and with a unique zero $x_{0,\ell}$ (so that $x_{0,\ell}$ is the critical point of $H_\ell$ of order $\ell$).
Furthermore,
by~\cite{EpsLas}, \cite{leswi:feig}, $E_\ell$ extends to a univalent map (denoted again by $E_\ell$) from some Jordan domain $\Omega_\ell$ onto a slit complex plane (see Sect.~\ref{rev} for details). This implies in particular that $H_\ell$ has a polynomial-like extension onto some disk $D(0, R)$, $R>1$ with a single critical point of order $\ell$. Let $J_\ell$ be the Julia set of this polynomial-like mapping.

For every $\ell$ even and positive
the {tower} map (cf.~\cite{profesorus}) $\hat{T}_\ell: \CC\to \CC$ is defined almost everywhere as follows. Introduce the fundamental annulus $A_\ell=\Omega_\ell\setminus \tau_\ell^{-1}\overline{\Omega_\ell}$ (geometrically, this is indeed an annulus domain, for every finite $\ell$). For every $n\in \ZZ$ and every $z\in \tau_\ell^n A_\ell$, let
$$\hat{T}_\ell(z)=\tau_\ell^n H_\ell \tau_\ell^{-n}(z).$$
Note that
$\hat{T}_\ell(z)=H_\ell^{2^n}(z)$ for $n\ge 0$ and $z\in \tau_\ell^{-n}A_\ell$.
By \cite{EW}, \cite{leswi:feig}, the quadruple $(H_\ell, \tau_\ell, \Omega_\ell, x_{0,\ell})$, as $\ell\to \infty$, has a well-defined non-trivial limit $(H_\infty, \tau_\infty, \Omega_\infty, x_{0,\infty})$ so that the limit tower $\hat{T}_\infty: \CC\to \CC$ and $A_{\infty} := \Omega_{\infty}\setminus \overline{\tau_{\infty}^{-1}\Omega_{\infty}}$ are defined as well.

Main results of the present paper are summarized in the following Theorems~\ref{m1}-\ref{m2}.
\begin{theo}\label{m1}
For any $\ell\in 2\NN$ or $\ell=\infty$, there exists a unique measure $\mu_\ell$ supported on $A_{\ell}$ which satisfies the following conditions (1)-(2):

(1) $\mu_\ell$ is absolutely continuous w.r.t. the Lebesgue measure on
the plane and $\mu_\ell(A_\ell)=1$ and with a density which is
real-analytic and positive on $A_{\ell}$.

(2) $\hat{\mu}_\ell$ defined by $\hat{\mu}_{\ell}(S)=\mu_{\ell}\left(\tau_{\ell}^n S\right)$ for every Borel set $S\subset \tau_\ell^{-n}A_\ell$ and every $n\in \ZZ$ is a $\sigma$-finite measure on $\CC$ which is invariant under $\hat{T}_\ell: \CC\to \CC$.
\end{theo}
Define the level function $\hat{m}: \CC\to \ZZ$ and the map $T_\ell: A_\ell\to A_\ell$ so that
$\hat{m}(z)=n$ for $z\in\tau_\ell^{-n} A_\ell$ and
$T_\ell=\tau_\ell^{\hat{m}\circ H_\ell}H_\ell$.
Then Theorem~\ref{m1} means that $\mu_\ell$ is invariant under $T_\ell$.

Let $0<\ell< \infty$ be any even number. Since $T_\ell$ is $\mu_\ell$-ergodic and $\hat{m}\circ H_\ell$ is integrable, by the Birkhoff Ergodic Theorem, for Lebesgue almost every $z\in \CC$ the following limit (called 'drift') exists:
$$\vartheta(\ell):=\lim_{N\to \infty} \frac{1}{N}\hat{m}(\hat{T}_\ell^N(z))=\lim_{N\to \infty} \frac{1}{N}\sum_{i=0}^{N-1}\hat{m}\circ H_\ell(T_\ell^i(y))
=\int_{A_\ell}\hat{m}\circ H_{\ell}(x)\,d\mu_\ell(x),$$
where $y=\tau_\ell^kz\in A_\ell$, for an appropriate $k\in \ZZ$.
It follows from here, similar to~\cite{leswi:limit}, that the Lebesgue measure of the Julia set $J_\ell$ is positive if and only if $\vartheta(\ell)>0$.

We are interested in the behavior of $\vartheta(\ell)$ as $\ell$ tends to infinity.

\begin{theo}\label{m2}
  (1) The sequence of measures $\{\mu_\ell\}_{\ell\in 2\NN}$ tends strongly to $\mu_\infty$, i.e. the corresponding densities converge in $L_1(\CC,\Leb_2)$; moreover their convergence is analytic on some disk which contains the critical point of $H_{\ell}$ for all $\ell$ large enough.

   (2) the sequence of drifts $\{\vartheta(\ell)\}_{\ell\in 2\NN}$ converges to a finite number (the limit drift)
\begin{equation}\label{limdrift}
\vartheta(\infty)=-\frac{1}{\log\tau_\infty}\lim_{r\to 0}\int_{A_\infty\setminus B(x_{0,\infty},r)}\log\frac{|H_\infty(z)|}{|z|}d\mu_\infty(z).
\end{equation}
\end{theo}
Note that this integral exists only in the Cauchy sense given above
and is undefined on $A_{\infty}$, see~\cite{leswi:feig}.

The present paper is a sequel of \cite{leswi:feig},~\cite{leswi:hd},~\cite{leswi:measure},~\cite{leswi:common} and particularly \cite{leswi:limit}. 
The problem of \cite{leswi:limit} and the present paper can be tracked down to \cite{bkns} and \cite{ns}, see \cite{leswi:limit} for discussions.
In \cite{leswi:limit}, a formula for the limit drift which is similar to (\ref{limdrift}) is proved in a class of smooth covering circle maps. The proof (in the present paper) for the class of Feigenbaum maps follows similar lines, but substantially more technical.

\cite{leswi:limit} was a ground for a computer-assistant evaluation of the limit drift in the class of circle covers. The result shows that the limit drift in this class is negative which implies in particular that those maps of the circle with high enough criticalities do not have wild attractor.

In a recent preprint \cite{ds}, the authors present a computer-assisted proof that the area of $J_2$ is zero. The case of $\ell=2$ represents the opposite end
of the range of possibilities compared with our interest in $\ell$ that
tend to $\infty$.

\paragraph{Acknowledgement.} We thank the referees for their careful
reading of the manuscript and many good remarks.

\section{The Feigenbaum Function}
\subsection{Review of known properties.}\label{rev}
We consider the Feigenbaum-Coullet-Tresser fixed point
equation with the critical point of order $\ell$ even and positive,
set at some point $x_{0,\ell}\in (0,1)$ and normalized so that the critical value
is at $0$ and its image at $1$. The equation has the form of~(\ref{equ:1hp,1})
and $H$ is assumed to be unimodal on some interval which contains $0$
with Feigenbaum topological type.

It is well known that for each $\ell$
a unique solution $(H_{\ell},\tau_{\ell})$ exists.
We will now describe it following~\cite{leswi:feig}.

$H_{\ell}$ is a holomorphic map defined on a domain $\Omega_{\ell}$
which is a bounded topological disk symmetric with respect to the real line and
mapping into $\CC$. $\Omega_{\ell}$ can be split into two disks by an
arc $\mathfrak{w}_{\ell}$ which is tangent at $x_{0,\ell}$ to the line
$\{ z :\: \Re z = x_{0,\ell}\}$ and mapped by $H_{\ell}$ into the real
line. One can further observe that the image of $\mathfrak{w}_{\ell}$
is the positive half-line for $\ell$ divisible by $4$ and the negative
half-line otherwise.

The right connected component of $\Omega_{\ell}\setminus
\mathfrak{w}_{\ell}$ will be denoted by $\Omega_{+,\ell}$ and the left
one by $\Omega_{-,\ell}$. We will also write $H_{\pm,\ell}$ for $H_{\ell}$
restricted to $\Omega_{\pm,\ell}$.

\paragraph{Convergence as $\ell\rightarrow\infty$.}
When $\ell\rightarrow\infty$ triples $(H_{\ell},\tau_{\ell}, x_{0,\ell})$ converge
to a limit $(H_{\infty},\tau_{\infty}, x_{0,\infty})$ where $\tau_{\infty}>1$, $x_{0,\infty}\in(0,1)$ and $H_{\ell}$
converge to $H_{\infty}$ uniformly at least on the interval $[0,1]$. Mapping
$H_{\infty}$ is unimodal with the critical point at $x_{0,\infty}$ and $(H_{\infty},\tau_{\infty})$
satisfy the Feigenbaum equation~(\ref{equ:1hp,1}).

Furthermore, $H_{\infty}$ has a holomorphic continuation which is
similar to $H_{\ell}$. Namely, its domain is $\Omega_{\infty}$ which
is symmetric with respect to $\RR$ and
is the union of two bounded disks $\Omega_{\pm,\infty}$ with closures
intersecting exactly at $\{x_{0,\infty}\}$. We then define
restrictions $H_{\pm,\infty}$ to the corresponding
$\Omega_{\pm,\infty}$.

\paragraph{Holomorphic continuation.}
These mappings can then be
described by the following statement.

\begin{fact}\label{fa:1hp,1}
For every $\ell$ even and positive, the mapping $H_{\ell}$ only takes
the value $0$ at the critical point $x_{0,\ell}$ while the image of $H_{\infty}$
avoids $0$ at all. Subsequently, one can consider a pair of univalent mappings $\phi_{\pm,\ell}$, real and thus uniquely
determined by the condition
\[ \exp \left( \phi_{\pm,\ell} \right) = \tau_{\ell}^{-2} H_{\pm,\ell}\]
for $\ell$ which is even or infinite.
Then each $\phi_{\pm,\ell}$ maps the corresponding
$\Omega_{\pm,\ell}$ onto the set
\[ \varPi_{\ell} := \{ z\in\CC :\: |\Im z| < \frac{\ell\pi}{2} \} \setminus
[0,+\infty) \]
and is univalent.
\end{fact}

We can now formulate the convergence of mappings as
$\ell\rightarrow\infty$.

\begin{fact}\label{fa:1hp,2}
As $\ell$ tends to $\infty$ mappings $(\phi_{\pm,\ell})^{-1}$
converge to $(\phi_{\pm,\infty})^{-1}$ uniformly on compact subsets
of $\varPi_{\infty} := \CC \setminus [0,+\infty)$. \marginpar{??}
\end{fact}

For $\ell$ finite we will also consider an analytic continuation of mappings
$\phi_{\pm,\ell}$ which is described next.

\begin{fact}\label{fa:6hp,1}
Transformations $\phi_{\pm,\ell}$ for $\ell$ finite each have two univalent
analytic continuations, one with domain equal to $\Omega_{\ell} \cap
\HH_+$ and another one to $\Omega_{\ell} \cap\HH_-$ with ranges
$\{ z\in\CC :\: 0<\Im z<\ell\pi\}$ and $\{ z\in\CC :\: 0<\Im
z<\ell\pi\}$, respectively.
\end{fact}

\paragraph{Geometric properties of $\Omega_{\pm,\ell}$.}

Below we state a couple of properties which will be used.

\begin{fact}\label{fa:1hp,3}
  For any $\ell$ positive and even or infinite,
  \begin{itemize}
  \item
  \[ \overline{\Omega}_{\ell} \cap \RR = [ y_{\ell}, \tau_{\ell}
    x_{0,\ell} ] \]
  where $y_{\ell} < 0$ and $H_{\ell}(\tau^{-1}y_{\ell}) =
  \tau_{\ell}x_{0,\ell}$,
  \item
  \[   \overline{\Omega}_{\ell} \setminus  \tau_{\ell} \Omega_{-,\ell}=
  \{\tau_{\ell}\cdot x_{0,\ell}\} \]

\item
  \[ \overline{\Omega}_{\ell} \subset D(0,\tau_{\ell}) \; .\]
  \end{itemize}
\end{fact}

\paragraph{Associated mapping.}
\begin{defi}\label{defi:1hp,1}
  For any $\ell$ positive and even or infinite, define the {\em
    associated mapping}
  \[ G_{\ell}(z) = H_{\ell}(\tau_{\ell}^{-1}z) \]
  where $z\in\tau_{\ell}\Omega_{\ell}$.
  We also define the {\em principal inverse branch}
  $\mathbf{G}^{-1}_{\ell}$ which is defined on $\CC\setminus \{
  x\in\RR :\: x\notin [0,\tau^2_{\ell}]\}$ and fixes $x_{0,\ell}$.
\end{defi}

We list key properties of the associated mapping.
\begin{fact}\label{fa:1hp,4}
\begin{itemize}
\item
  $G_{\ell}$ has a fixed point at $x_{0,\ell}$ which is attracting for $\ell$
  finite and neutral for $\ell=\infty$.
\item
  the range of the principal inverse branch
  $\mathbf{G}^{-1}_{\ell}$ is contained in $\tau_{\ell}\Omega_{-,\ell}$.
\item
  \begin{eqnarray*}
    \mathbf{G}^{-1}_{\ell}(\Omega_{+,\ell}) &=& \Omega_{-,\ell} \\
    \mathbf{G}^{-1}_{\ell}(\Omega_{-,\ell}\setminus(-\infty,0]) &=&
\Omega_{+,\ell} \; .
   \end{eqnarray*}
\item
  $\tau_{\ell}^{-1} H_{\ell} = H_{\ell} G_{\ell}$ on $\Omega_{\ell}$.
\end{itemize}
\end{fact}

\paragraph{Coverings.}
\begin{fact}\label{fa:3ha,1}
A holomorphic mapping $\psi :\: U\rightarrow V$, where $U$ and $V$ are
domains on $\CC$, is a covering if and only if for every $v\in V$,
every simply-connected domain $W$ which contains $v$ and is compactly
contained in $V$ and every $u :\:
\psi(u)=v$ there exists a univalent inverse branch of $\psi$ defined
on $W$ which sends $v$ to $u$.
\end{fact}

\subsection{Analytic continuations.}
\begin{defi}\label{defi:3hp,1}
For $k :\: 0\leq k\leq \infty$ and $0<\ell\leq\infty$, let us define
\[ \varPi^k_{\ell} := \{ z\in\CC :\: |\Im z| < \frac{\ell\pi}{2} \} \setminus
\left( \left\{ 2j\log\tau_{\ell} :\: j=0,\cdots,k-1\right\} \cup [2k\log\tau_{\ell},+\infty)
  \right) \; .\]
\end{defi}
Thus $\varPi_{\ell}^0 = \varPi_{\ell}$ in the notation of Fact~\ref{fa:1hp,1},
while
\[ \varPi^{\infty}_{\ell} = \{ z\in\CC :\: |\Im z| < \frac{\ell\pi}{2} \} \setminus
\{ 2j\log\tau :\: j=0,1,\cdots\}  \; .\]

\begin{prop}\label{prop:2hp,1}
  For every $k\geq 0$ and every $\ell$ positive and even or infinite,
  there exist domains $\hat{\Omega}^k_{\pm,\ell}$ where
  $\hat{\Omega}^0_{\pm,\ell}=\Omega_{\pm,\ell}$, respectively. Furthermore
  $\phi_{\pm,\ell}$
  continue analytically to the corresponding
  $\hat{\Omega}^k_{\pm,\ell}$ with non-vanishing derivative and the
  claims below hold:
  \begin{itemize}
  \item
  $\hat{\Omega}^k_{+,\ell}$ and $\hat{\Omega}^k_{-,\ell}$ are
    disjoint,
  \item
    for $k>0$

   \[ \hat{\Omega}^k_{+,\ell} =
    G^{-1}_{\ell}\left(\hat{\Omega}^{k-1}_{-,\ell}
    \right) \]
    for $\ell$ finite and
   \[ \hat{\Omega}^k_{+,\infty} =
    G^{-1}_{\infty}\left(\hat{\Omega}^{k-1}_{-,\infty}\right) \cap
      \tau_{\infty}\Omega_{-,\infty}  \]
    for $\ell=\infty$, while for $k\geq 0$ one also has
    \[  \hat{\Omega}^k_{-,\ell} =
        \mathbf{G}^{-1}_{\ell}(\hat{\Omega}^{k}_{+,\ell}) \]
      for all $\ell$, where $\mathbf{G}^{-1}_{\ell}$ is the principal
      inverse branch, cf. Definition~\ref{defi:1hp,1}.
  \item
    \[ \phi_{\pm,\ell} :\: \hat{\Omega}^k_{\pm,\ell}\setminus
    \phi_{\pm,\ell}^{-1}\bigl(\{j\log\tau_{\ell}^2 :\: 0\leq j<k\}\bigr) \rightarrow
    \varPi^k_{\ell} \]
    is a covering.
  \end{itemize}
\end{prop}

\paragraph{Proof of Proposition~\ref{prop:2hp,1}.}
The proof will naturally proceed by induction with respect to $k$. For
$k=0$ all claims are known, in particular the second one follows from
Fact~\ref{fa:1hp,4} and the third from
Fact~\ref{fa:1hp,1}.

In an inductive step from $k-1$ to $k$, $\hat{\Omega}^k_{\pm,\ell}$
are already defined by the second claim. The first one is easy, since
the each of following inclusions implies the next one by the second claim:
\begin{eqnarray*}
z &\in& \hat{\Omega}^k_{+,\ell} \cap \hat{\Omega}^{k}_{-,\ell}\\
G_{\ell}(z) &\in& \hat{\Omega}^{k-1}_{-,\ell} \cap \hat{\Omega}^k_{+,\ell}\\
G^2_{\ell}(z) &\in& \hat{\Omega}^{k-1}_{+,\ell} \cap
\hat{\Omega}^{k-1}_{-,\ell} \; .
\end{eqnarray*}
obviously. Thus, we need to prove the third claim.

Let us begin with a lemma.
\begin{lem}\label{lem:2hp,1}
  For $\ell$ finite and even,
  \[ G_{\ell} :\: \tau\Omega_{\ell} \setminus
  \left( \{\tau_{\ell}x_{0,\ell}\} \cup G^{-1}_{\ell}\left([\tau_{\ell},\infty)\right)\right) \rightarrow
    \CC\setminus\left( \{0\} \cup [\tau_{\ell},\infty) \right) \]
      is a covering. For $\ell=\infty$, the corresponding claim is
      that
   \[ G_{\infty} :\: \tau_{\infty}\Omega_{-,\infty} \rightarrow
   \CC\setminus\left( \{0\} \cup [\tau_{\infty}^2,+\infty)\right) \]
     is a covering.
\end{lem}
\begin{proof}
 Let us deal with the case of $\ell$ finite. Since $G_{\ell} =
 H_{\ell} \tau^{-1}_{\ell}$, $G_{\ell}$ can be be viewed as composed
 of two branches one defined on $\tau_{\ell}\Omega_{-,\ell}$ and the
 other on $\tau_{\ell}\Omega_{+,\ell}$ which match analytically on the
 common boundary $\tau_{\ell}\mathfrak{w}_{\ell}$. By
 Fact~\ref{fa:1hp,1} $\log G_{\ell}$ is a univalent mapping of its
 domain with $\tau_{\ell}x_{0,\ell}$ removed onto
 $\CC$ with infinitely many slits of the form $\{x+iy :\: y=2\pi k,\,
 k\in \ZZ,\, x\geq X_k\}$ where $X_k$ is $\log\tau_{\ell}$ or $2\log\tau_{\ell}$
 depending on which branch of $G_{\ell}$ acts. A projection by $\exp$
 then yields the claim.

 A similar reasoning works for $\ell=\infty$ except that $G$ already
 maps $\tau_{\infty}\Omega_{-,\infty}$ univalently onto
 $\CC\setminus [\tau^2_{\infty},+\infty)$.
\end{proof}
\subparagraph{Mapping $\phi_{+,\ell}$.}
Since $\hat{\Omega}^{k-1}_{-,\ell} \cap \RR \subset
(-\infty,x_{0,\ell})$,
\[ G_{\ell} :\: \hat{\Omega}^k_{+,\ell}=G^{-1}_{\ell}(\hat{\Omega}^{k-1}_{-,\ell}) \rightarrow
\hat{\Omega}^{k-1}_{-,\ell} \setminus \{0\} \]
is a covering. Since $\phi_{-,\ell}(0)=\log\tau^{-2}_{\ell}$, it
follows that
\[ G_{\ell} :\: \hat{\Omega}^{k}_{+,\ell} \setminus
(\phi_{-,\ell}\circ G_{\ell})^{-1}\bigl( \{ \log
\tau_{\ell}^{-2} \} \bigr) \rightarrow
\hat{\Omega}^{k-1}_{-,\ell} \setminus \phi^{-1}_{-,\ell}\bigl(\{ \log
\tau_{\ell}^{-2}\bigr) \})
\]
is a covering as well. Furthermore,
\[ \phi_{-,\ell} :\: \hat{\Omega}^{k-1}_{-,\ell} \setminus \phi^{-1}_{-,\ell}(\{ \log
\tau_{\ell}^{2j} :\: j=-1,\cdots,k-2 \} ) \rightarrow \varPi^{k-1}_{\ell}\setminus \{ \log
\tau^{-2}_{\ell}\} = \varPi^k_{\ell} - 2\log\tau_{\ell}\]
is also a covering by the hypothesis of induction.

To prove that their composition is also a covering, take
$z\in U \subset \varPi^k_{\ell} - 2\log\tau_{\ell}$ and recall
Fact~\ref{fa:3ha,1}. The every inverse branch defined of
$\phi_{-,\ell}$ defined on $U$ into $\hat{\Omega}^{k-1}_{-,\ell}$
has a range which is a disk in $\hat{\Omega}^{k-1}_{-,\ell} \setminus
\phi^{-1}_{-,\ell}(\{ \log\tau_{\ell}^{-2} \})$ which therefore avoids
$0$. Since $G_{\ell}$ is a a covering of
$\hat{\Omega}^{k-1}_{-,\ell}\setminus\{0\}$, for every such disk one
can find an
inverse branch of
$G_{\ell}$. Hence,
\[ \phi_{-,\ell}\circ G_{\ell} :\: \hat{\Omega}^k_{+,\ell} \setminus
(\phi_{-,\ell}\circ G_{\ell})^{-1}\left(\{ \log
\tau_{\ell}^{2j} :\: j=-1,\cdots,k-2 \} \right) \rightarrow
\varPi^k_{\ell} - 2\log\tau_{\ell} \]
is a covering.
By the functional equation
\[ \phi_{+,\ell} = \phi_{-,\ell} \circ G_{\ell} + 2\log\tau_{\ell} \]
which completes the proof of the third claim for $\phi_{+,\ell}$.

\subparagraph{Mapping $\phi_{-,\ell}$.}
The associated map $G_{\ell}$ maps $\hat{\Omega}^k_{-,\ell}$
univalently onto $\hat{\Omega}^k_{+,\ell}$ so that
\[ \phi_{+,\ell}\circ G_{\ell} :\: \hat{\Omega}^k_{-,\ell} \setminus
\mathbf{G}_{\ell}^{-1}\Bigl( \phi_{+,\ell}^{-1}\bigl(\{ \log
\tau_{\ell}^{2j} :\: j=0,\cdots,k-1 \}\bigr)\Bigr)\rightarrow
\varPi^k_{\ell} \]
is clearly a covering. By the functional equation
\[ \phi_{-,\ell} = \phi_{+,\ell}\circ G_{\ell} \]
which concludes the proof of Proposition~\ref{prop:2hp,1}.

Let us state the main result about the analytic continuation of
$\phi_{\pm,\ell}$.

\begin{theo}\label{theo:3hp,1}
  For every $\ell$ positive and even or infinite there exist domains
  $\hat{\Omega}_{\pm,\ell}$, disjoint, simply-connected and symmetric
  with respect to $\RR$. The following inclusions hold:
  \begin{eqnarray*}
  \Omega_{-,\ell} \subset &\hat{\Omega}_{-,\ell} &\subset \tau_{\ell}\Omega_{-,\ell}\\
  \Omega_{+,\ell} \subset & \hat{\Omega}_{+,\ell} &\subset \tau_{\ell}\Omega_{\ell}\\
  \hat{\Omega}_{\pm,\infty} &\subset& \tau_{\infty}\Omega_{-,\infty}\; .
  \end{eqnarray*}

  Furthermore, $\phi_{\pm,\ell}$ continue analytically to
  $\hat{\Omega}_{\pm,\ell}$, respectively, with non-zero
  derivative and mappings
  \[ \phi_{\pm,\ell} :\: \hat{\Omega}_{\pm,\ell} \setminus
  \phi_{\pm,\ell}^{-1}\bigl( \{ \log \tau_{\ell}^{2j} :\:
  j=0,1,\cdots\}\bigr) \rightarrow \varPi^{\infty}_{\ell} \]
  are coverings, cf. Definition~\ref{defi:3hp,1}.
\end{theo}

\paragraph{Proof of Theorem~\ref{theo:3hp,1}.}
From the second claim of Proposition~\ref{prop:2hp,1} one easily
concludes each of the sequences
$\left(\hat{\Omega}^k_{\pm,\ell}\right)_{k=0}^{\infty}$ is a
non-decreasing sequence of simply-connected domain, symmetric with
respect to $\RR$ and contained in the appropriate component of the
domain of $G_{\ell}$. If we set
\[ \hat{\Omega}_{\pm,\ell} : = \bigcup_{k=0}^{\infty}
\hat{\Omega}^k_{\pm,\ell} \]
then the only claim of Theorem~\ref{theo:3hp,1} which is not obvious
concerns the maps being coverings.

We will use the criterion of Fact~\ref{fa:3ha,1}. Fix
$z\in\varPi^{\infty}_{\ell}$ and let $U$ be its simply-connected
neighborhood compactly contained in $\varPi^{\infty}_{\ell}$ and hence
bounded. The $U\subset \varPi^{k_0}_{\ell}$ for some finite
$k_0$. Then pick a preimage $u$ of $z$ which is also contained in some
$\hat{\Omega}_{\pm,\ell}^{k_1}$ with $k_1$ finite. Then by
Proposition~\ref{prop:2hp,1} there is an inverse branch of
$\mathbf{\phi}^{-1}_{\pm,\ell} :\: U \rightarrow
\hat{\Omega}^{\max(k_0,k_1)}_{\pm,\ell}$. Its range obviously avoids
the set $\phi_{\pm,\ell}^{-1}\bigl( \{ \log \tau_{\ell}^{2j} :\:
  j=0,1,\cdots\}\bigr)$. So, the condition of Fact~\ref{fa:3ha,1} is
  satisfied.

  \subsection{Special considerations for $\ell$ finite}
  We will write for $0\leq k\leq \infty$
  \[ \mathit{P}_{\ell}^{k} := \CC\setminus \left( \left\{ \tau_{\ell}^{2j} :\: 0\leq j<k\right\} \cup [\tau^{2k}_{\ell},+\infty)\right) .\]
\paragraph{Coverings and slits.}
  From Theorem~\ref{theo:3hp,1} we derive the following corollary.
  \begin{coro}\label{coro:6ha,1}
    For $\sigma=\pm$ mappings
    \begin{equation*}
      \exp(\phi_{\sigma,\infty}) :\: \hat{\Omega}_{\sigma,\infty}\setminus \left(\phi_{\sigma,\infty}\right)^{-1}\bigl( \{ \tau^{2j}_{\infty} :\: j=0,1,\cdots \}\bigr)  \rightarrow  \mathit{P}^{\infty}_{\infty}
    \end{equation*}
    are coverings and their domains are contained in
    $\tau_{\infty}\Omega_{-,\infty}$.
  \end{coro}
  \begin{proof} The mappings
    are seen to be coverings by Theorem~\ref{theo:3hp,1} using the criterion of Fact~\ref{fa:3ha,1}. The inclusion of domains also follows from Theorem~\ref{theo:3hp,1}.
  \end{proof}

  However, an analogous statement for finite $\ell$ instead of
  $\infty$ would be false for two reasons. First, the inclusion of
  domains would fail since $\hat{\Omega}_{+,\ell}$ extends to the
  right of $\tau_{\ell}x_{0,\ell}$ and secondly $\exp$ is not a
  covering of the plane by a vertical strip. We will now address these
  difficulties one by one.

  \paragraph{Restricting the domain of $H_{+,\ell}$.}
\begin{lem}\label{lem:6ha,1}
   Suppose that $\mathfrak{S}$ is either $(-\infty,0]$ or
   $[0,+\infty)$.
    Suppose $u_0\in\Omega_{+,\ell}$ is mapped  into
    \[ \mathbf{\Pi} := \{ z\in \varPi^{\infty}_{\ell}\setminus\mathfrak{S} :\: |\Im z|<\pi\} \]
   by $\phi_{+,\ell}$.
   Then there is a
   covering of $\mathbf{\Pi}$ by the analytic continuation of $\phi_{+,\ell}$
   restricted to a domain which is contained in $\hat{\Omega}_{+,\ell} \setminus
    [\tau_{\ell} x_{0,\ell},\infty)$ and contains $u_0$.
   \end{lem}
   \begin{proof}
   Recall that $\phi_{+,\ell} = \phi_{-,\ell} \circ G_{\ell} +
   2\log\tau_{\ell}$. Then $G_{\ell}$ maps $\hat{\Omega}_{+,\ell}\cap\RR =
   (x_0,\tau^2_{\ell} x_0)$ into $(0,\tau_{\ell})$. Subsequently, $(0,\tau_{\ell})$ is transformed by
   $\phi_{-,\ell}+2\log\tau_{\ell}$, inasmuch as it fits into its
   domain, into $(-\infty,0)$. Hence, when $\mathfrak{S} = (-\infty,0]$
    the domain of the covering does not intersect $\RR$ and the claim follows.

    When $\mathfrak{S}=[0,+\infty)$, the $\mathbf{\Pi} = \{ z\in\CC :\: |\Im z|<\pi\} \setminus [0,+\infty)$ is simply connected and
        the covering is univalent.
        Clearly, $\phi_{+,\ell}(u_0)$ can be connected to any point
        $x$ in the negative half-line by a path inside $\mathbf{\Pi}$
        which is otherwise disjoint from $\RR$. The lifting of this
        path by $\phi_{+,\ell}$ avoids $\partial\Omega_{\ell}$ except
        for the endpoint which is the preimage of $x$ by the
        covering. Then this preimage must be in the closure of $\Omega_{\ell}$
        which avoids $(\tau_{\ell}x_0,+\infty)$. Neither is
        $\tau_{\ell}x_{0,\ell}$ possible as the preimage, since it
        goes to $-\infty$ by $\phi_{+,\ell}$.
   \end{proof}

   \paragraph{Covering slit domains by $\tau^{-n}_{\ell}H_{\pm,\ell}$.}
   Now we will address the second difficulty.

   \begin{prop}\label{prop:6hp,1}
   Suppose that $\ell>0$ is even or infinite.
   Let $V$ be any one of $V^+,V^-,V^{\circ}$ where
   \begin{align*}
   V^-  := & \CC\setminus \bigl( \{ \tau^{2j}_{\ell} :\: j=1,\cdots\}
   \cup (-\infty,1] \bigr)\\
 V^+  := & \CC \setminus [0,+\infty)\\
   V^{\circ} := & \CC\setminus \bigl( [1,+\infty) \cup (-\infty,0] \bigr)
   \end{align*}
Let $\sigma=\pm$, $u_0\in\Omega_{\sigma,\ell}$ and suppose
$\exp\left(\phi_{\sigma,\ell}(u_0)\right) \in V$.

Then there is a domain $U, u_0\in U$ such that $\exp\phi_{\sigma,\ell} :\: U
\rightarrow V$ is a covering. $U$ is contained in
     \begin{itemize}
     \item
       $\hat{\Omega}_{-,\ell}\cup\Omega_{+,\ell} \setminus
       [x_{0,\ell},+\infty) $ when $\sigma=-$,
     \item
       $\hat{\Omega}_{+,\ell}\cup\Omega_{-,\ell} \setminus
       \bigl( (-\infty,x_{0,\ell}] \cup
       [\tau_{\ell}x_{0,\ell},+\infty) \bigr)
         $ when $\sigma=+$,
       \item
         $\CC\setminus\RR$ provided that $|\Im\phi_{\sigma,\ell}(u_0)|\geq\pi$.
     \end{itemize}
   \end{prop}
   \begin{proof}
   If $|\Im\phi_{\sigma,\ell}(u_0)|\geq\pi$, then
   $\phi_{\sigma,\ell}(u_0)$ belongs to a horizontal strip of width
   $2\pi$ which is mapped by $\exp$ univalently onto $V^+$ or
   $\CC\setminus (-\infty,0]$ which contains both $V^-, V^{\circ}$.
   $U$ is the preimage of the strip by $\phi_{\sigma,\ell}$ which may
       require the use of the extension from
       Fact~\ref{fa:6hp,1}. Accordingly, $U\subset\Omega_{\ell}\cap
       \HH_{\pm}$. This implies the inclusions postulated by
       Proposition~\ref{prop:6hp,1} and the covering in this case is
       just a univalent map.

    Similar case is $V=V^+$ since here again $\phi_{\sigma,\ell}(u_0)$
    is contained in a horizontal strip which is mapped onto $V^+$ by
    $\exp$ and $U$ is constructed in the same way.

   The remaining case is when $\left|
   \Im\phi_{\sigma,\ell}(u_0) \right| < \pi$
   and $V=V^-,V^{\circ}$ which implies $V\subset \CC\setminus
   (-\infty,0]$.   Then $V$ is equivalent by $\mathbf{\log}$ to a subset of
   $\mathbf{\Pi} := \{ z\in \varPi^{\infty}_{\ell} :\: |\Im z|<\pi$ and $U$ is chosen
 as a subset of $U'$, where $U'$
   covers $\mathbf{\Pi}$ by $\phi_{\sigma,\ell}$.
   By Theorem~\ref{theo:3hp,1} such a covering by $\phi_{\sigma,\ell}$
   exists with $U' \subset \hat{\Omega}_{\sigma,\ell}$. Moreover,
   by Lemma~\ref{lem:6ha,1}, $U'$ does not intersect
   $[\tau_{\ell}x_{0,\ell},+\infty)$.
   \end{proof}

\subsection{Convergence estimates.}
Fact~\ref{fa:1hp,2} states almost uniform convergence of mappings
$\phi_{\pm,\ell}^{-1}$. That is not good enough for our purposes.
The goal is the following statement.

\begin{prop}\label{prop:11jp,1}
Mappings $\phi_{\pm,\ell}^{-1}$ converge to $\phi^{-1}_{\pm,\infty}$
uniformly, i.e.
\[ \lim_{\ell\rightarrow\infty} \sup \bigl\{
\left|\phi^{-1}_{\pm,\ell}(\zeta)-\phi^{-1}_{\pm,\infty}(\zeta)\right|
:\: |\Im\zeta|<\frac{\ell\pi}{2},\, \zeta\notin [0,+\infty) \bigr\} =
  0 .\]
  Additionally,
  \[ \lim_{|\zeta|\rightarrow\infty} \phi^{-1}_{\pm,\infty}(\zeta) =
  x_{0,\infty} \; .\]
\end{prop}

The proof will be achieved in a sequence of Lemmas.

\begin{lem}\label{lem:17hp,1}
There exists $R_0>0$ such that for every $0<r\leq R_0$ there are
$\varepsilon(r)>0$ and $\ell(r)<\infty$ such that for every
$\ell\geq\ell(r)$ even and every $z\in \overline{B(x_{0,\infty},R_0)}
\setminus B(x_{0,\infty},r) $ the estimate
$|G^2_{\ell}(z)-z| \geq \varepsilon(r)$ holds.
\end{lem}
\begin{proof}
For $\ell$ sufficiently large and $z$ in a neighborhood of
$x_{0,\infty}$ we can represent
\begin{equation}\label{equ:17hp,1}
    G^2_{\ell}(z) - z = a_{0,\ell} + a_{1,\ell}\zeta + a_{2,\ell}\zeta^2 +
    a_{3,\ell}\zeta^3 + g_{4,\ell}(\zeta)\zeta^4
\end{equation}
    where $\zeta:= z-x_{0,\infty}$, $a_{0,\ell},a_{1,\ell},a_{2,\ell}$ all
tend to $0$ as $\ell\rightarrow\infty$, $\lim_{\ell\rightarrow\infty}
a_{3,\ell} = a_{3,\infty} < 0$ and $g_{4,\ell}(\zeta)$ are a sequence
of analytic functions convergent in a neighborhood of $0$ to
$g_{4,\infty}$, see Theorem 2 in~\cite{leswi:feig}.

Then $R_0$ is chosen so that $C(x_{0,\infty},R_0)$ fits inside
the required neighborhoods and $\left|a_{3,\infty}\right| > 3R_0
\sup\{|g_{4,\infty}(\zeta) :\: |\zeta|\leq R_0\}$. For $\ell$ large
enough
\[ \left|a_{3,\ell}\right| > 2R_0 \sup\Bigl\{\bigl|g_{4,\ell}(\zeta)\bigr| :\: |\zeta|\leq R_0\Bigr\}. \]  This implies that
\[ |a_{3,\ell}\zeta^3 + g_{4,\ell}(\zeta)\zeta^4| >
\frac{1}{2}|a_{3,\ell}|r^3 \]
for $\zeta\in C(0,r)$, $r\leq R_0$.  Set $\varepsilon(r) = \frac{r^3}{4} |a_{3,\infty}|$. The proof is finished by
choosing $\ell(r)$ so that for all $\ell\geq \ell(r)$
\begin{align*}
  |a_{0,\ell}| + |a_{1,\ell}|R_0 + |a_{2,\ell}|R_0^2 < & \frac{\varepsilon(r)}{2}\\
  \frac{|a_{3,\ell}|}{|a_{3,\infty}|} > & \frac{3}{4} .
\end{align*}
\end{proof}

\begin{lem}\label{lem:10jp,1}
  \begin{multline*}
    \forall R_0,\epsilon>0\; \exists L(\epsilon),\ell(\epsilon)<\infty\\
    \forall \ell :\: \ell(\epsilon)\leq\ell\leq\infty\; \forall
    \zeta :\: |\Im\zeta|<\frac{\ell\pi}{2},\, \dist\left(\zeta, [0,+\infty)\right) > L(\epsilon) \\
      \left| G^2_{\ell}\circ\phi^{-1}_{\pm,\ell}(\zeta) - \phi^{-1}_{\pm,\ell}(\zeta) \right| < \epsilon .
  \end{multline*}
\end{lem}
\begin{proof}
  If $0\leq \Im\zeta < \frac{\ell\pi}{2}$, then for any $M$ if $\ell$ is sufficiently large and finite and the distance from $\zeta$
  to the slit $[),+\infty)$ is sufficiently large as well, then
  points $\zeta$ and  $\zeta-\log\tau_{\ell}^2$ can be surrounded by an annulus with modulus $M$ in at least one of the domains
  \begin{equation}\begin{split}
  \left\{ \xi\in\CC\setminus [0,+\infty) :\:
  |\Im\xi| < \frac{\ell\pi}{2}\right\}\;  \text{or}\\
  \left\{ \xi\in\CC\setminus [0,+\infty) :\:
  0<\Im\xi < \ell\pi \right\}
  \end{split}\end{equation}
  When $\ell=\infty$ the same is true and the first type domain which is $\CC\setminus [0,+\infty)$ suffices.

  The first type domains
    are mapped by $\phi^{-1}_{\pm,\ell}$ univalently on $\Omega_{\pm,\ell}$
    by Fact~\ref{fa:1hp,1} and those of the second type to $\Omega_{\ell} \cap \HH_+$
    by Fact~\ref{fa:6hp,1}. Also, $G_{\ell}^2\circ\phi^{-1}_{\pm,\ell}(\zeta) = \phi_{\pm,\ell}^{-1}(\zeta-\log\tau^2_{\ell})$.
    Any of those domains has diameter bounded independently of $\ell$ by Fact~\ref{fa:1hp,3}. Hence, the distance between
    $G_{\ell}^2\circ\phi^{-1}_{\pm,\ell}(\zeta)$ and $\phi_{\pm,\ell}^{-1}(\zeta)$ is less than $\epsilon$ if $M\geq M(\epsilon)$.
\end{proof}

\begin{lem}\label{lem:10jp,2}
  \begin{multline*}
    \forall r>0\; \exists L(r),\ell(r)<\infty\; \forall \ell\geq \ell(r)\; \forall
    \zeta :\: |\Im\zeta|<\frac{\ell\pi}{2},\, \dist\left(\zeta, [0,+\infty)\right) > L(r) \\
      \left| \phi^{-1}_{\pm,\ell}(\zeta) - x_{0,\infty} \right| \leq r .
  \end{multline*}
\end{lem}
\begin{proof}
Fix $R_0$ from Lemma~\ref{lem:17hp,1} and take any $r :\: r\in(0,R_0)$.
From Lemma~\ref{lem:17hp,1} we then get $\varepsilon(r)$ and set $\epsilon :=
\varepsilon(r)/2$ in Lemma~\ref{lem:10jp,1}. The bound $\ell(r)$ can now be
fixed so that for $\ell\geq\ell(r)$ both Lemmas apply and $|x_{0,\ell}-x_{0,\infty}|<\frac{r}{2}$. Set also
$L(r) := L(\epsilon)$ given by Lemma~\ref{lem:10jp,1}.

For any $\ell\geq\ell(r)$ we consider the set
\[ S_{\ell} := \bigl\{ \zeta\in\CC :\: |\Im\zeta|<\frac{\ell\pi}{2},\,\dist\left(\zeta,[0,+\infty)\right)>L(\epsilon),\,
\left|\phi_{\pm,\ell}^{-1}(\zeta) -x_{0,\infty}\right| \leq r \bigr\} .\]

$S_{\ell}$ is obviously closed in $\bigl\{ \zeta\in\CC :\: |\Im\zeta|<\frac{\ell\pi}{2},\,\dist\left(\zeta,[0,+\infty)\right)>L(\epsilon)$ and also non-empty
  since $\lim_{x\rightarrow\infty} \phi^{-1}_{\pm,\ell}(x) = x_{0,\ell}$.
The proof is finished once we have shown that $S_{\ell}$ is also
open. If $\zeta$ is a non-interior point, we must have
$\left|\phi_{\pm,\ell}^{-1}(\zeta) -x_{0,\infty}\right| = r$. Then by
Lemma~\ref{lem:17hp,1} we get $\left|
G_{\ell}^2\circ\phi_{\pm,\ell}^{-1}(\zeta) -
\phi^{-1}_{\pm,\ell}(\zeta)\right| \geq \varepsilon(r) > \epsilon$.
But by Lemma~\ref{lem:10jp,1},   $\left|
G_{\ell}^2\circ\phi_{\pm,\ell}^{-1}(\zeta) -
\phi^{-1}_{\pm,\ell}(\zeta)\right| \geq \varepsilon(r) < \epsilon$
for all $\zeta\in S_{\ell}$. Hence, there are no non-interior points.
\end{proof}

\begin{lem}\label{lem:11jp,2}
  For any $L,\epsilon>0$ define the set
  \[ V(L,\epsilon) :=    \bigl\{ \zeta\in\CC :\: \Re\zeta \leq
  L,\,|\Im\zeta|<\pi,\,\dist\left(\zeta,[0,\infty)\right)\geq \epsilon\bigr\} .\]
  Then the family $\left(\phi_{-,\ell}\right)_{\ell}$, $\ell$ positive
  and even or infinite, is
  equicontinuous on $V(L,\epsilon)$ and converges uniformly to
  $\phi_{-,\infty}^{-1}$.
\end{lem}
\begin{proof}
Let us begin by proving that for each $\ell$,
$\phi^{-1}_{-,\ell}$ is uniformly continuous on
$V(L,\epsilon)$. That is clear on the set
$V(L',L,\epsilon) := \{ \zeta\in V(L,\epsilon) :\: \Re\zeta \geq L'\}$
for any $L'$ by
compactness and in particular since $\phi_{-,\ell}^{-1}$ extends
through each line $\{ \zeta :\: \Im\zeta = \pm \pi\}$ by
Fact~\ref{fa:6hp,1}. It remains to see that
\begin{equation}\label{equ:11jp,2}
  \lim_{\Re\zeta\rightarrow-\infty} \phi_{-,\ell}^{-1}(\zeta) =
  x_{0,\ell} .
\end{equation}
  This is the case when $\zeta$ is real and any other
sequence of points remains in a bounded hyperbolic distance from $\RR$
in an extended domain $-\frac{\pi}{2} \leq \Im\zeta \leq
\frac{3}{2}\pi$ or its symmetric image.

Now equicontinuity will follows if we show uniform convergence. That
again is clear on $V(L',L,\epsilon)$ for any $L'$ by
Fact~\ref{fa:1hp,2}. On the set $V(L,\epsilon)\setminus
V(L',L,\epsilon)$ we conclude from Lemma~\ref{lem:10jp,2} that for any
$r>0$ there is $L'(r)$ sufficiently close to $\infty$ and $\ell(r)$
that $\phi^{-1}_{-,\ell}\left( V(L,\epsilon)\setminus
V(L'(r),L,\epsilon) \right) \subset D(x_{0,\infty},r)$ for all
$\ell\geq\ell_0(r)$. Uniform convergence follows.
\end{proof}

Recall the principal inverse branch $\mathbf{G}_{\ell}^{-1}$,
cf. Definition~\ref{defi:1hp,1}.

\begin{lem}\label{lem:11jp,1}
  For any $L,\epsilon>0$ define the set
  \[ W(L,\epsilon) := \bigl\{ z\in\CC :\: |z|\leq L,\,\dist\left(z,[\tau_{\infty},\infty)\right)\geq \epsilon\bigr\} \setminus (-\infty,0] .\]
  Then for some $\ell_0$ the sequence
  $\left(\mathbf{G}_{\ell}^{-1}\right)_{\ell=\ell_0}^{\infty}$ is
  equicontinuous and
  converges to $\mathbf{G}^{-1}_{\infty}$ uniformly on
  $W(L,\epsilon)$.
\end{lem}
\begin{proof}
The basis of the proof is the representation
\begin{equation}\label{equ:11jp,1}
  \mathbf{G}_{\ell}^{-1}(z) =
  \tau_{\ell}\phi_{-,\ell}^{-1}\left(\log(z) - \log\tau_{\ell}^2\right)
\end{equation}
where the principal branch of the $\log$ is used.
Then uniform convergence follows from the
representation~(\ref{equ:11jp,1}) and Lemma~\ref{lem:11jp,2}.

It remains to show uniform continuity of $\mathbf{G}_{\ell}^{-1}$ for
each $\ell$. Here $\ell_0$ should be chosen so that for
$\ell\geq\ell_0$ the difference
$\left|\log\tau^2_{\infty}-\log\tau^2_{\ell}\right|$ is less than
$\epsilon/2$. Then uniform continuity also follows from
Lemma~\ref{lem:11jp,2} on a set
where $\log(z)$ is uniformly continuous, which is outside of
$D(0,\eta),\eta>0$. Additionally, by formula~(\ref{equ:11jp,2})
$\mathbf{G}_{\ell}^{-1}$ can be extended continuously to $0$ by
setting $\mathbf{G}_{\ell}^{-1}(0) := x_{0,\ell}$ and uniform
continuity follows.
\end{proof}

\begin{lem}\label{lem:11jp,3}
There exist $\ell_0$ and $L_0,\epsilon_0>0$ such that for
every $L, \epsilon>0$ and $\ell\geq\ell_0$ the
distance from the
set $\mathbf{G}_{\ell}^{-1}\left(W(L,\epsilon)\right)$ to
$[\tau_{\infty},\infty)$ is at least $\epsilon_0$ and the set is
  contained in $\{ \zeta\in\CC :\: \Re\zeta < L_0\}$.

\end{lem}
\begin{proof}
  We begin by observing that
  \[
  \mathbf{G}_{\ell}^{-1}\left(W(L,\epsilon)\right)\subset\tau_{\ell}\Omega_{-,\ell}
  .\]

For
$\ell=\infty$, $\tau_{\infty}\Omega_{-,\infty}$  is compactly
contained in $\CC\setminus [\tau_{\infty},+\infty)$ and hence the
  distance from the claim of this Lemma is positive and is bounded
  from the right.
By uniform convergence from Lemma~\ref{lem:11jp,1} this situation persists
for all $\ell$ sufficiently large.
\end{proof}

\begin{coro}\label{coro:11jp,1}
There exist $\ell_0, L_0<\infty$ and $\epsilon_0>0$ such that for
every $L\geq L_0$, $\epsilon :\: 0<\epsilon\leq\epsilon_0$ and $k\in\ZZ
:\: k>0$ the
family $\left(\mathbf{G}_{\ell}^{-k}\right)_{\ell\geq\ell_0}$ is
equicontinuous and uniformly convergent in $W(L,\epsilon)$.
\end{coro}
\begin{proof}
  This follows from an inductive use of Lemma~\ref{lem:11jp,1}
  once we pick $L_0,\epsilon_0$ as in Lemma~\ref{lem:11jp,3}.
\end{proof}

   \paragraph{Wedge lemma.}
   For every $\ell$ even, positive and finite there is a repelling
   orbit of period $2$
under $G_{\ell}$ which consists of points $x_{+,\ell}\in \HH_+$ and
$x_{-,\ell}\in \HH_-$. When $\ell\rightarrow\infty$ points
$x_{\pm,\ell}$ tend to $x_{0,\infty}$.

The key observation is that there are two inverse branches of
$G^2_{\ell}$, which will be written as $\mathbf{G}^{-2}_{\pm,\ell}$ and
transform $H_{\pm}$ into itself, respectively. Then $x_{\pm,\ell}$ are fixed points
of the corresponding $\mathbf{G}^{-2}_{\pm,\ell}$ which attract
$\HH_{\pm}$.

The lemma below is stated for the upper half-plane without loss of
generality.
\begin{lem}\label{lem:15hp,1}
Suppose that $z_0\in\HH_+$ and $0<\eta\leq \Im z_0$. Furthermore,
assume that $\frac{1}{3}\pi < \arg(z_0-x_{0,\infty}) <
\frac{2}{3}\pi$. For every $\eta>0$ there is $\ell(\eta)<\infty$ such
that whenever $\ell(\eta)\leq \ell <\infty$, then the forward orbit of
$z_0$ under $\mathbf{G}^{-2}_{+,\ell}$ is contained in $D(x_{0,\ell},e\Im
z_0)$.
\end{lem}
\begin{proof}
  We choose $\ell(\eta)$ so that for every $\ell\geq \ell(\eta)$
  the stunted wedge $\{ z\in\HH_+ :\: \Im z\geq\eta ,\, \frac{1}{3}\pi <
  \arg(z-x_{0,\infty}) < \frac{2}{3}\pi \}$ is contained in the wedge
  $\mathfrak{W} := \{ z\in \HH_+ :\: \frac{1}{4}\pi < \arg (z-x_{+,\ell}) <
  \frac{3}{4}\pi \}$.

  Then for any $z_0\in\mathfrak{W}$ the hyperbolic distance in $\HH_+$
  from $z_0$ to $x_{+,\ell}$ is less than $\log\frac{\Im z_0}{\Im
    x_{+,\ell}} + 2$. It is not expanded by the action
  $\mathbf{G}^{-2}_{+,\ell}$. Given a hyperbolic distance the
  maximum Euclidean distance is of obtained when the real parts
  coincide, which yields the estimate of the Lemma.
\end{proof}

A consequence of the wedge lemma is the following estimate.
\begin{lem}\label{lem:11jp,4}
For every $r,H>0$ there exists $K(r,H)$
and $\ell(r,H)$ such that for every $\ell\geq\ell(r,H)$
we get
\[ \phi_{\pm,\ell}^{-1} \left( \{ \zeta\in\CC :\:
|\Im\zeta|\leq H,\,\Re\zeta>K(r,H)\}\right) \subset D(x_{0,\infty},r) .\]
\end{lem}
\begin{proof}
Since for $\ell=\infty$ the mapping $\phi^{-1}_{\pm,\infty}$ is a Fatou
coordinate, it maps every horizontal half-ray
\[ \{ \zeta\in\CC :\: \Im\zeta=H,
\Re\zeta>0\} \] to a curve convergent to $x_{0,\infty}$ and tangent to
the repelling direction. Hence, given $r,H$ for some $K(r,H)$ the set
$\left\{ \zeta\in\CC :\: |\Im\zeta|\leq H,\,\Re\zeta > K(r,H) -2\log\tau^2_{\infty}\right\}$ is
mapped into $\left\{ z\in\CC :\: |z-x_{0,\infty}|<\frac{r}{10},\,
\arg\bigl(-(z-x_{0,\infty})^2\bigr) < \frac{\pi}{10} \right\}$. Now choose an
integer $k(r)$ so that
\begin{equation}\label{equ:11jp,3}
  \left(k(r)-1\right)\log\tau^2_{\infty} \in \left( K(r,H)-2\log\tau^2_{\infty},
  K(r,H)\right) .
\end{equation}
  On the other hand, for any $\ell$ consider
\begin{multline*} W(r,H,\ell) := \phi_{\pm,\ell}^{-1}\bigl( \left\{ \zeta\in\CC :\: |\Im\zeta|\leq H,\,
\left(k(r)-1\right)\log\tau^2_{\ell} \leq \Re\zeta \leq k(r)\log\tau^2_{\ell} \right\}
\bigr) =\\  \mathbf{G}^{-2\left(k(r)+1\right)}_{\ell}\circ\phi^{-1}_{\pm,\ell}\bigl( \left\{ \zeta\in\CC :\:
|\Im\zeta|\leq H,\, -2\log\tau^2_{\ell}\leq \Re\zeta\leq
-\log\tau^2_{\ell}\right\}\bigr) .\end{multline*}
By Corollary~\ref{coro:11jp,1} for all $\ell$ sufficiently large
\[ W(r,H,\ell) \subset \left\{ z\in\CC :\: |z-x_{0,\infty}|<\frac{r}{3},\,
\arg\bigl(-(z-x_{0,\infty})^2\bigr) < \frac{\pi}{3} \right\} .\]
Then by Lemma~\ref{lem:15hp,1} all subsequent images of $W(r,H,\ell)$
by iterates of $\mathbf{G}_{\ell}^{-1}$ are contained in
$D(x_{0,\infty},r)$. But that means the entire set
\[ \phi_{\pm,\ell}^{-1}\bigl( \left\{ \zeta\in\CC :\: |\Im\zeta|\leq H,\,
\left(k(r)-1\right)\log\tau^2_{\ell} \leq \Re\zeta \right\}\bigr) \subset
D(x_{0,\infty},r) . \]
Recalling expression~(\ref{equ:11jp,3}),  we get the claim.
\end{proof}

\begin{lem}\label{lem:11jp,5}
For every $r>0$ there is $K(r)$ and $\ell(r)$ so that for any
$\ell\geq\ell(r)$ and $\zeta :\: |\zeta|>K(r),\, |\Im\zeta| <
\frac{\ell\pi}{2},\, \zeta\notin [0,+\infty)$
  we get $\phi^{-1}_{\pm,\ell}(\zeta) \in D(x_{0,\infty},r)$.
\end{lem}
\begin{proof}
  Let us begin with Lemma~\ref{lem:10jp,2} which implies the claim for
$\zeta :\: |\Im\zeta|<\frac{\ell\pi}{2},
  \dist\left(\zeta,[0,+\infty)\right) > L(r)$.

 Then invoke Lemma~\ref{lem:11jp,4} with $H:=L(r)$ to conclude that
 the claim also holds on the infinite half-strip $\bigl\{ \zeta\in\CC
 :\: |\Im\zeta|\leq L(r),\, \Re\zeta >
 K\left(r,L(r)\right)\bigr\}$. What remains is a bounded set.
 \end{proof}

\paragraph{Proof of Proposition~\ref{prop:11jp,1}.}
The limit at $\infty$ for $\phi^{-1}_{\pm,\infty}$ follows from
Lemma~\ref{lem:11jp,5}. It remains for check uniform convergence.
Fix $r>0$
By Lemma~\ref{lem:11jp,5}, for $|\zeta|>K\left(\frac{r}{2}\right)$ we
get $\phi^{-1}_{\pm\ell}(\zeta) \in D(x_{0,\infty},\frac{r}{2})$ for
all $\ell$ large enough and hence
$\left|\phi^{-1}_{\pm,\ell}(\zeta)-\phi^{-1}_{\pm,\infty}(\zeta)\right|
< r$. The remaining bounded set after shifting by some multiple
of $\log\tau^2_{\infty}$ is compactly contained in $\CC\setminus
[0,+\infty)$. Hence uniform convergence follows form
  Fact~\ref{fa:1hp,1} and Corollary~\ref{coro:11jp,1}.

\paragraph{Diameter of $\mathfrak{w}_{\ell}$.}
Recall the arc $\mathfrak{w}_{\ell}$ which for finite $\ell$ separates
$\Omega_{+,\ell}$ from $\Omega_{-,\ell}$. $\mathfrak{w}_{\ell}\cap\HH_+$ is
invariant under $\mathbf{G}_{+,\ell}^{-2}$.

\begin{lem}\label{lem:17ha,3}
For every $\epsilon>0$ there is $\ell(\epsilon)$ such for any
$\ell\geq \ell(\epsilon)$, even and finite, and any
$z\in \mathfrak{w}_{\ell}\cap\HH_+$
the hyperbolic diameter in $\HH_+$ of the subarc of $\mathfrak{w}$
between $z$ and $G^2_{\ell}(z)$ is bounded by $\epsilon$.
\end{lem}
\begin{proof}
Let $\mathfrak{w}(z)$ denote the segment of $\mathfrak{w}_{\ell}$ between $z$ and
$G^2_{\ell}(z)$. Then its hyperbolic diameter is bounded by the
hyperbolic diameter of $\mathfrak{w}\left(G^{2n}_{\ell}(z)\right)$ for
any $n$ positive by Schwarz' Lemma. On the other hand,
$\mathfrak{w}_{\ell}$ is a preimage of a line by an analytic mapping,
hence a smooth curve at $x_{0,\ell}$ tangent to the vertical line
$x_{0,\ell}+\iota\RR$. It develops that the limit of the hyperbolic
diameter of $\mathfrak{w}\left(G^{2n}(x)\right)$ as
$n\rightarrow\infty$ is $-2\log G'_{\ell}(x_{0,\ell})$ which tends to
$0$ as $\ell\rightarrow\infty$.
\end{proof}

\begin{lem}\label{lem:17ha,2}
  \[ \lim_{\ell\rightarrow\infty} \diam (\mathfrak{w}_{\ell}) = 0 \; .\]
\end{lem}
\begin{proof}
  It is enough to prove the claim for $\mathfrak{w}_{\ell} \cap
  \HH_+$. Fix $r>0$ and suppose that for $\ell$ arbitrary large
  $\mathfrak{w}_{\ell} \cap \HH_+$ intersects $C(x_{0,\infty},r)$ at $z_0$.
  Then by Lemma~\ref{lem:17hp,1} for $\ell\geq \ell(r)$ the Euclidean
  diameter of the subarc of $\mathfrak{w}_{\ell}$ between $z_0$ and
  $G^2_{\ell}(z_0)$ is at least $\varepsilon(r)$. But by
  Lemma~\ref{lem:17ha,3} the hyperbolic diameter of the same arc tends
  to $0$ as $\ell\rightarrow\infty$ which yields a contradiction.
\end{proof}

\section{Dynamics near an almost parabolic point.}
 \subsection{Elementary estimates.}

\paragraph{Double wedge in $\Omega_{\ell}$.}
Start with the following fact:
\begin{fact}\label{fa:18ha,1}
  For any $\delta>0$ there is $r(\delta)>0$ such that the double wedge
  \[  \left\{ x_{0,\infty}+\zeta :\: |\zeta|<r(\delta),\, |\arg\zeta^2| <
  \pi-\delta \right\} \]
is contained in $\Omega_{\infty}$.
\end{fact}

We will now work to obtain a similar estimate for finite $\ell$,
uniform in $\ell$.

\begin{defi}\label{defi:24ha,1}
  For $\delta>0$ and $0<r<R$ and $s\in\{+,-,0\}$ we will write
  \[W_s(\delta,r,R) := \left\{ x_{s,\infty}+\zeta :\:
r<|\zeta|<R,\, |\arg\zeta^2| < \pi-\delta \right\} \; .\]
\end{defi}

\begin{lem}\label{lem:18ha,1}
For every $\delta>0$ and $s\in\{+,-,0\}$ there is $r(\delta)>0$ and, additionally, for
every $r_1>0$   there is $\ell(\delta,r_1)<\infty$ such that
\[ \forall\ell\geq\ell(\delta,r_1)\; W_s(\delta,r_1,r(\delta)) \subset
\Omega_{\ell} \; ,\]
cf. Definition~\ref{defi:24ha,1}.
\end{lem}
\begin{proof}
By Fact~\ref{fa:18ha,1} for every $r_1>0$ and $r(\delta)$ taken from
that Fact the set $W_0(\delta,r_1,r(\delta))$ is compactly contained in
$\Omega_{\infty}$. Moreover, for some $\epsilon(r_1,\delta)>0$,
\[ \bigcup_{z\in W_0(\delta,r_1,r(\delta))}
\overline{D\left(z,\epsilon(r_1,\delta)\right)} \]
remains compactly contained in $\Omega_{\infty}$.
By Fact~\ref{fa:1hp,2}, for $\ell$ large enough, the mappings
$\phi_{\pm,\ell}^{-1} \circ \phi_{\pm,\infty}$ send
$C\left(z,\epsilon(r_1,\delta)\right)$ to a Jordan curve which
surrounds $z$. By the argument principle, $z$ also has a preimage
$\phi_{\pm,\ell}^{-1}(z)$.
The claim for $s=\pm$ follows since $\lim_{\ell\rightarrow\infty}
|x_{0,\ell}-x_{s,\ell}| = 0$.
\end{proof}

\subsection{Main theorem.}
\begin{defi}\label{defi:27ha,1}
For an analytic function $g$, $z$ which can be forever iterated by $g$
and $\sigma>0$ define
\[ P(g,z,\sigma) := \sum_{k=0}^{\infty} |Dg^k(z)|^{\sigma} \; .\]
\end{defi}

We now state a general theorem whose hypotheses are satisfied by
functions $G_{\ell}$ we considered so far. In particular, the geometric
condition of $\Omega_{\ell}$ follows from Lemma~\ref{lem:18ha,1}.

Recall that a mapping $g$ symmetric about $\RR$ and
defined in $\CC$ doubly slit along the
real axis is in the {\em Epstein class} if its derivative does not
vanish in $\RR$ and has an inverse branch defined on $\HH_+$ which
maps into $\HH_+$ or $\HH_{-}$.

\begin{theo}\label{theo:27ha,1}
  Suppose that $(G_{\ell})$ is a sequence of mappings which are all defined
  of $\CC\setminus\left( (-\infty, X_1] \cup [ X_2,+\infty)\right)$, $X_1<X_2$,
  which are holomorphic, symmetric about $\RR$ and in the Epstein
  class.
  Next, for some sequence
  $(x_{0,\ell})$ of points contained in $(X_1,X_2)$ and convergent to
  $x_{0,\infty}\in (X_1,X_2)$, there is
  a representation
  \[ G_{\ell}(z+x_{0,\ell})-x_{0,\ell} = \sum_{k=1}^{\infty} \alpha_{k,\ell} z^k \]
  where $\forall \ell\; \alpha_{1,\ell}\in (-1,0)$ and
  $\lim_{\ell\rightarrow\infty} \alpha_{1,\ell} = -1$. Suppose finally
    that $G_{\ell}$ converge almost uniformly in their domain.

  For every $\ell$, let $\Omega_{\ell}$ be a domain which is fully
  invariant under $G_{\ell}$ and assume further
  \begin{multline*} \exists \delta_0 > \frac{\pi}{2}\; \exists R_0>0\; \forall r>0\;
  \exists \ell_0(r)<\infty \; \forall \ell\geq\ell_0(r)\; \\
  \left\{ x_{0,\infty}+z :\: r<|z|<R_0,\, |\arg -z^2| < \delta_0\right\}
\subset \Omega_{\ell} \; . \end{multline*}

Then, for some $R_1>0$, every $\ell$ and $\sigma>\frac{4}{3}$
\[ \int_{D(x_{0,\infty},R_1)\setminus\Omega_{\ell}}
P(\mathbf{G}_{\ell}^{-2},x+\iota y,\sigma),\, dx\, dy\]
are uniformly bounded for all $\ell$, where $\mathbf{G}^{-2}_{\ell}$
is the inverse branch of $G^2_{\ell}$ which fixes $x_{0,\ell}$.
\end{theo}

From these hypotheses for every $\ell$ we get a repelling periodic
orbit of period $2$, $\{x_{\pm,\ell}\}$ under $G_{\ell}$.

The next lemma is stated $\HH_+$ without loss of generality, since by
symmetry the analogous statement holds in the lower half-plane.
\begin{lem}\label{lem:24ha,1}
For some $\delta<\frac{\pi}{4}$ there is $r_0>0$ such that for every $r
:\: 0<r<r_0$ there exists
$\ell(r)<\infty$ so that the following claim holds.

If $u\in \HH_+ \cap D\left(x_{+,\ell},\frac{r}{2}\right)$ and
$u\notin\Omega_{\ell}$, then for some positive $n$ and all $\ell\geq\ell(r)$
\[ G^{2n}_{\ell}(u) \in \left\{ x_{+,\ell}+\iota z\in \HH_+ :\:
\frac{r}{2} < |z| < r,\, |\arg z| <\delta\right\} .\]
\end{lem}
\begin{proof}
  Initially choose $\ell(r)$ so large that $|x_{0,\ell}-x_{+,\ell}|<\frac{r}{2}$ for all $\ell\geq\ell(r)$. Additionally,
  when $r$ is small enough and $\ell$ large, then
  $G_{\ell}^2\left(D(x_{+,\ell},\frac{r}{2}) \cap \HH_+\right) \subset
  \HH_+$.

  Consider the orbit of $u$ under $G_{\ell}^2$. First we show that for
  some $n$ it must leave $D(x_{+,\ell},\frac{2r}{3})$. Suppose
  not. Since $G_{\ell}^2$ expands the hyperbolic metric of $\HH_+$,
  the orbit must eventually leave every compact neighborhood of
  $x_{+,\ell}$. It follows that $\lim_{n\rightarrow\infty}
  \Im G^{2n}_{\ell}(u) = 0$. By choosing $r$ small, we can make sure
  that $[x_{0,\ell}-r,x_{0,\ell}+r]\subset\Omega_{\ell}$ for all
  $\ell$. Thus, $G^{2n}_{\ell}(u) \in\Omega_{\ell}$ which contradicts
  the hypothesis of Theorem~\ref{theo:27ha,1} by which
  $\Omega_{\ell} \cap \HH_+$ is
  completely invariant under $G^2_{\ell}$.

  Now we see that for some $n\geq 0$ we have
  \begin{align*}
  \left|G^{2n}_{\ell}(u)-x_{+,\ell}\right| \leq & \frac{r}{2},\, \text{but}\\
  \left|G^{2(n+1)}_{\ell}(u)-x_{+,\ell}\right| > & \frac{r}{2} .
  \end{align*}
  Since
  $G_{\ell} \rightarrow G_{\infty}$ uniformly on compact neighborhoods
  of $x_{0,\infty}$ and the derivative is $1$ at that point, by
  choosing $r$ small and $\ell$ large, we can have
  \[\left|
  \frac{G^{2n}_{\ell}(u)-x_{+,\ell}}{G^{2(n+1)}_{\ell}(u)-x_{+,\ell}}
  \right| > \frac{1}{2} .\]
  Hence
  \[ G^{2(n+1)}(u) \in  \left\{ z\in \HH_+ :\: \frac{r}{2} < |z-x_{+,\ell}|
    < r\right\} \setminus \Omega_{\ell} \; .\]
  The condition on the argument follows from the geometric hypothesis
  of Theorem~\ref{theo:27ha,1} . The possibility of $\arg z$ being close to
  $\pi$ can be ruled out when $\ell$ is made sufficiently large so
  that $|x_{0,\ell}-x_{+,\ell}|$ becomes small compared to $r$.
\end{proof}

\subsection{Generalized Fatou coordinate.}
Let us write
\[ \mathbf{G}^{-2}_{\ell}(x_{0,\ell}+z) - x_{0,\ell} = \sum_{k=1}^{\infty}
a_{k,\ell} z^k .\]
For $a_{2,\ell}$ the condition of dominant convergence is satisfied and
so it can be removed by a change of coordinate which for all $\ell$
belongs to a compact family of diffeomorphisms of a fixed neighborhood
of $x_{0,\infty}$, see~\cite{profesorus1}, proof of Theorem 7.2.
With a slight abuse of notations we internalize this change of coordinate
simply assuming $a_{2,\ell}=0$.

Next, we write $a_{1,\ell}=1+\frac{\rho_{\ell}}{4}$ where
$\rho_{\ell}>0$. We also know that  $\lim_{\ell\rightarrow\infty} a_{3,\ell} = a_{3,\infty} > 0$.

Now
$x_{+,\ell}=x_{0,\ell}+\iota\sqrt{\frac{\rho_{\ell}}{4a_{3,\ell}}}\mathfrak{E}(\ell)$

where we shall write $\mathfrak{E}(\ell) := \exp\bigl(O\left(\sqrt{\rho_{\ell}}\right)\bigr)$.

Consider the development of $\mathbf{G}^{-2}_{\ell}$ at $x_{+,\ell}$:
\begin{equation*}
\mathbf{\Gamma}_{\ell}(z) := \iota^{-1}\left(\mathbf{G}^{-2}_{\ell}(x_{+,\ell}+\iota
z)-x_{+,\ell}\right) = \sum_{k=1}^{\infty} \hat{a}_{k,\ell} z^k
\end{equation*}
where
\begin{equation}\label{equ:25hp,2}
  \begin{split}
  \hat{a}_{1,\ell} = &  \left(1-\frac{\rho_{\ell}}{2}\mathfrak{E}(\ell)\right)\\
  \hat{a}_{2,\ell} = & -
  \frac{3}{2}\sqrt{a_{3,\ell}\rho_{\ell}}\mathfrak{E}(\ell)\\
  \hat{a}_{3,\ell} = & -a_{3,\ell}\mathfrak{E}(\ell) .
  \end{split}
\end{equation}

Observe two features of $\mathbf{\Gamma}_{\ell}$ which make its
analysis non-standard. First, the quadratic term cannot be removed or
neglected as $\ell\rightarrow\infty$, i.e. no dominant convergence in
the sense of~\cite{profesorus1}.
\paragraph{Definition of the generalized Fatou coordinate.}
\begin{defi}\label{defi:30hp,1}
Define $\zeta_{\ell} :\: \CC \rightarrow \hat{\CC}\setminus\{0\}$
by
\[ \zeta_{\ell}(z) = \frac{1}{2A_{\ell} z^2} \; \]
where $A_{\ell} = -\frac{\hat{a}_{3,\ell}}{\hat{a}_{1,\ell}} + 3\frac{\hat{a}^2_{2,\ell}}{\hat{a}^2_{1,\ell}} =
a_{3,\ell}\mathfrak{E}(\ell)$.
\end{defi}

Let us introduce a variable $\zeta$ on the Riemann surface of the
function $\zeta_{\ell}$ which means that $\zeta_{\ell}^{-1}$ is well
defined as well as
$\bm{\gamma}_{\ell}(\xi):=\zeta_{\ell}\circ\mathbf{\Gamma}_{\ell}\left(\frac{1}{\xi\sqrt{2A_{\ell}}}\right)$,
where the principal branch of the root is applicable since $\arg
A_{\ell}$ is close to $0$. Then
\[ \zeta_{\ell}\circ\mathbf{\Gamma}_{\ell} \circ
\zeta_{\ell}^{-1}(\zeta) =
\left(\bm{\gamma}_{\ell}\circ\sqrt{}\right)(\zeta) .\]

We get the representation:
\begin{equation}\label{equ:25hp,1}
  \zeta_{\ell}\left(\mathbf{\Gamma}_{\ell}(z)\right) =
\gamma_{\ell}\left(\sqrt{\zeta}\right) := \hat{a}^{-2}_{1,\ell}\zeta+1
  +\sqrt{\frac{9}{2}\rho_{\ell}}\mathfrak{E}(\ell)\sqrt{\zeta} + O\left(|\zeta|^{-1/2}\right) \; .
\end{equation}

Observe that $\sqrt{\zeta}$ is correctly defined
by substituting
\begin{equation}\label{equ:25ka,1}
  \sqrt{\zeta} = \frac{1}{\zeta_{\ell}^{-1}(\zeta) \sqrt{2A_{\ell}}} .
  \end{equation}
In particular, for $\Re z > 0$ one should choose the principal branch
of $\sqrt{\zeta}$. When the principal branch of $\sqrt{\zeta}$ is
used, we will talk of the principal branch $\bm{\gamma}_{\ell}$.

\begin{lem}\label{lem:25hp,1}
There exist constants $\ell_0<\infty$, $R_0,K_1,K_2,K_3$ such that for any
$z :\: |z| \geq R_0$ and $\ell\geq\ell_0$
\[
\left|\bm{\gamma}_{\ell}(\sqrt{\zeta})-\zeta-\rho_{\ell}\zeta-\sqrt{\frac{9}{2}\rho_{\ell}\zeta}-1\right|
\leq K_1 \rho_{\ell}^{3/2} |\zeta| + K_2 \rho_{\ell} \sqrt{|\zeta|} + K_3
|\zeta|^{-1/2} \; .\]
\end{lem}
\begin{proof}
From formula~(\ref{equ:25hp,1}) the linear term in
$\bm{\gamma}_{\ell}(\sqrt{\zeta})-\zeta$ is $(\hat{a}_{1,\ell}^{-2}-1)\zeta=\rho_{\ell}\zeta +
O(\rho_{\ell}^{3/2})\zeta$ which gives rise to the term of order
$|\zeta|$ in the claim of the Lemma.

The root term in formula~(\ref{equ:25hp,1}) is
\[ \sqrt{\frac{9}{2}\rho_{\ell}}\mathfrak{E}(\ell)\sqrt{\zeta} =
 \sqrt{\frac{9}{2}\rho_{\ell}}\sqrt{\zeta} +
 O\left(\rho_{\ell}\sqrt{|\zeta|}\right) \]
 and the $O\left(|\zeta|^{-1/2}\right)$ term is directly copied.
\end{proof}

\subsection{Dynamics in the $\zeta$ coordinate.}
Although the goal of Theorem~\ref{theo:27ha,1} is an estimate uniform
in $\ell$, the description of the dynamics will be split into cases
depending on $\ell$: the mid-range case of $\zeta =
O\left(\rho^{-1}_{\ell}\right)$ which generally reminiscent of a
parabolic point and the far-range for larger $\zeta$ where the
true nature of the fixed point at $x_{+,\ell}$ becomes evident.

\begin{lem}\label{lem:27hp,1}
For any $\delta :\: 0<\delta<\frac{\pi}{2}$ and $Q\geq 1$ there are
$r(\delta)$ and $\ell_0(\delta,Q)$
such that for every $\ell\geq\ell_0(\delta,Q)$
if $\zeta :\: r(\delta)<\Re\zeta<Q\rho^{-1}_{\ell},\, |\arg\zeta|<\delta$, then
$\Re \bm{\gamma}_{\ell}(\sqrt{\zeta}) > \Re \zeta + \frac{1}{2}$
and $|\arg \bm{\gamma}_{\ell}(\sqrt{\zeta})| < \delta$.
\end{lem}
\begin{proof}
  According to Lemma~\ref{lem:25hp,1}
  \[ \bm{\gamma}_{\ell}(\sqrt{\zeta})-\zeta= \rho_{\ell}\zeta +
  \sqrt{\frac{9\rho_{\ell}}{2}\zeta} + 1 + \text{corrections}
  \; .\]

  Both the linear and root terms are helping the estimate of the
  Lemma, by increasing the real part of the expression and bringing
  its argument closer to $0$. So we ignore them. What is left is $1$
  and the corrections. We make each of them less than
  $\frac{\delta}{30}$. For the $K_3$ term this requires making $\zeta$
  sufficiently large depending on $\delta$. The next term is bounded
  by
  $K_2\sqrt{Q}\rho_{\ell}^{1/2}$ and requires $\ell$ large enough depending
  on $\delta,Q$ and the first term is estimated similarly.

  Thus,
  \[  \bm{\gamma}_{\ell}(\sqrt{\zeta})-\zeta= 1 + E(\delta) +
  \text{terms of $\Re>0$ and $|\arg| < \delta$}  \]
  with $\left|E(\delta)\right| < \frac{\delta}{10}$. The claim of the Lemma follows.
  \end{proof}

\begin{lem}\label{lem:27hp,2}
  With the same notations as in the previous lemma, for every $\delta>0$
  there are $r(\delta)>0, L(Q)<\infty$ such that if
  $r(\delta) < \Re\zeta, |\arg\zeta|<\delta$ and for every $j=0,\cdots, k$,
  $\Re \left(\bm{\gamma}_{\ell}\circ\sqrt{}\right)^j(\zeta) < Q\rho^{-1}_{\ell}$, then
  \[ \forall \ell\geq\ell_0(\delta,Q)\; \Bigl| D_{\zeta}\left(\bm{\gamma}_{\ell}\circ\sqrt{}\right)^k(\zeta) \Bigr| < L(Q) \; .\]
\end{lem}
\begin{proof}
  We choose $r(\delta)$ at least as large as in Lemma~\ref{lem:27hp,1}
  and as a consequence of Lemma~\ref{lem:25hp,1} and Cauchy's estimates we get
  \[ \bigr|\log\left(D_{\zeta}\bm{\gamma}_{\ell}(\sqrt{\zeta})\right)\bigr| \leq K_1\rho_{\ell} + K_2\rho_{\ell}|\zeta|^{-1/2} + K_3|\zeta|^{-3/2} \]
  for $|\zeta|, \ell$ greater than some constants.

  By Lemma~\ref{lem:25hp,1}, $\|\bm{\gamma}_{\ell}^j(\sqrt{\zeta})\|\geq r(\delta)+\frac{j}{2}$. If $r(\delta)>1$, this leads to the following estimate:
   \[ \bigl|\log\left(D^k_{\zeta}\bm{\gamma}_{\ell}(\sqrt{\zeta})\right)\bigr| \leq K_1k \rho_{\ell} + K_2\rho_{\ell}\sqrt{k} + 2K_3\sum_{j=1}^{\infty} j^{-3/2} \; . \]

   At the same time, since $\Re\bm{\gamma}_{\ell}^k(\zeta)<Q\rho^{-1}_{\ell}$, $k<2Q\rho^{-1}_{\ell}$ which yields the claim of the Lemma.
\end{proof}
\paragraph{Far-range dynamics.}
Here we $|\zeta|\geq Q\rho_{\ell}^{-1}$.

\begin{lem}\label{lem:27hp,3}
  For every $\eta>0$ there is $Q(\eta) :\: 1<Q(\eta)<\infty$ and $\ell_0(\eta)$ such that
  for every $\ell\geq \ell_0(\eta)$ and $\zeta :\: |\zeta| \geq Q(\eta)\rho_{\ell}^{-1}$
  \begin{itemize}
    \item
      \[ \left|\bm{\gamma}_{\ell}(\sqrt{\zeta})\right| \geq \left| \zeta\right|(1+\rho_{\ell})^{1-\eta} \; ,\]
      \item
      \[ \left|D_{\zeta}\bm{\gamma}_{\ell}(\sqrt{\zeta})\right| \leq (1+\rho_{\ell})^{1+\eta} \; ,\]
   \end{itemize}
\end{lem}
\begin{proof}
  From Lemma~\ref{lem:25hp,1} we conclude that for $|\zeta| \geq
  Q\rho_{\ell}^{-1}$, $\ell\geq\ell_0$,
  \[
  \left|\frac{\bm{\gamma}_{\ell}\left(\sqrt{\zeta}\right)}{\zeta+\rho_{\ell}\zeta}\right|
  \geq 1-Q^{-1}\rho_{\ell} - \rho_{\ell}\sqrt{\frac{9}{2Q}} -
  KQ^{-1/2} \rho^{3/2}_{\ell} = 1 - K(\ell,Q)\rho_{\ell}\]
  where $\forall \ell_0\; \lim_{Q\rightarrow\infty} \sup\{ K(\ell,Q)
  :\: \ell\geq\ell_0\} = 0$.
  For $\rho_{\ell}$ small enough and $K(\ell,Q)\leq 1$ this leads to

  \[
  \left|\frac{\bm{\gamma}_{\ell}\left(\sqrt{\zeta}\right)}{\zeta+\rho_{\ell}\zeta}\right|
  \geq \left(1+\rho_{\ell}\right)^{-2K(\ell,Q)} \]
  and it suffices to choose $Q(\eta)$ so that $\sup\left\{
  K(\ell,Q(\eta)) :\: \ell\geq\ell_0\right\} <
  \frac{\eta}{2}$
in order to obtain the first claim.

For the second claim, we similarly get from Lemma~\ref{lem:25hp,1} that
  \[ \left| D_{\zeta}\bm{\gamma}_{\ell}(\sqrt{\zeta})\right| \leq 1+ \rho_{\ell} + \rho_{\ell} \sqrt{\frac{9}{2Q}} + K_1\rho_{\ell}^{3/2} \]
  for $\ell$ and $Q$ suitably bounded below. Similar to the previous
  case the right side can be bounded above by
  $\left(1+\rho_{\ell}\right)^{1+2K'(\ell,Q)}$ and the second claim follows.
\end{proof}

\paragraph{Joint estimates.}
We will now write general estimates on the absolute value and
derivative along the orbits in the $\zeta$ coordinate, i.e. compositions of functions
$\bm{\gamma}_{\ell}\circ\sqrt{}$.

\begin{lem}\label{lem:30hp,1}
 For every $\delta :\: 0<\delta<\frac{\pi}{2}$ there is
 $r_0(\delta)>0$ and for every $\eta>0$ there are
 $\ell_0(\delta,\eta), L(\eta), Q(\eta)>1$ such that
 \begin{multline*}
   \forall \ell\geq \ell_0(\delta,\eta)\; \forall \zeta\in\CC :\:
 |\zeta|>r_0,\,|\arg\zeta|<\delta \;  \\ \exists k(\zeta,\ell)\;
 \Re\left(\bm{\gamma}_{\ell}\circ\sqrt{}\right)^{k(\zeta,\ell)}(\zeta)\geq Q(\eta)\rho_{\ell}^{-1} :\:
 \end{multline*}
 \begin{itemize}
\item $\forall 0\leq k\leq k(\zeta,\ell)\;
\Re\left(\bm{\gamma}_{\ell}\circ\sqrt{}\right)^k(\zeta)\geq \Re\zeta+\frac{k}{2}$,
\item \begin{multline*}
  \forall k\geq 0\;    \Bigl|\left(\bm{\gamma}_{\ell}\circ\sqrt{}\right)^k(\zeta)\Bigr|
    \geq \\
     \Bigl(\Re\left(\bm{\gamma}_{\ell}\circ\sqrt{}\right)^{\min\left(k,k(\zeta,\ell)\right)}(\zeta)\Bigr)(1+\rho_{\ell})^{\max\left(k-k(\zeta,\ell\right),0)(1-\eta)} ,\end{multline*}
\item $\forall k\geq 0\;\bigl|D_{\zeta}\left(\bm{\gamma}_{\ell}\circ\sqrt{}\right)^k(\zeta)\bigr| \leq
 L(\eta) (1+\rho_{\ell})^{\max\left(k-k(\zeta,\ell),0\right)(1+\eta)}$ .
\end{itemize}
 \end{lem}
\begin{proof}
 By Lemma~\ref{lem:27hp,1} when $\zeta$ is chosen in the specified
 set, it will move inside the same set by at least $\frac{1}{2}$ to
 the right by each iterate of $\left(\bm{\gamma}_{\ell}\circ\sqrt{}\right)$, which then must be the
 principal branch. $Q(\eta)$ is chosen by Lemma~\ref{lem:27hp,3}. The
 key point is the choice of $k(\zeta,\ell)$ which the smallest $k$ for
 which $\Re\left(\bm{\gamma}_{\ell}\circ\sqrt{}\right)^k(\zeta) \geq Q(\eta)\rho_{\ell}^{-1}$. Until
 that point the dynamics is controlled by Lemma~\ref{lem:27hp,1} and
 the estimate of Lemma~\ref{lem:27hp,2} on the derivative,
 while
 afterwards the dynamics becomes complicated, but simple estimates of
 Lemma~\ref{lem:27hp,3} hold.
\end{proof}

Now we draw conclusions for iterates of $\mathbf{G}_{\ell}^{-2}$.
\begin{lem}\label{lem:30hp,2}
 For every $\delta :\: 0<\delta<\frac{\pi}{4}$ there is
 $r_0(\delta)>0$ and for every $\eta>0$ and $r :\: 0<r<r_0$ there are
 $\ell_0(\delta,\eta), L(\eta,r)$ such that
 \begin{equation*}
   \forall \ell\geq \ell_0(\delta,\eta)\; \forall z \in \bigl\{ z\in\CC :\:
 r<|z-x_{+,\ell}|<r_0,\,\left|\arg
 \iota^{-1}(z-x_{+,\ell})\right|<\delta \bigr\}\;  \exists
 k(z,\ell)
 \end{equation*}
 \begin{equation*}
   \begin{split}
 \forall 0\leq k \leq k(z,\ell)\; &
 \left| D_z\mathbf{G}_{\ell}^{-2k}(z)\right|  \leq
 L(\eta,r) \bigl(1 + \frac{k}{2}\bigr)^{-3/2}\\
 \forall k\geq k(z,\ell)\;
 &\left| D_z\mathbf{G}_{\ell}^{-2k}(z)\right|  \leq
 L(\eta,r)\rho_{\ell}^{3/2}\left(1+\rho_{\ell}\right)^{(-\frac{1}{2}+3\eta)\max\left(k-k(z,\ell),0\right)}
 .
 \end{split}
 \end{equation*}
\end{lem}
\begin{proof}
This is a consequence of Lemma~\ref{lem:30hp,1} and the change of
coordinate $\zeta := \zeta_{\ell}(z)$,
cf. Definition~\ref{defi:30hp,1}. The derivative of that change of
coordinate is bounded above in terms of $r$ and the derivative of the
inverse change is bounded above by a constant times $|\zeta|^{-3/2}$.
The bound on the argument $\delta$ doubles by the generalized Fatou
coordinate, hence different values in the hypotheses of the Lemmas.
Now the claim follows directly from Lemma~\ref{lem:30hp,1}, except
that by taking $r_0(\delta)$ small enough we guarantee $\Re\zeta\geq
1$.
\end{proof}

\subsection{Estimates of the Poincar\'{e} series.}
Define a domain
$W(r,\delta) : = \{ x_{+,\ell}+\iota z \in\CC\setminus\Omega_{\ell}:\: \frac{r}{2} < |z| < r,\, |\arg z|
< \delta \}$ where $\delta<\frac{\pi}{4}$.

For $z\in W(r,\delta_0)$, where $\delta_0$ comes from the hypothesis
of Theorem~\ref{theo:27ha,1}, define
\begin{equation}\label{equ:30hp,2}
  \hat{P}(z,\sigma) := \sum_{k=1}^{\infty}\sum_{j=1}^{k}
|D\mathbf{G}_{\ell}^{-2j}(z)|^{2-\sigma}
|D\mathbf{G}_{\ell}^{-2k}(z)|^{\sigma} \; .
\end{equation}

\begin{lem}\label{lem:28hp,1}
  For some $r_0>0$ and $0<\delta<\frac{\pi}{4}$, any $0<r<r_0$
  and every $\ell\geq \ell_0(r)$,
  \[ \int_{D(x_{0,\infty},\frac{r}{4})\setminus\Omega_{\ell}}
  P(\mathbf{G}_{\ell}^{-2},x+\iota y,\sigma)\,dx\,dy \leq 2 \int_{W(r,\delta)}
  \hat{P}(x+\iota y,\sigma)\,dx\,dy \; .\]
\end{lem}
\begin{proof}
Constants $r_0$ and $\delta$ are chosen from Lemma~\ref{lem:24ha,1}
which asserts that $W(r,\delta)$ is a fundamental domain such that
every orbit which starts in $\HH_+ \cap
D\left(x_{+,\ell},\frac{r}{2}\right)\setminus\Omega_{\ell}$ passes through it under the
forward iteration by $G^2_{\ell}$. For $\ell$ sufficiently large
depending on $r$, $D\left(x_{0,\infty},\frac{r}{4}\right) \subset
D\left(x_{+,\ell},\frac{r}{2}\right)$. Taking into account the
symmetry about $\RR$, the claim of the Lemma is reduced to
  \begin{equation}\label{equ:30hp,1} \int_{D(x_{+,\ell},\frac{r}{2}) \cap \HH_+\setminus\Omega_{\ell}}
  P(\mathbf{G}_{\ell}^{-2},x+\iota y,\sigma)\,dx\,dy \leq  \int_{W(r,\delta)}
  \hat{P}(x+\iota y,\sigma)\,dx\,dy \; .
  \end{equation}
By the fundamental domain property
 \begin{multline*} \int_{D(x_{+,\ell},\frac{r}{2}) \cap \HH_+\setminus\Omega_{\ell}}
   P(\mathbf{G}_{\ell}^{-2},x+\iota y,\sigma)\,dx\,dy \leq \\
   \int_{W(r,\delta)} \sum_{j=1}^{\infty}
  P\left(\mathbf{G}_{\ell}^{-2},\mathbf{G}_{\ell}^{-2j}(x+\iota y),\sigma\right)
  \left|D_z\mathbf{G}_{\ell}^{-2j}(x+\iota y)\right|^2\,dx\,dy \; .
 \end{multline*}
Representing the Poincar\'{e} series from the definition, we evaluate
the sum under the second integral:
\begin{multline*} \sum_{j=1}^{\infty}
  P\left(\mathbf{G}_{\ell}^{-2},\mathbf{G}_{\ell}^{-2j}(x+\iota y),\sigma\right)
  \left|D_z\mathbf{G}_{\ell}^{-2j}(x+\iota y)\right|^2 =\\
  \sum_{j=1}^{\infty} \sum_{p=0}^{\infty} \bigl|
  D_z\mathbf{G}_{\ell}^{-2p}\left( \mathbf{G}_{\ell}^{-2j}(x+\iota
  y)\right)\bigr|^{\sigma} \left|D_z\mathbf{G}_{\ell}^{-2j}(x+\iota
  y)\right|^2 = \\ \sum_{j=1}^{\infty} \sum_{k=j}^{\infty} \bigl|
  D_z\mathbf{G}_{\ell}^{-2k}\left(x+\iota y\right)\bigr|^{\sigma} \left|D_z\mathbf{G}_{\ell}^{-2j}(x+\iota
  y)\right|^{2-\sigma}
  \end{multline*}
  with $k:=j+p$ and estimate~(\ref{equ:30hp,1}) follows by
  interchanging the order of summation.
  \end{proof}

\paragraph{Proof of Theorem~\ref{theo:27ha,1}.}
The proof will follow from Lemmas~\ref{lem:30hp,2}
and~\ref{lem:28hp,1}. We begin by setting the parameters, starting with
$\delta$ of Lemma~\ref{lem:28hp,1}. Given that,
we choose $2r$ in Lemma~\ref{lem:30hp,2} then $r_0(\delta)$ as well as
$r_0$ of Lemma~\ref{lem:28hp,1}.

Now $\eta$ is fixed so that $3\eta<\frac{1}{2}$, thus
$\eta:=\frac{1}{8}$ will do. Then all the bounds
$\ell_0(\delta,\eta),L(\eta,r)$  of
Lemma~\ref{lem:30hp,2} become constants and will be written simply as
$\ell_0,Q,L$. Only the dependence of $\ell$ through $k(z,\ell)$ and
$\rho_{\ell}$ remains.

By lemma~\ref{lem:30hp,2} for all $z\in W(r,\delta), \ell\geq\ell_0$ and $k\geq 0$,
$D_z \mathbf{G}_{\ell}^{-2k}(z)$ are uniformly bounded above. Then, by
inspecting the formula of Definition~\ref{defi:27ha,1} for
$g:=\mathbf{G}_{\ell}^{-2}$ we see that increasing $\sigma$ increases the sum
of the Poincar\'{e} series at most by a uniform constant for any $z\in
W(r,\delta)$. Hence, without loss of generality we can restrict our
considerations to $\frac{4}{3} < \sigma < 2$.

Then
\[ \hat{P}(z,\sigma) \leq K \sum_{k=1}^{\infty} k\left|
D_z\mathbf{G}^{-2k}_{\ell}(z) \right|^{\sigma} \; .\]

First we estimate the sum for $k\leq k(z,\ell)$:
\[ \sum_{k=1}^{k(z,\ell)} k\left| D_z\mathbf{G}^{-2k}_{\ell}(z)
\right|^{\sigma} \leq L^{\sigma} \sum_{k=1}^{\infty}
k\left(1+\frac{k}{2}\right)^{-\frac{3\sigma}{2}} \leq K \sum_{k=1}^{\infty}
k^{1-\frac{3}{2}\sigma} \leq K(\sigma) \]
for $\sigma>\frac{4}{3}$.

Now we deal with
\[ \sum_{k>k(z,\ell)} k\left| D_z\mathbf{G}^{-2k}_{\ell}(z)
\right|^{\sigma} \leq
L^{\sigma}\rho_{\ell}^{3\sigma/2}\sum_{k=0}^{\infty} k\left(1+\rho_{\ell}\right)^{-k\sigma/8}\]
using the estimate of Lemma~\ref{lem:30hp,2} with
$\eta=\frac{1}{8}$. Since $2>\sigma>\frac{4}{3}$,
$\rho_{\ell}^{3\sigma/2}\leq \rho_{\ell}^{2+\sigma'}$ where
$\sigma':=\frac{3}{2}\sigma-2>0$, while $L^{\sigma}$ is
just another constant $L'$. For $\rho_{\ell}$ sufficiently small
\[ \left(1+\rho_{\ell}\right)^{-\sigma/8} \leq 1 -
\frac{\sigma\rho_{\ell}}{9} .\]

Hence, for all $\ell$ sufficiently large,
\begin{multline*} \sum_{k>k(z,\ell)} k\left| D_z\mathbf{G}^{-2k}_{\ell}(z)
\right|^{\sigma} \leq L'\rho_{\ell}^{2+\sigma'} \sum_{k=0}^{\infty}
k\left(1-\frac{\sigma\rho_{\ell}}{9}\right )^k = \\L'\rho_{\ell}^{2+\sigma'}
\left(1-\frac{\sigma\rho_{\ell}}{9}\right)\left(\frac{9}{\sigma\rho_{\ell}}
\right)^2 \leq \frac{81 L' \rho_{\ell}^{\sigma'}}{\sigma^2}
\end{multline*}
which tends to $0$ as $\ell\rightarrow\infty$.

So, $\hat{P}(z,\sigma)$ is uniformly bounded for all $z\in
W(r,,\delta)$ and $\ell$ large enough. For such $\ell$
Theorem~\ref{theo:27ha,1}
follows from Lemma~\ref{lem:28hp,1}.

For each of the remaining
finitely many $\ell$ the point $x_{+,\ell}$ is a hyperbolic attractor
for $\mathbf{G}_{\ell}^{-2}$, so the Poincar\'{e} series is integrable
as well.

\section{Induced maps}
  \subsection{Induced mapping $T_{\ell}$.}
  \begin{defi}\label{defi:3hp,2}
  For every $\ell$ finite and even or infinite, consider the {\em
    fundamental annulus} $A_{\ell} :=
  \Omega_{\ell}\setminus\tau_{\ell}^{-1}\overline{\Omega}_{\ell} $.
  We further define {\em fundamental half-annuli} $A_{\pm,\ell} :=
  A_{\ell} \cap \mathbb{H}_{\pm}$.
  \end{defi}
  For $\ell=\infty$, the fundamental annulus is not in fact a
  topological annulus since it is pinched at $x_0$. However
  fundamental half-annuli are always topological disks by
  Fact~\ref{fa:1hp,3}.

\begin{defi}\label{defi:3hp,3}
For any $z\in A_{\ell}$ define $T_{\ell}(z) = \tau^{n(z)}
H_{\ell}(z)$ where $n(z)$ is chosen so that $T_{\ell}(z) \in
A_{+,\ell} \cup A_{-,\ell}$. The domain of $T_{\ell}$ is the set of
all $z\in A_{\ell}$ for which such $n(z)$ exists.
\end{defi}

From the definition of the fundamental
annulus at most one such $n(z)$ exists for each $z$. Moreover, it can
always by found if the condition is relaxed to $T_{\ell}(z) \in
\overline{A}_{\ell}$. Hence, $T_{\ell}$ is defined on $\Omega_{\ell}$
except for a countable union of analytic arcs.

\paragraph{Branches of $T_{\ell}$.}
Since $A_{\pm,\ell}$ is simply connected and avoids the singularities
of $H_{\ell}$ which are all on $\RR$, the mapping $\tau_{\ell}^{-n}
H_{\ell}$ has univalent inverse branches whose ranges for all $n$
cover the domain of $T_{\ell}$. Thus, the domain of $T_{\ell}$ is a countable union
of topological disks. The restriction of $T_{\ell}$ to any connected
component of its domain will be called a {\em branch} of
$T_{\ell}$. Any branch $\mathfrak{z}:=\mathfrak{z}_{\sigma,s,n,p,\ell}$ can be uniquely
determined by its
\begin{itemize}
\item  {\em side} $\sigma$ which can be $+$ or $-$ depending
on whether the domain of $\mathfrak{z}$ is in $\Omega_{+,\ell}$ or
$\Omega_{-,\ell}$,
\item
{\em sign} $s$ which can be $+$ or $-$ depending on whether the domain of
$\mathfrak{z}$ lies in the upper or lower half-plane,
\item
{\em level} $n$ defined by $\mathfrak{z} = \tau_{\ell}^{n}\exp(\phi_{\sigma,\ell})$
where $\sigma$ is the side of the branch and
\item
  {\em height} $p$. To determine the height map the domain of $\mathfrak{z}$ by
  $\phi_{\sigma,\ell}$. Since the range of $\mathfrak{z}$ which is
  equal to $\exp(\phi_{\sigma,\ell})$ rescaled by a power of
  $\tau_{\ell}$ avoids $\RR$, $\phi_{\sigma,\ell}\left(\Dm\mathfrak{z}\right)$ is contained in a horizontal strip $\{ z\in\CC :\: p\pi <
  \Im z < (p+1)\pi\}$ if the sign $s=+$, or   $\{ z\in\CC :\: (-p-1) \pi <
  \Im z < -p\pi\}$ if $s=-$.
\end{itemize}

So, the range of $\mathfrak{z}_{\sigma,s,n,p,\ell}$ is $A_{+,\ell}$ if and only if
$s(-1)^p=1$, since the statement holds true for $p=0$ and $s=+$ and
then flips each time $s$ changes or $p$ changes by $1$.

\begin{defi}\label{defi:3hp,4}
  A branch is called {\em inner} if its height is positive.
\end{defi}

\begin{lem}\label{lem:3hp,1}
$T_{\ell}$ has no branches of side $-$, height $0$ and level greater
than $1$.
\end{lem}
\begin{proof}
Level greater than $1$ means that the image of the domain of the
branch
by $H_{\ell}$ is inside $\Omega_{\ell}$. Also, the range of any branch
avoids $\RR$ by definition. Height $0$ means that the
domain of such a branch touches the boundary of
$\Omega_{\ell}\cap\HH_s$, where $s$ is the sign of the branch.
However, the set $\partial\bigl(\Omega_{\ell}\cap\HH_s\bigr)
\setminus\RR$ is sent by $H_-$ to $(\tau^2_{\ell},+\infty)$ which is disjoint
from $\Omega_{\ell}$ and so the domain of each of those branches is
adjacent to $\RR$ an actually to $(-\infty,x_{0,\ell})$ since the side
is $-$.

Since $G_{\ell}$ maps $\Omega_{\ell}$ univalently onto
$\Omega_{\ell}\setminus (-\infty,0]$, it makes sense to consider
$G_{\ell}^{-1}\circ H_{\ell} = \tau_{\ell} H_{\ell}^{-1} \circ H_{\ell}$,
$H_{\ell}^{-1}$ is the inverse branch which sends $\Omega_{\ell}\cap
(0,+\infty)$ into $(-\infty,x_{0,\ell})$. It follows that it cancels
with $H_{\ell}$ on the domain of the branch, which is therefore
contained in $\tau^{-1}_{\ell}\Omega_{\ell}$, excluded from $A_{\ell}$
by Definition~\ref{defi:3hp,2}.
 \end{proof}

\paragraph{Generic branches.}
So far we consider the mapping $T_{\ell}$ and its branches which all
depend of $\ell$. However, since a branch is uniquely defined by its
symbol $(\sigma,s,n,p)$ we can also talk of $T$ as ``generic mapping''
independent of $\ell$ and consider its generic branches defined by
their symbols. The only limitation is on the height $p\leq
\frac{\ell}{2}$.

\subsection{Extensibility of compositions of branches.}
\begin{defi}\label{defi:3hp,5}
  Define for any $s\in\ZZ$ and $\ell$ positive and even or infinite:
\begin{align*}
  Z_{\ell} := & [0,\tau^{-1}_{\ell}] \cup \{1,\tau_{\ell}\} \cup
  [\tau^2_{\ell},+\infty)\\
  Z_{+,\ell}(s) := & [0,+\infty)\\
    Z_{-,\ell}(s) := & Z_{\ell} \cup (-\infty,\tau_{\ell}^s]\\
    Z_{\circ,\text{small},\ell}(s) := & (-\infty,0] \cup
  [\tau_{\ell}^s,+\infty)\\
    Z_{\circ,\ell}(s) := & Z_{\ell} \cup Z_{\circ,\text{small},\ell}(s)
    \end{align*}
\end{defi}

\begin{lem}\label{lem:3hp,2}
  Consider any composition of branches in the form
  \[ \xi := \mathfrak{z}_{\sigma_k,s_k,n_k,p_k,\ell}\circ \ldots \circ \mathfrak{z}_{\sigma_1,s_1,n_1,p_1,\ell}\; .\]

Then, there exists $s\in \ZZ$ such that for every
$\mathfrak{m}\in\{+,-,\circ\}$ and every $\hat{s}\in \ZZ$ there exists
$\hat{\mathfrak{m}}\in \{+,-,\circ\}$ and the mapping $\xi$ continues
analytically to a covering of the set $V^{\mathfrak{m}}_{\ell}(s):=\CC\setminus
Z_{\mathfrak{m},\ell}(s)$ defined on a domain which is contained in:

\begin{itemize}
\item
  \[ \hat{\Omega}_{-,\ell}\cup\Omega_{+,\ell} \setminus
    \bigl( [x_{0,\ell},+\infty) \cup
      Z_{\hat{\mathfrak{m}},\ell}(\hat{s}) \bigr) \]
      when $\sigma_1=-$, or
    \item
      \[ \hat{\Omega}_{+,\ell}\cup\Omega_{-,\ell} \setminus
       \bigl( (-\infty,x_{0,\ell}] \cup
       [\tau_{\ell}x_{0,\ell},+\infty) \cup Z_{\hat{\mathfrak{m}},\ell}(\hat{s}) \bigr)\] when $\sigma_1=+$.
\end{itemize}
Furthermore, if the final symbol in the composition
$(\sigma_k,s_k,n_k,p_k) = (+,\pm,2,0)$, then the claim can be
strengthened for $\mathfrak{m}=\circ$ by saying that $\xi$ continues
analytically to a covering  of the set
$V^{\circ}_{\text{large},\ell}(s) := \CC\setminus Z_{\circ,\text{small},\ell}(s)$.
\end{lem}
\begin{proof}
The proof will proceed by induction with respect to $k$.
\paragraph{Verification for $k=1$.}
We begin be representing the branch as
$\tau_{\ell}^{n}\exp(\phi_{\sigma_1,\ell})$
and observing that $V^{\mathfrak{m}}_{\ell}(n) \subset \tau^n_{\ell}
V^{\mathfrak{m}}$ in the notations of
Proposition~\ref{prop:6hp,1}. Also, $V^{\circ}_{\text{large},\ell}(n)
= \tau^n_{\ell} V^{\circ}$. Hence, the claim of that Proposition
holds regarding the existence of a covering and inclusions of its
domain. What is left to do is checking that domain is disjoint from
$Z_{\hat{\mathfrak{m}},\ell}(\hat{s})$.

\subparagraph{Inner branches.}
We set $s:=n$. Since the branch is inner,
$|\Im\phi_{\sigma_1,\ell}|>\pi$ on its domain and hence the domain of
the extension from Proposition~\ref{prop:6hp,1} is disjoint from the real line which contains any
$Z_{\hat{\mathfrak{m}},\ell}(\hat{s})$.

\subparagraph{Branches of height $0$ and side $-$.}
By Lemma~\ref{lem:3hp,1} they have positive level which means $n\leq
1$. We set $s:=n$ in this case, too.
Then, by Proposition~\ref{prop:6hp,1}
there is a covering defined on some domain contained in
$\hat{\Omega}_{-,\ell} \cup \Omega_{+,\ell} \setminus
[x_{0,\ell},+\infty)$. We need to choose $\hat{\mathfrak{m}}$ so that
 this domain is disjoint from $Z_{\mathfrak{m},\ell}(\hat{s})$. The
 appropriate choice here is $\hat{\mathfrak{m}},+$, since then
 $Z_{\mathfrak{m},\ell}(\hat{s})\in [0,+\infty)$, The possible
   intersection of $[0,+\infty)$ with the domain of covering is at
     most $[0,x_{0,\ell})$ whose image under
       $\tau^s_{\ell}\exp(\phi_{-,\ell})$ is $(0,\tau^{s-2}_{\ell}]
     \subset Z_{\ell}$.  This is always contained in
     $Z_{\ell}$.

\subparagraph{Branches of height $0$ and side $+$.}
First, let us assume that $\hat{s}\neq 0$. In this case we also specify
$s:=n$ and choose
\[ \hat{\mathfrak{m}}= \left\{ \begin{array}{ccc} - &\mbox{if}& \hat{s}<0\\
                                                  \circ&\mbox{if}&\hat{s}>0 \end{array}
\right. \; .\]
By this choice,
\[ Z_{\hat{\mathfrak{m}},\ell}(\hat{s}) \cap (x_{0,\ell}, \tau_{\ell}
x_{0,\ell}) = \{1\} \]
and so $1$ is the only possible point of
$Z_{\hat{\mathfrak{m}},\ell}(\hat{s})$ in the domain of the
covering. However, $\tau_{\ell}^s\exp\phi_{+,\ell}(1) =
\tau_{\ell}^{n-2} \in Z_{\ell}$.  So, $1$ is mapped by the branch
outside of $\CC\setminus Z_{\mathfrak{m},\ell}(s)$ and the domain of
the covering is disjoint from $Z_{\hat{m},\ell}(\hat{s})$.

So from now on $\hat{s}=0$. The analysis is further split depending on
$n$.
\begin{itemize}
\item $n\leq 1$. In this case, set $s:=n$ and
  $\hat{\mathfrak{m}}:=-$. Then
  $ Z_{-,\ell}(0) \cap
  (x_{0,\ell},\tau_{\ell}x_{0,\ell}) = (x_{0,\ell},1]$. The image of
   this under the branch is contained $(0,\tau_{\ell}^{-1}] \subset
 Z_{\ell}$ and hence disjoint from the domain of the covering from
 Proposition~\ref{prop:6hp,1}.
\item $n\geq 4$. We also set $s:=n$ and now $\hat{\mathfrak{m}}:=\circ$.
This leads to $Z_{\circ,\ell}(0) \cap
  (x_{0,\ell},\tau_{\ell}x_{0,\ell}) = [1,\tau_{\ell}x_{0,\ell})$
  being excluded from the domain of the covering by the claim of the
  Lemma. Indeed, this set is mapped by the branch to
  $[\tau^{n-2}_{\ell},\tau^n_{\ell})    \subset Z_{\ell}$.
    \item $n=2$. In that case we will set $s:=0$. It is still true
      that $Z_{\mathfrak{m},\ell}(0) \supset
      \CC\setminus\tau^2_{\ell}V^{\mathfrak{m}}$ when $\mathfrak{m}=+,\circ$, but
      not when $\mathfrak{m}=-$,   Instead,
      \[ Z_{-,\ell}(0)= (-\infty,1] \cup \{\tau_{\ell}\} \cup
  [\tau^2_{\ell},+\infty) \supset (-\infty,0] \cup
  [\tau^2_{\ell},+\infty) = \CC\setminus\tau^2_{\ell}V^{\circ} \; .\]
    Hence, Proposition~\ref{prop:6hp,1} is applicable again and a
    covering of $V^{\mathfrak{m}}_{\ell}(0)$ exists by an extension of
    branch to a domain whose intersection with $\RR$ is contained in
    $(x_{0,\ell},\tau_{\ell}x_{0,\ell})$. We must set
    $\hat{\mathfrak{m}} := -$ if $\mathfrak{m}=-$ and
    $\hat{\mathfrak{m}} := \circ$ otherwise. Then
\begin{equation}\label{equ:10ha,1}
    \begin{split}
    Z_{-,\ell}(0) \cap
    (x_{0,\ell},\tau_{\ell}x_{0,\ell}) = & (x_{0,\ell},1] \\
    Z_{\circ,\ell}(0) \cap
    (x_{0,\ell},\tau_{\ell}x_{0,\ell}) = & [1,\tau_{\ell}x_{0,\ell})
    \end{split}
    \end{equation}
which are mapped by the branch to $(0,1]\subset Z_{-,\ell}(0)$ or
    $[1,\tau^2_{\ell})\subset Z_{\circ,\text{small},\ell}(0) \subset Z_{\circ,\ell}(0) \subset Z_{+,\ell}(0)$,
      respectively.
      The additional claim of Lemma~\ref{lem:3hp,2} concerns this type of
      branches. Indeed, we observe that
      $V^{\circ}_{\text{large},\ell}(0) \subset \tau^2_{\ell}
      V^{\circ}$ and the inclusion for the image of
      $Z_{\circ,\ell}(0)$ by the branch has already been observed.
      \item $n=3$. In this case $s:=1$ and as in the previous case one
        checks that
        $Z_{\mathfrak{m},\ell}(1) \supset
        \CC\setminus\tau^3_{\ell}V^{\mathfrak{m}}$ when $\mathfrak{m}=+,\circ$ and
        $Z_{-,\ell}(1) \supset \CC\setminus\tau_{\ell}^3 V^{\circ}$.
        We pick $\hat{\mathfrak{m}}=-$
      if $\mathfrak{m}=-$ and $\circ$ otherwise, as in the preceding
      case, which leads to inclusions~(\ref{equ:10ha,1}). Then the
      branch maps $(x_{0,\ell},1]$ to $(0,\tau_{\ell}]\subset
Z_{-,\ell}(1)$ and $[1,\tau_{\ell}x_{0,\ell})$ to
  $[\tau_{\ell},\tau^3_{\ell})\subset Z_{\circ,\ell}(1) \subset
    Z_{+,\ell}(1)$, respectively.
\end{itemize}
    \paragraph{The inductive step.}
    We decompose $\xi = \xi'\circ\mathfrak{z}$.
    By the inductive claim applied to $\xi'$, $\CC\setminus
    Z_{\mathfrak{m},\ell}(s)$ is covered by an extension of $\xi'$
    restricted to a domain which then itself is covered by an
    extension of $\mathfrak{z}$. This yields a covering by Fact~\ref{fa:3ha,1}
    whose domain is contained in the domain of the extension of
    $\mathfrak{z}$.
\end{proof}

Let us conclude with a technical observation.
\begin{lem}\label{lem:10hp,1}
For any $s\in \ZZ$ each of the sets $Z_{+,\ell}(0),
Z_{\circ,\ell}(0), Z_{-,\ell}(0)\cup [\tau_{\ell},+\infty)$ contains
   $Z_{\mathfrak{m},\ell}(s)$ for some $\mathfrak{m}$, where
  $\mathfrak{m}$ is generally different for each of the three cases.
\end{lem}
\begin{proof}
Certainly $Z_{+,\ell}(0) \supset Z_{+,\ell}(s)$ for any $s$ since this domain is
independent of $s$. When $s<0$, $Z_{-,\ell}(s)=(-\infty,\tau^{-1}_{\ell}]
\cup \{1,\tau_{\ell}\} \cup [\tau^2_{\ell},+\infty)$, cf. Definition~\ref{defi:3hp,5}.
This is contained in both
$Z_{\circ,\ell}(0)$ and $Z_{-,\ell}(0)$. When $s=0$ the statement is
obvious. For $s>0$ we get $Z_{\circ,\ell}(0) \supset Z_{\circ,\ell}(s)$
as well as $Z_{-,\ell}(0)\cup [\tau_{\ell},+\infty) \supset
  Z_{\circ,\ell}(s)$.
\end{proof}

\paragraph{Univalent extensibility.}
While Lemma~\ref{lem:3hp,2} provides a general statement which was
suitable for a proof by induction, the goal of a dynamicist is to work
with univalent extensions. We will proceed to derive them.

\begin{defi}\label{defi:10hp,1}
  For $0\leq \theta_0,\theta_1 \leq \pi$ define the domain
  \begin{multline*} \tilde{V}(\theta_0,\theta_1) :=\\
    \CC \setminus \bigl( [0,\tau_{\ell}^{-1}]
  \cup [\tau_{\ell},+\infty) \cup \{ r\exp(-\iota\theta_0) :\: r\geq
    0\} \cup \{ 1 + r\exp(-\iota\theta_1) :\: r\geq 0\} \bigr) \; .
\end{multline*}
    By definition, $V(\theta_0,\theta_1)$ is the connected component of
    $\tilde{V}(\theta_0,\theta_1)$ which
  contains $\HH_+$.

\end{defi}

\begin{prop}\label{prop:10hp,1}
Let $\xi$ be any composition of branches of $T_{\ell}$ with the range equal
to $A_{+,\ell}$, without loss of generality. Suppose that the domain $\xi$ is
contained $\Omega_{\sigma_1,\ell}$, $\sigma_1=+,-$. Then, for any
$0\leq\theta_0,\theta_1\leq\pi$ the branch $\xi$ has an
analytic continuation which maps univalently onto
$V(\theta_0,\theta_1)$ and the domain of the extension satisfies the
inclusion from the claim of Lemma~\ref{lem:3hp,2}.
\end{prop}

As a consequence of Lemma~\ref{lem:3hp,2} and
Lemma~\ref{lem:10hp,1}, $\xi$ has three different covering extensions: $\xi_+$
with the range $\CC\setminus Z_{+,\ell}(0)$, $\xi_{\circ}$ with the
range $\CC \setminus Z_{\circ,\ell}(0)$ and $\xi_{-}$ to the range
$\CC \setminus \bigl( Z_{-,\ell}(0) \cup [\tau_{\ell},+\infty)
  \bigr)$. The domains of those extensions satisfy the inclusions
  from the claim of Lemma~\ref{lem:3hp,2}. Since each of the ranges is
  a simply-connected domain, each covering reduces to a univalent map.
  They all coincide on the preimage of $\HH_+$.

  Now $V(\theta_0,\theta_1)\setminus \HH_+$ splits into three
  connected components: $V_-(\theta_0,\theta_1)$ which
  contains the interval $(1,\tau_{\ell})$,
  $V_{\circ}(\theta_0,\theta_1)$ containing $(\tau_{\ell}^{-1},1)$ and
  $V_+(\theta_0,\theta_1)$ which contains $(-\infty,0)$. Define the
  domain of the desired extension as the union of $\xi^{-1}(\HH_+)$ and
  $\xi_{\mathfrak{m}}^{-1}\bigl( V_{\mathfrak{m}}(\theta_0,\theta_1)
  \bigr)$. Since all three extensions coincide on the preimage of
  $\HH_+$ and map the remaining parts of the domain into disjoint
  sets, the extension of $\xi$ on this domain is analytic and
  one-to-one. It is also proper, since each of the three extensions
  was a homeomorphism, thus univalent.

  This concludes the proof of Proposition~\ref{prop:10hp,1}.

  \subsection{Uniform tightness.}
  In this section we consider a generic mapping $S$ induced by the generic
  mapping $T$.
  \begin{defi}\label{defi:22ja,1}
    A generic mapping $S$ induced by $T$ is a collection  of finite sequences
    of symbols  $(\sigma_j,s_j,n_j,p_j)_{j=1}^r$. Given $\ell\leq\infty$
    the induced mapping $S_{\ell}$ is obtained first by defining the
    {\em return time} $r_{S,\ell}(z)$ as the length $r$ of the longest sequence
    in $S$ which observes the limitation $0\leq p_j \leq \frac{\ell}{2}$ and
    such that the composition
    \[ \mathfrak{z}_{\sigma_r,s_r,n_r,p_r}\circ\cdots\circ \mathfrak{z}_{\sigma_1,s_1,n_1,p_1} \] is applicable at $z$. Then one obtains the induced mapping
    $S_{\ell} :\: S_{\ell}(z) := T^{r_{S,\ell}(z)}(z)$.
  \end{defi}

For example, $T$ itself is the collection of all possible sequences
$(\sigma,s,n,p)$ of length $1$ and the empty collection determines the
identity map.

  \begin{defi}\label{defi:3jp,1}
  A generic induced mapping $S$ will be called {\em uniformly tight} if
    for every $\epsilon>0$ there is $\ell_0(\epsilon)$ and finite set
    of generic branches
    $\mathfrak{Z}(S,\epsilon)$ of $S$
        such that if we define
    $\omega_{\ell}(\mathfrak{Z}):=\bigcup_{\mathfrak{z}\in\mathfrak{Z}(S,\epsilon)}
    \Dm(\mathfrak{z}_{\ell})$, then for all
    $\ell\geq\ell_0(\epsilon)$
    \[ \int_{\Omega_{\ell}\setminus\omega_{\ell}(\mathfrak{Z})} r_{S,\ell}(x+\iota y)\,dx\,dy <
    \epsilon .\]
\end{defi}

  \paragraph{A fact about convergence in measure.}
  \begin{fact}\label{fa:4np,1}
  Suppose that $(W_n),\, n=1,\cdots,\infty$ are bounded open sets and
  $\overline{W}_n \rightarrow \overline{W}_{\infty}$ in the Hausdorff topology.
  If $|\partial W_{\infty}|=0$, then
  \[ \lim_{n\rightarrow\infty}
  \left|(\overline{W}_{\infty}\setminus W_n) \cup
  (\overline{W}_n\setminus W_{\infty}) \right| = 0 .\]
  \end{fact}
  \begin{proof}
  Consider an open neighborhood with arbitrarily
  small measure which contains $\partial W_{\infty}$.
  \end{proof}

\begin{lem}\label{lem:21ja,1}
  \[ \lim_{\ell\rightarrow\infty} \int
  \|\chi_{\Omega_{\pm,\ell}} - \chi_{\Omega_{\pm,\infty}}\|\,
  d\Leb_2 = 0 \] and likewise if
  $\mathfrak{z} := \mathfrak{z}_{\sigma_k,s_k,n_k,p_k}\circ\cdots\circ
  \mathfrak{z}_{\sigma_1,s_1,n_1,p_1}$, then
  \[ \lim_{\ell\rightarrow\infty} \int\left|\chi_{\Dm(\mathfrak{z}_{\ell})} -
  \chi_{\Dm(\mathfrak{z}_{\infty})}\right|\, d\Leb_2 = 0 .\]
\end{lem}
\begin{proof}
By Proposition~\ref{prop:11jp,1} we observe that the closures of
the sets under consideration converge in the Hausdorff topology and
the claim follows
from Fact~\ref{fa:4np,1}.
\end{proof}

\begin{coro}\label{coro:21ja,1}
   For any generic induced mapping $S$
   and for every $\epsilon>0$ there is $\ell_0(\epsilon)$ and finite set
    of generic branches
    $\mathfrak{Z}(S,\epsilon)$ of $S$ such that for every $\ell\geq\ell_0$
    \[ \sum_{\mathfrak{z}\notin\mathfrak{Z}(S,\epsilon)}
    \left|\Dm(\mathfrak{z}_{\ell})\right| < \epsilon .\]
\end{coro}
\begin{proof}
  This is an easy consequence of Lemma~\ref{lem:21ja,1}.
\end{proof}

\paragraph{Uniform tightness under composition.}

\begin{defi}\label{defi:4ja,1}
We recall that a univalent mapping $\varphi :\: U \rightarrow V$ has
{\em distortion bounded} by $Q$ onto $Z\subset V$ provided that
$\sup \bigl\{ \log\left|\frac{D\varphi(z_1)}{D\varphi(z_2)}\right|
:\: z_1,z_2\in \varphi^{-1}(Z) \bigr\} \leq Q$.
\end{defi}

The next lemma immediately generalizes by induction to any finite
composition of uniformly tight mappings.

By the {\em domain} of an induced map we understand the set where its
return time is positive.

\begin{lem}\label{lem:4ja,1}
Suppose $S_1,S_2$ are generic mappings induced by $T$ and for every
$\ell\geq\ell_0$ the image of every branch of $S_{1,\ell}$ contains the domain
of $S_{2,\ell}$; moreover there is $Q$ such that the distortion of
every branch of $S_{1,\ell}$, for every $\ell\geq\ell_0$ is bounded by
$Q$.

Then, if $S_1, S_2$ are uniformly tight, so is $S_2 \circ S_1$.
\end{lem}
\begin{proof}
Fix an $\epsilon>0$. Let $\mu$ denote the Lebesgue measure of the
domain of $S_2$.
The hypothesis of uniformly bounded distortion implies that $S_1$
transports the Lebesgue measure with a Jacobian uniformly bounded
above by $K_0\exp(Q)$ and below by $K_0\exp(-Q)$, where $K_0>0$ is a constant. By the uniform tightness of $S_2$ and
Corollary~\ref{coro:21ja,1}
given any $\epsilon_1,\epsilon_2>0$ we find a finite set
$\mathfrak{Z}_2$ of branches of $S_2$ such that for all $\ell$ sufficiently large

\begin{equation*}
  \begin{split}
\sum_{\mathfrak{z}\notin\mathfrak{Z}_2} \left|\Dm(\mathfrak{z}_{\ell})\right| &
< Q^{-1}\mu\epsilon_1\\
\int_{\Omega_{\ell}\setminus\bigcup_{\mathfrak{z}\in\mathfrak{Z}_2}\Dm(\mathfrak{z}_{\ell})}
r_{S_2,\ell}(x+\iota y)\,dx\,dy & < Q^{-1}\mu\epsilon_2
  \end{split}
\end{equation*}

Now for every branch $\mathfrak{z}_{\ell}$ of $S_{1,\ell}$
the preimages of the domains of
branches not in $\mathfrak{Z}_2$ occupies at most $Q\epsilon_1$-part
of $\Dm(\mathfrak{z})$.  We get the estimate
\[ \int_{\Dm(\mathfrak{z}_{\ell})\setminus\mathfrak{z}_{\ell}^{-1}\left(\bigcup_{\mathfrak{w}\in\mathfrak{Z}_2}\Dm(\mathfrak{w}_{\ell})\right)}
r_{S_2\circ S_1,\ell}(x+\iota y)\,dx\,dy \leq \bigl( \epsilon_1
r_{S_1,\ell}\left(\Dm(\mathfrak{z}_{\ell})\right) + \epsilon_2\bigr) \left|
\Dm(\mathfrak{z}_{\ell}) \right|  .\]

Summing up over all branches of $\mathfrak{z}_{\ell}$ of $S_1$ we arrive at
\[ \int_{\Omega_{\ell}\setminus\bigcup_{\mathfrak{w}\in\mathfrak{Z}_2}\Dm(\mathfrak{w}_{\ell}\circ\mathfrak{z}_{\ell})}
r_{S_2\circ S_1,\ell}(x+\iota y)\,dx\,dy \leq \epsilon_1
\int_{\Omega_{\ell}} r_{S_1,\ell}(x+\iota y)\,dx\,dy + \epsilon_2 \left|
\Omega_{\ell} \right|  .\]

Uniform tightness of $S_1$ implies that $\int_{\Omega_{\ell}}
r_{S_1,\ell}(x+\iota y)\,dx\,dy$ is uniformly bounded for all $\ell$
large enough. Hence, $\epsilon_1$ and $\epsilon_2$ can be chosen so
that
\begin{equation}\label{equ:22jp,2}
  \forall\ell\geq\ell_0\; \int_{\Omega_{\ell}\setminus\bigcup_{\mathfrak{w}\in\mathfrak{Z}_2}\Dm(\mathfrak{w}_{\ell}\circ\mathfrak{z}_{\ell})}
r_{S_2\circ S_1,\ell}(x+\iota y)\,dx\,dy < \frac{\epsilon}{2} .
\end{equation}

Now use the uniform tightness of $S_1$ to find a finite set
$\mathfrak{Z}_1$ of its branches such that
\begin{equation}\label{equ:22jp,1}
\begin{split}
  \sum_{\mathfrak{z}\notin\mathfrak{Z}_1} \left|\Dm(\mathfrak{z}_{\ell})\right| &
< \epsilon_3 \\
  \int_{\Omega_{\ell}\setminus\bigcup_{\mathfrak{z}\in\mathfrak{Z}_1}\Dm(\mathfrak{z})}
r_{S_1,\ell}(x+\iota y)\,dx\,dy  &< \frac{\epsilon}{4} .
\end{split}
\end{equation}

Since $\mathfrak{Z}_2$ is a finite set, the return time of all its
branches is bounded by $R_2<\infty$. Hence, for any branch
$\mathfrak{z}\notin \mathfrak{Z}_1$, we estimate
\[ \int_{\Dm(\mathfrak{z}_{\ell})\setminus\mathfrak{z}_{\ell}^{-1}\left(\bigcup_{\mathfrak{w}\in\mathfrak{Z}_2}\Dm(\mathfrak{w}_{\ell})\right)}
r_{S_2\circ S_1,\ell}(x+\iota y)\,dx\,dy \leq \bigl(
r_{S_1,\ell}\left(\Dm(\mathfrak{z}_{\ell})\right) + R_2\bigr) \left|
\Dm(\mathfrak{z}_{\ell}) \right|  \]
which after summing up over all $\mathfrak{z}\notin \mathfrak{Z}_1$
leads to
\begin{multline*} \int_{\bigcup_{\mathfrak{z}\notin\mathfrak{Z}_1}\Dm(\mathfrak{z}_{\ell})\setminus\bigcup_{\mathfrak{w}\in\mathfrak{Z}_2,\mathfrak{z}\notin\mathfrak{Z}_1}\Dm(\mathfrak{w}_{\ell}\circ\mathfrak{z}_{\ell})}
r_{S_2\circ S_1,\ell}(x+\iota y)\,dx\,dy \leq \\
\int_{\bigcup_{\mathfrak{z}\notin\mathfrak{Z}_1} \Dm(\mathfrak{z})}
r_{S_1,\ell}(x+\iota y)\,dx\,dy  + R_2
\sum_{\mathfrak{z}\notin\mathfrak{Z}_1} \left|
\Dm(\mathfrak{z}_{\ell}) \right| \leq \frac{\epsilon}{4}
+R_2\epsilon_3
\end{multline*}
where the final estimates come from inequalities~(\ref{equ:22jp,1}).
We now take $\epsilon_3$ so small that $R_2\epsilon_3 <
\frac{\epsilon}{4}$ and together with estimate~(\ref{equ:22jp,2}) we
obtain
\[ \int_{\Omega_{\ell}\setminus\bigcup_{\mathfrak{w}\in\mathfrak{Z}_2,\mathfrak{z}\in\mathfrak{Z}_1}\Dm(\mathfrak{w}_{\ell}\circ\mathfrak{z}_{\ell})}
r_{S_2\circ S_1,\ell}(x+\iota y)\,dx\,dy < \epsilon .\]
\end{proof}

\subsection{Post-singular branches.}
The singular value of many branches which is contained in the
fundamental annulus is $1$. That singular value is adjacent to the
domains of two branches $\mathfrak{z}_{+,\pm,2,0,\ell}$ for all
$\ell$. These branches will be described as post-singular and will
require special attention if we want to induce a uniformly hyperbolic
map.

Post-singular branches both have the form $\tau_{\ell}H_{\ell}$. One
quickly sees that this is conjugated to the multiplication by
$\tau_{\ell}$ by $H_{\ell}$:
\[ \mathbf{H}_{\ell}^{-1}\tau_{\ell} H^2_{\ell} =
\mathbf{H}_{\ell}^{-1} \tau_{\ell} H_{\ell} G_{\ell} \tau_{\ell} =
\mathbf{H}_{\ell}^{-1} H_{\ell} \tau_{\ell} = \tau_{\ell} \]
where $\mathbf{H}_{\ell}^{-1}$ is the inverse branch defined on
$\CC\setminus \left( (-\infty,0] \cup
\tau^2_{\ell},+\infty)\right)$.

Define {\em exit time} at $\mathfrak{E}_{\text{sing},\ell}(z)$ as the
smallest non-negative number of iterates of
$\tau_{\ell} H_{\ell}$ needed to map $z$ outside the union of domains
of the post-singular branches.

\begin{lem}\label{lem:10kp,1}
There exists a constant $K$ such that for some $\ell_0$, all $\ell\geq\ell_0$
and all $z\in\CC$

\[ \mathfrak{E}_{\text{sing},\ell}(z) \leq
K\log\max\left(\frac{1}{|z-1|},2\right) .\]
\end{lem}
\begin{proof}
  For $\ell$ large enough
  $\mathbf{H}_{\ell}^{-1}$ maps
  $\Omega_{+,\ell}$ for $\ell$ into $D(0,R_0)$,
  for a fixed $R_0$, and with uniformly bounded distortion. Then the exit time is bounded as follows:
  \[ \mathfrak{E}_{\text{sing},\ell}(z) \leq \frac{\log
    \frac{R_0}{\left|\mathbf{H}_{\ell}^{-1}(z)\right|}}{\log\tau_{\ell}} .\]
  Since  $|\mathbf{H}_{\ell}^{-1}(z)| > K_1 |z-1| $ with $K_1>0$ because of
  the bounded distortion, the estimate of the Lemma follows.
\end{proof}

\begin{defi}\label{defi:7mp,1}
  Define the generic {\em post-singularly refined} map $\tsing$
  as consisting of sequences (cf. Definition~\ref{defi:22ja,1}) which
  begin with any symbol {\em other than} $(+,\pm,2,0)$
  and followed by any, possibly empty, sequence consisting of the
  two post-singular symbols $(+,\pm,2,0)$.
\end{defi}

Dynamically, $\tsing$ is $T$ restricted to the complement of the domains of the post-singular branches followed by the first exit map
from the union of the domains of the post-singular branches.

Our main goal will be the following.
\begin{prop}\label{prop:7mp,1}
  Mapping $\tsing$ is uniformly tight, cf. Definition~\ref{defi:3jp,1}.
\end{prop}

For convenience, in the proof we will use
\[ \dsing(\ell) := \Dm(\mathfrak{z}_{+,+,2,0,\ell}) \cup
\Dm(\mathfrak{z}_{+,-,2,0,\ell}),\]
i.e. the union of the domains of
the post-singular branches.

\paragraph{Reduction to pointy branches.}
A natural way to approach the proof is by using
Lemma~\ref{lem:4ja,1}. Here $S_{2,\ell}$ is the first exit map from $\dsing(\ell)$. $S_2$ is uniformly tight by
Lemma~\ref{lem:10kp,1}, however $S_1$ cannot be made $T$, since not
all branches of $T_{\ell}$ map onto $\dsing(\ell)=\Dm(S_{2,\ell})$
with distortion uniformly bounded for large $\ell$.

Since $\exp(\phi_{\pm,\ell})$ extend to coverings of
$\CC\setminus \left(\{0\} \cup \{ \tau_{\ell}^{2n} :\: n=0,1,\cdots\}\right)$ by Theorem~\ref{theo:3hp,1}, all branches of
odd level can be continued univalently to map onto
$\Omega_{+,\ell}$. The same can be said of all branches of positive
level or non-zero height. Thus, if in $T$ one eliminates all symbols
except for $(\pm,\pm,-2k,0) :\: k=0,1,\cdots$ such a map $S_{1,\ell}$ for every
$\ell$ large enough will map onto $\dsing(\ell)$ with uniformly
bounded distortion and $S_2\circ S_1$ is uniformly tight by
Lemma~\ref{lem:10kp,1}. For the proof of Proposition~\ref{prop:7mp,1}
it will suffice to demonstrate that $S_2 \circ S'_{1}$ is uniformly
tight where $S'_1$ consists of sequences of length $1$ of ``pointy''
symbols  $(\pm,\pm,-2k,0) :\: k=0,1,\cdots$. The use of adjective
``pointy'' is based on the fact that the domains of those branches for
any $\ell$ are exactly though that touch the cusps in the boundary of
$\Omega_{\ell}$. Those cusps are critical points, or essential
singularities in the case of $\ell=\infty$, of analytic continuations
of those branches for which $1$ is the critical, respectively asymptotic,
value. Since the limiting singularities are flat, the uniform integrability
$\mathfrak{E}_{\text{sing},\ell}\circ S'_{1,\ell}$ is far from obvious
and will require estimates.

\paragraph{A uniform estimate with respect to $\ell$.}
Let us write
\[ \mathfrak{Q}(\lambda,\ell) := \{z\in\Omega_{-,\ell} :\: \Re \phi_{-,\ell}(z) < -\lambda,\, \left|\Im\phi_{-,\ell}(z)\right| < \pi\} .\]
\begin{lem}\label{lem:12ma,1}
  \[ \exists \ell_0<\infty\; \lim_{\lambda_0\rightarrow\infty} \sup_{\ell\geq\ell_0} \int_{\mathfrak{Q}(\lambda_0,\ell)} \Re \phi_{-,\ell}(z)\, d\Leb_2(z) = 0 .\]
\end{lem}
\begin{proof}
The problem of uniformity with respect to $\ell$ here is different
from the situation treated in the proof of
Theorem~\ref{theo:27ha,1}. It is described in literature as ``dominant
convergence'', see~\cite{profesorus1} Thms. 8.1-8.3. The result in our
notations can be stated as follows.
\begin{fact}\label{fa:12mp,1}
 For every $K>1$ there exist $\lambda(K)>0$ and $\ell_0(K)<\infty$ such that
 the mapping $\frac{1}{\phi_{-,\ell}(z)}$ on the set $\{z\in\Omega_{-,\ell} :\:
 \Re\phi_{-,\ell}(z) < -\lambda(K)\}$ for all $\ell\geq\ell_0$ takes form
  $\frac{1}{\phi_{-,\ell}(z)} = \Upsilon_{\ell}(2C_{\ell}z^2)$ where
 $C_{\ell} = -D^3G_{\ell}(x_0,\ell) > 0$ and $\Upsilon_{\ell}$ is a
 $K$-quasi-conformal mapping of $\hat{\CC}$ fixing $0,1,\infty$.
\end{fact}

This will now be used to estimate the Lebesgue measure of
$\mathfrak{Q}(\lambda,\ell)$ for $\lambda>\lambda(K)$. The image $I_\lambda$ of the set $\{ u\in\CC :
-2\lambda<\Re u< -\lambda, -\pi < \Im u <\pi \}$ by the complex inversion is
easily seen to have measure bounded above by $K_1
\lambda^{-3}$. Consider now the set $\Upsilon_{\ell}^{-1}(I_\lambda)$.

Since
$\Upsilon_{\ell}$ belongs to a compact family of quasi-conformal
mappings the constant in the change of area theorem of Bojarski,
see~\cite{lehvi} Theorem 5.2, is uniform and the measure of that set is
bounded above by $K_1 \lambda^{-\frac{5}{2}}$ for all
$\ell\geq\ell_0(K)$ provided that $K>1$ was
chosen close enough to $1$. Additionally, by the H\"{o}lder continuity
of quasiconformal mappings in the usual sense, that set is disjoint
from $D\left(0, K_3 \lambda^{-\frac{5}{4}}\right)$.

Next, we take a preimage of the same set
by the mapping $y=2C_{\ell}z^2$ observing that
$C_{\ell}$ is bounded below by $\frac{C_{\infty}}{2} > 0$ for all
$\ell$ large enough.

The Jacobian of the inverse mapping is equal to
$\frac{1}{8C_{\ell}|y|}$ which is bounded
above by $K_4\lambda^{\frac{5}{4}}$. This leads to
\[ \left|\mathfrak{Q}(\lambda,\ell)\setminus
\mathfrak{Q}(2\lambda,\ell) \right| \leq K_6 \lambda^{-\frac{5}{4}} \]
for all $\ell\geq\ell_0$ which by summing up a geometric progression
leads to
\begin{equation}\label{equ:12mp,1}
  \left|\mathfrak{Q}(\lambda,\ell)  \right| \leq 2K_6
  \lambda^{-\frac{5}{4}} .
\end{equation}

Let us write $q_{\ell}(\lambda) =
\bigl|\mathfrak{Q}(\lambda,\ell)\bigr|$. Then the integral in the claim of the
Lemma can be written as
\[ \int_{\mathfrak{Q}(\lambda_0,\ell)} \Re \phi_{-,\ell}(z)\,
d\Leb_2(z) = \int_{\lambda_0}^{\infty} \lambda dq_{\ell}(\lambda) =
\left. \lambda q_{\ell}(\lambda) \right|_{\lambda_0}^{\infty} -
  \int_{\lambda_0}^{\infty} q_{\ell}(\lambda)\,d\lambda =
  O\left(\lambda_0^{-\frac{1}{4}}\right) \]
 independently of
 $\ell\geq\ell_0$ by estimate~(\ref{equ:12mp,1}).
\end{proof}

\paragraph{The primary pair of pointy branches.}
In this fragment we consider the generic induced map $S_2
\circ\mathfrak{z}_{+,\pm,0,0}$. It consists of sequences of symbols
which begin with $(+,\pm,0,0)$ and are followed by a sequence of post-singular
symbols $(+,\pm,2,0)$ of any positive finite length. For any
$\ell$
\begin{equation}\label{equ:12mp,3} \mathfrak{z}_{+,\pm,0,0,\ell} = \tau^{-1}_{\ell} H_{+,\ell} =
H_{-,\ell} \circ G_{\ell} = H_{-,\ell}\circ H_{-,\ell}\circ \tau^{-1}_{\ell} .
\end{equation}

\begin{lem}\label{lem:12mp,1}
The generic mapping $S_2\circ\mathfrak{z}_{+,\pm,0,0}$ is uniformly
tight.
\end{lem}
\begin{proof}
The final action by $H_{-,\ell}$ in the representation~(\ref{equ:12mp,3})
with the image $\dsing(\ell)$ has distortion uniformly bounded in
terms of $\ell$. Taking into account Lemma~\ref{lem:10kp,1} we
conclude that $\mathfrak{E}_{\text{sing},\ell} \circ
\mathfrak{z}_{+,\pm,0,0}(z) \leq K_1
\max\left( -\Re\phi_{-,\ell}(\tau^{-1}_{\ell} z),1\right)$. Also, since
the part of the border $\Omega_{+,\ell}$ adjacent to $\tau_{\ell}
x_{0,\ell}$ consists of preimages of segments in the positive  half-line, we have
$|\Im\phi_{-,\ell}(\tau^{-1}_{\ell} z)|<\pi$ for all
$z\in\Dm\mathfrak{z}_{+,\pm,0,0,\ell}$.

By Lemma~\ref{lem:12ma,1}, for some $\ell_0<\infty$ any $\epsilon>0$
there exists $r(\epsilon)>0$
and all for all $\ell\geq \ell_0$
\[ \int_{\Dm\mathfrak{z}_{+,\pm,0,0,\ell}} \left(
\mathfrak{E}_{\text{sing},\ell}\cdot
\chi_{|z-1|<r(\epsilon)}\right)\circ\mathfrak{z}_{+,\pm,0,0,\ell}(z)\,
d\Leb_2(z) < \epsilon .\]

By Lemma~\ref{lem:10kp,1} there is an upper limit $K_2$, independent
of $\ell$ sufficiently large, on $\mathfrak{E}_{\text{sing},\ell}$ for
branches not contained in $D(1,r(\epsilon))$ and since only two
post-singular symbols are allowed that translates to a number of
branches bounded depending on $\epsilon$ for all such $\ell$.
\end{proof}

\paragraph{Proof of Proposition~\ref{prop:7mp,1}.}
All remaining pointy branches have the form $\tau_{\ell}^{-n}H$ for $n>1$ and
hence are in the form $\mathfrak{z}_{+,\pm,0,0,\ell} \circ
G_{\ell}^{n-1}$. Since $G_{\ell}$ maps as a covering of
$\CC\setminus\left(\{0\}\cup [\tau^2_{\ell},+\infty)\right)$ and its
  post-singular set under iteration on $\tau_{\ell}\Omega_{-,\ell}$ is
  contained in $[0,1]$, the mapping by iterates of $G_{\ell}$
  onto $\Dm\mathfrak{z}_{+,\pm,0,0,\ell}$ has distortion uniformly bounded
  independently of $\ell$ sufficiently large. Thus, for any particular
  pointy branch its composition with $S_2$ is uniformly tight by
  Lemma~\ref{lem:4ja,1}. On the other hand, the closures of domains of pointy
  branches for $\ell=\infty$ converge to $\{x_{0,\infty}\}$ in
  Hausdorff topology so for $\ell \geq \ell(\epsilon)$ only a fixed
  number are not contained in $D(x_{0,\infty},\epsilon)$ and hence
  their joint measure is bounded by $\pi\epsilon^2$. As a corollary to
  Lemma~\ref{lem:12mp,1} the integral
  $\int_{\Dm\mathfrak{z}_{=,\pm,0,0.\ell}}
  \mathfrak{E}_{\text{sing},\ell}\circ
  \mathfrak{z}_{+,\pm,0,0.\ell}(z)\,d\Leb_2(z) \leq K_1$ for all $\ell$
  sufficiently large. Consequently, on the domain of any pointy branch
  the integral of the return time of $\tsing$ is bounded by
  $K_2$ times the measure of that domain of the branch. Hence,
  all pointy branches except for finitely many can carry arbitrarily
  small integral of the return time.

  \subsection{Parabolic branches.}\label{sec:4qa,1}
  Another pair of branches which cause problems are {\em parabolic
    branches} with symbols $(-,\pm,1,0)$. They have the form
  $\tau_{\ell}^{-1}H_{\ell}$ which is just $G_{\ell}$ conjugated by
  $\tau_{\ell}$ and has a period $2$ point $\tau_{\ell}^{-1}x_{\pm,\ell}$
  which is on the boundary of the domain of such a branch. Hence the
  name, since when $\ell\rightarrow\infty$ this period $2$ orbit bifurcates into a
  parabolic fixed point $x_{0,\infty}$.

  We will proceed to get rid of them by inducing, much in the way we
  dealt with the post-singular branches, except that now $\tsing$
  rather than $T$ is our starting point.

  Define $\dpar(\ell)$, $\ell\leq\infty$ as the union of the domains
  of two parabolic branches.
  Next, the {\em exit time} $\mathfrak{E}_{\text{par},\ell}(z)$ as
  the smallest non-negative number of iterates of
  $\tau_{\ell}^{-1}H_{-,\ell}$ needed to take $z$ outside
  $\dsing(\ell)$.

  \begin{defi}\label{defi:21ma,1}
  Define the generic {\em parabolically refined} map $\tpar$
  as consisting of sequences of $\tsing$,
  cf. Definition~\ref{defi:7mp,1}  which
  begin with any symbol {\em other than} a parabolic one $(-,\pm,1,0)$
  and followed by any, possibly empty, sequence consisting of the
  two parabolic symbols $(-,\pm,1,0)$.
\end{defi}

  Our goal is
  \begin{prop}\label{prop:21ma,1}
  Mapping $\tpar$ is uniformly tight, cf. Definition~\ref{defi:3jp,1}.
  \end{prop}

  Here $S_2 := S_{\text{par}}$ is the first exit
  map from the parabolic branches, given by all non-empty sequences
  of parabolic symbols $(-,\pm,1,0)$ and $S_1$ is $\tsing$ restricted
  by excluding sequences with initial parabolic symbols as in
  Definition~\ref{defi:21ma,1}. Since after such exclusion the
  distortion of the map is uniformly bounded, then we are in the
  position to use Lemma~\ref{lem:4ja,1}.

  The bounded distortion follows from
  the additional claim of Lemma~\ref{lem:3hp,2} by which the branches
  of $\tsing(\ell)$ extend univalently onto
  $V^{\circ}_{\text{large},\ell}(0)$ which compactly contains
  $\dpar(\ell)$ and the nesting is uniform for large $\ell$ by
  Proposition~\ref{prop:11jp,1}.

  Hence, Proposition~\ref{prop:21ma,1} is reduced to the uniform
  tightness of $S_{\text{par}}$.

  \paragraph{Connection with Theorem~\ref{theo:27ha,1}.}
  That theorem will be our main tool, since after conjugation by
  $\tau_{\ell}$ the pair of parabolic branches becomes $G_{\ell}$ and
  $\tau_{\ell}\dpar(\ell)$ is contained in the complement of
  $\Omega_{\ell}$.

  For $N$ natural define
  $\dpar(\pm,N,\ell) := \{ z\in\dpar(\ell)\cap\HH_{\pm} :\:
  \mathfrak{E}_{\text{par},\ell}(z) \geq N \}$.

  \begin{lem}\label{lem:21ma,1}
  For any $r>0$ there are $\ell(r), N(r)<\infty$ such that for
  every $\ell\geq\ell(r)$ the inclusion $\dpar\left(\pm,N(r),\ell\right) \subset
  D(x_{0,\infty}\tau_{\infty}^{-1},r)$ holds.
  \end{lem}
  \begin{proof}
  If not, then by taking convergent subsequences we construct a point
  $z_0$ whose complete forward orbit by $G_{\infty}$ is contained in a
  bounded set and avoids a wedge $\{ x_{0,\infty}+\zeta :\:
  |\arg\zeta^2|<\frac{\pi}{4}, 0<|\zeta|<R\}$ with some $R>0$. This
  is not consistent with the action of $G_{\infty}$ in a half-plane
  under which every bounded orbit tends to $x_{0,\infty}$ tangentially
  to the real line.
  \end{proof}

  Denote by $G_{1,\ell} = \tau^{-1}_{\ell}G_{\ell}\tau_{\ell}$ and write
  $\mathbf{G}_{1,\ell}^{-1}$ for its principal inverse branch,
  cf. Definition~\ref{defi:1hp,1}.
  Now estimate for any $n\geq 1$
  \begin{multline}\label{equ:21mp,1}
    \int_{\dpar(+,n,\ell) \cup \dpar(-,n,\ell)}
      \mathfrak{E}_{\text{par},\ell}(z)\,d\Leb_2(z) = \\
      (n-1)\left(|\dpar(+,n,\ell)|+|\dpar(-,n,\ell)|\right)
      + \\\sum_{k=0}^{\infty} \left( |\dpar(+,n+k,\ell)| +
      |\dpar(-,n+k,\ell)| \right) = \\
  (n-1)\left(|\dpar(+,n,\ell)|+|\dpar(-,n,\ell)|\right) +\\
    \sum_{k=0}^{\infty} \left|
    \mathbf{G}_{1,\ell}^{-2k}\left(\dpar(+,n,\ell) \cup
    \dpar(+,n+1,\ell)\right)\right| +\\     \sum_{k=0}^{\infty} \left|
    \mathbf{G}_{1,\ell}^{-2k}\left(\dpar(-,n,\ell) \cup
    \dpar(-,n+1,\ell)\right)\right| \leq \\
    (n-1)\left(|\dpar(+,n,\ell)|+|\dpar(-,n,\ell)|\right) +\\
    2\sum_{k=0}^{\infty} \left|
    \mathbf{G}_{1,\ell}^{-2k}\left(\dpar(+,n,\ell)\right)\right|+2\sum_{k=0}^{\infty} \left|
    \mathbf{G}_{1,\ell}^{-2k}\left(\dpar(-,n,\ell)\right)\right| =\\
      \sum_{s=+,-} \Bigl[ (n-1)|\dpar(s,n,\ell)| +2\int_{\dpar(s,n,\ell) }
P(\mathbf{G}_{1,\ell}^{-2},z,2)\, d\Leb_2(z) \Bigr]
  \end{multline}
  introducing the Poincar\'{e} series,
  cf. Definition~\ref{defi:27ha,1}.
  By symmetry, we will fix $s=+$ in the final estimate
  of~(\ref{equ:21mp,1}) and show that the quantity tends to $0$ as
  $n\rightarrow\infty$.

  \begin{lem}\label{lem:21mp,1}
    Suppose that $0<n'<n$ and $n-n'$ is even. Then, for any $\ell$,
    \[ \frac{n-n'+2}{2} | \dpar(+,n,\ell) | \leq \int_{\dpar(+,n',\ell)}
    P(\mathbf{G}_{1,\ell}^{-2},z,2)\, d\Leb_2(z) .\]
  \end{lem}
  \begin{proof}
    By the change of variable formula
    \[ |\dpar(+,n,\ell)| =
    \int_{\dpar(+,n,\ell)\setminus\dpar(+,n+2,\ell)}
    P(\mathbf{G}_{1,\ell}^{-2},z,2)\, d\Leb_2(z) .\]
    For the same reason, for $k>0$
       \begin{multline*} \int_{\dpar(+,n-2k,\ell)\setminus\dpar(+,n-2k+2,\ell)}
    P(\mathbf{G}_{1,\ell}^{-2},z,2)\, d\Leb_2(z) \geq\\  \int_{\dpar(+,n,\ell)\setminus\dpar(+,n+2,\ell)}
    P(\mathbf{G}_{1,\ell}^{-2},z,2)\, d\Leb_2(z) \end{multline*}
    and the Lemma~\ref{lem:21mp,1} follows.
  \end{proof}
  \begin{coro}\label{coro:21mp,1}
    For $n\geq 2$,
    \[ (n-1)\bigl|\dpar(+,n,\ell)\bigr| \leq 5  \int_{\dpar(+,\lfloor\frac{n}{2}\rfloor,\ell)}
    P(\mathbf{G}_{1,\ell}^{-2},z,2)\, d\Leb_2(z) .\]
  \end{coro}

  The main estimate is given by the next Lemma.
  \begin{lem}\label{lem:21mp,2}
  \[ \lim_{n\rightarrow\infty} \sup \left\{ \int_{\dpar(+,n,\ell)}
    P(\mathbf{G}_{1,\ell}^{-2},z,2)\, d\Leb_2(z) :\: \ell=2,4,\cdots,\infty\right\} = 0 . \]
  \end{lem}
  \begin{proof}
  Let $\sigma$ be either $2$ or $2-\delta$ for some
  $0<\delta<\frac{2}{3}$. Using Lemma~\ref{lem:21ma,1} fix $N$ to use
  Theorem~\ref{theo:27ha,1} and assert that for all $\ell$
  \begin{equation}\label{equ:30zp,1}
    \int_{\dpar(+,N,\ell)}
  P(\mathbf{G}_{1,\ell}^{-2},z,\sigma)\, d\Leb_2(z) \leq K_1 .
  \end{equation}

  For $\ell\geq\ell_0$ all sets $\dpar(+,N,\ell)$ are contained in a
  compact subset of $\HH_+$. Since $\mathbf{G}^{-2}_{1,\ell}$ for
  $\ell\leq\infty$ is a
  contraction in the Poincar\'{e} metric of $\HH_+$ with the limit
  $x_{+,\ell}$, by taking
  convergent subsequences we get that $\lim_{n\rightarrow\infty} d_n =
  0$, where
  \[ d_n := \inf \left\{
  |D_z\mathbf{G}^{-2n}_{1,\ell}(z)| :\: z\in\dpar(+,N,\ell),\,
  \ell=2,4,\cdots,\infty \right\} .\]

  For $n\geq N$ and of the same parity and every $\ell$, we get
  \[ \int_{\dpar(+,n,\ell)}  P(\mathbf{G}_{1,\ell}^{-2},z,2) =
  \sum_{k\geq\frac{n-N}{2}} \int_{\dpar(+,N,\ell)}
  |D_z\mathbf{G}_{1,\ell}^{-2k}(z)|^2\, d\Leb_2(z)
  .\]
  On the other hand, for $\delta:\:0<\delta<\frac{2}{3}$, cf. estimate~(\ref{equ:30zp,1}),
  \begin{multline*} K_1 \geq\int_{\dpar(+,N,\ell)}
  P(\mathbf{G}_{1,\ell}^{-2},z,2-\delta) \geq\\
  \sum_{k\geq 0}\int_{\dpar(+,N,\ell)}
  |D_z\mathbf{G}_{1,\ell}^{-2k}(z)|^{2-\delta}\, d\Leb_2(z)| \geq\\
  \sup \left\{
  d^{-\delta}_m : m\geq \frac{n-N}{2}\right\} \int_{\dpar(+,N,\ell)}
  |D_z\mathbf{G}_{1,\ell}^{-2k}(z)|^2\, d\Leb_2(z) . \end{multline*}
  Since $d_m\rightarrow 0$, the Lemma~\ref{lem:21mp,2} follows.
    \end{proof}

  \paragraph{Conclusion of the proof of Proposition~\ref{prop:21ma,1}.}
  By formula~(\ref{equ:21mp,1}, Corollary~\ref{coro:21mp,1} and
  Lemma~\ref{lem:21mp,2}
  \[ \lim_{n\rightarrow\infty} \sup\left\{     \int_{\dpar(+,n,\ell) \cup \dpar(-,n,\ell)}
      \mathfrak{E}_{\text{par},\ell}(z)\,d\Leb_2(z) :\:
      \ell=2,4,\cdots,\infty \right\} = 0 .\]

  The are only $2^{n-1}$ ways to compose parabolic branches with
  return time less than $n$. Uniform tightness thus follows.

  \subsection{Outer branches.}
  Mappings $\tpar(\ell)$ already have uniformly bounded distortion by
  Lemma~\ref{lem:3hp,2}, since for any branch and $s$ it is
  possible to choose
  $\mathfrak{m}$ so that its domain is uniformly nested in
  $V^{\mathfrak{m}}_{\ell}(s)$. However, we would like to have a
  uniform expanding Markov structure. Such a structure is suggested by
  Proposition~\ref{prop:10hp,1} since any composition of branches can
  be extended univalently to map onto a slit plane
  $\CC\setminus \left(-\infty,\tau_{\ell}^{-1}]\cup [1,+\infty)
  \right)$.

  Let us choose a finite set $\mathfrak{B}$ of branches $T$
  which contains symbols $(+,\pm,2,0)$ and
  $(-,\pm,1,0)$ which correspond to post-singular and parabolic
  branches discussed before.

  \begin{defi}\label{defi:23mp,1}
  Define the generic {\em hyperbolic induced} mapping $\thyp(\mathfrak{B})$
  as consisting of sequences which
  begin with any symbol not in $\mathfrak{B}$
  and are followed by any, possibly empty, sequence consisting of
  exclusively of symbols from $\mathfrak{B}$.
  \end{defi}

  For any $\ell$, the domain of
  $\thyp\left(\mathfrak{B}\right)_{\ell}$ is the
  subset of $A_{\ell}$ with domains of the branches from $\mathfrak{B}$
  removed.

  Let $\vhyp$ be a bounded Jordan domain with smooth boundary chosen so that
  $x_{0,\infty} \in \vhyp$ and
  \[ \overline{\vhyp} \subset \CC\setminus
  \left(-\infty,\tau_{\ell}^{-1}]\cup [1,+\infty)\right) \]
  for all $\ell\geq\ell(\vhyp)$, where $\ell(\vhyp)<\infty$.

  \begin{theo}\label{theo:23mp,1}
  Fix any domain $\vhyp$ as specified above. Also choose
  a finite set of branches $\mathfrak{B}$ which contains post-singular and
  parabolic branches. Then the
  following properties hold.
  \begin{itemize}
  \item $\thyp(\mathfrak{B})$ is uniformly tight, cf.
    Definition~\ref{defi:3jp,1}.
  \item For every $\ell\geq\ell(\vhyp)$ any composition of branches of
    $\thyp(\ell)$ extends univalently onto
    \[ \CC\setminus\left(
    (-\infty,\tau_{\ell}^{-1}]\cup [1,+\infty) \right) .\]
  \item There exist a compact set $F_{\text{hyp}} \subset \vhyp$, a
    particular choice of $\mathfrak{B}$ and $\ell_0<\infty$  such
    that for every $\ell\geq\ell_0$ and every branch
    $\mathfrak{z}\in\thyp(\mathfrak{B})$, the inclusion
    $\mathfrak{z}_{\ell}^{-1}(\vhyp) \subset F_{\text{hyp}}$ holds, where
    $\mathfrak{z}_{\ell}^{-1}$ should be taken in the sense of the
    univalent extension of $\mathfrak{z}$ postulated by the previous claim.
  \end{itemize}
\end{theo}

  \paragraph{Uniform tightness of exit maps.}
  Recall that for any generic branch $\mathfrak{z}$ the {\em first
    exit map} from $\mathfrak{z}$ consists of sequences which repeat
  the symbol of $\mathfrak{z}$ an arbitrary number of times.

  \begin{lem}\label{lem:26ma,1}
If $\mathfrak{z}$ is not post-singular or parabolic, then the first
exit map from $\mathfrak{z}$ is uniformly tight.
  \end{lem}
  \begin{proof}
  For any $\ell$ let the {\em block} mean the union of domains of
  $\mathfrak{z}_{\ell}$ and the
  adjacent branch of the same side, sign and level and height greater
  by $1$.    Use
  Proposition~\ref{prop:10hp,1} to verify that $\mathfrak{z}_{\ell}$ maps with
  distortion that is bounded independently of $\mathfrak{z}$ and
  $\ell$ onto the block. Indeed, in the Proposition choose
  $\theta_0=\theta_1=\pm\frac{\pi}{2}$ with the sign depending on the
  sign of $\mathfrak{z}$. Then by Proposition~\ref{prop:11jp,1} the
  distance from the block to the slits is uniformly bounded away from
  $0$.

  But then the measure of the set of points which do not exit by
  the $n$-th iterate of $\mathfrak{z}_{\ell}$ shrinks uniformly
  exponentially with $n$ and uniform tightness follows.
  \end{proof}

  \paragraph{Proof of the first claim.}
  One can construct $\thyp(\mathfrak{B})$ by successively inducing on
  branches one by one. That is, we set $T(0) = \tpar$ and then $T(n+1)$ is the
  first exit map from the next branch followed by $T(n)$.
  Each of those maps is uniformly tight by
  Lemma~\ref{lem:26ma,1} and Lemma~\ref{lem:4ja,1} and since the set $\mathfrak{B}$ was assumed finite, that includes $\thyp(\mathfrak{B})$.

  \paragraph{Proof of the second claim.}
  This follows immediately from Proposition~\ref{prop:10hp,1} taken
  with $\theta_0=\pi, \theta_1=0$.

  \paragraph{Proof of the third claim.}
  $\vhyp$ has a finite hyperbolic diameter in $\CC\setminus\left(
  (-\infty,\tau_{\ell}^{-1}] \cup [1,+\infty) \right)$ and therefore
  $\mathfrak{z}^{-1}(\vhyp)$ has bounded hyperbolic diameter in the
  appropriate extension domain $\hat{\Omega}_{\ell}\setminus
  (-\infty,x_{0,\ell}]$ or $\hat{\Omega}_{\ell}\setminus
  [x_{0,\ell},+\infty)$, cf. Lemma~\ref{lem:3hp,2}. Note that $x_{0,\ell}$
    is on the boundary of that domain and hence Euclidean diameters
    of $\mathfrak{z}^{-1}(\vhyp)$ tend to $0$ as a uniform function of
    the distance from $\Dm(\mathfrak{z})_{\ell}$ to $x_{0,\ell}$.

    By Proposition~\ref{prop:11jp,1} for any $r>0$
    the domains of all branches of $T_{\ell}$ except
    for finitely many are contained in $D(x_{0,\infty},r)$ for all
    $\ell\geq \ell(r)$. By what was just observed, the same holds for
    perhaps larger sets $\mathfrak{z}_{\ell}^{-1}(\vhyp)$. We choose
    $r$ so small that $\overline{D(x_{0,\infty},r)} \subset\vhyp$ and
    set     $F_{\text{hyp}} :=
    \overline{D(x_{0,\infty},r)}$.

    As $\mathfrak{B}$ we pick precisely the finite set of branches
    $\mathfrak{z}$ characterized by the condition $\exists
    \ell\geq\ell(r)\;\;\mathfrak{z}_{\ell}^{-1}(\vhyp) \not\subset F$.
    Then $\ell_0 := \max\left(\ell(\vhyp),\ell(r)\right)$.

    \section{Invariant densities}
  \paragraph{Choice of the domain.}
We fix some $\vhyp$ in Theorem~\ref{theo:23mp,1} which implies a
choice of $\mathfrak{B}$. To unclutter notation, we will write $\thyp$
for $\thyp(\mathfrak{B})$ and $\thyp(\ell)$ for the instance of
$\thyp$ for a particular $\ell$.

  \paragraph{The Perron-Frobenius operator.}
  For all  $\ell$ sufficiently large  the Perron-Frobenius operator
  can be defined
  on $L_{1}(\Dm\left(\thyp(\ell)\right),\Leb_2,\RR)$ by
  \[ (\mathfrak{P}_{\ell} g)(u) =\sum_{\mathfrak{z}\in\thyp}
  |D\mathfrak{z}^{-1}_{\ell}(u)|^2 g\left(\mathfrak{z}_{\ell}^{-1}(u)\right)  \]
  where we identified a generic induced map $\thyp(\ell)$ with the set of its
  branches. The term {\em density} will be used for a non-negative function with
  integral $1$.

  \begin{fact}\label{fa:24mp,1}
  The operator $\mathfrak{P}_{\ell}$ is {\em stochastically stable}
  meaning that there is a invariant density $g^{\infty}_{\ell}$ and for
  any other density $g\in L_{1}\bigl(\Dm(\left(\thyp(\ell)\right),\Leb_2,\RR\bigr)$,
  $\lim_{n\rightarrow\infty}\|\mathfrak{P}_{\ell}^ng -
  g_{\ell}^{\infty}\|_1 = 0$ holds. Additionally, if $\gamma\in
  L_{1}\bigl(\Dm\left(\thyp(\ell)\right),\Leb_2,\RR\bigr)$ is a fixed point of $\mathfrak{P}_{\ell}$,
  then $\gamma = c g_{\ell}^{\infty}$, $c\in\RR$.
  \end{fact}

  Our goal will be to show that densities $g_{\ell}^{\infty}$ are
  real-analytic and converge analytically to $g_{\infty}^{\infty}$
  when $\ell\rightarrow\infty$.

  \subsection{The transfer operator.}
  \begin{defi}\label{defi:24mp,1}
  Let $X$ denote the space of complex-valued holomorphic functions
  of two variables defined on $\vhyp \times \vhyp$, continuous to
  the closure, and real on the diagonal:
  \[ \forall z\in\vhyp \; \forall f\in X\; f(z,\overline{z}) \in\RR .\]
  Endow $X$ with the sup-norm.
  \end{defi}

  Then $X$ is a Banach space over $\RR$.

  \begin{defi}\label{defi:24mp,2}
    The {\em transfer operator} ${\cal P}_{\ell} :\: X \rightarrow X$ is defined
    by
    \[ {\cal P}_{\ell} f(z,w) = \sum_{\mathfrak{z}\in\thyp(\ell)} D\mathfrak{z}^{-1}(z) D\mathfrak{z}^{-1} (w) f\left(\mathfrak{z}^{-1}(z),\mathfrak{z}^{-1}(w)\right) \]
where univalent extensions of branches onto $\overline{\vhyp}$ are
used, cf. Theorem~\ref{theo:23mp,1}.
  \end{defi}

  It is not immediately clear that the transfer operator is continuous
  or even well-defined. Observe that, at least formally, when
  $w=\overline{z}$, then
  \begin{multline*} {\cal P}_{\ell} f(z,\overline{z}) = \sum_{\mathfrak{z}\in\thyp(\ell)}
  D\mathfrak{z}^{-1}(z) D\mathfrak{z}^{-1} (\overline{z})
  f\left(\mathfrak{z}^{-1}(z),\mathfrak{z}^{-1}(\overline{z})\right) =\\
  \sum_{\mathfrak{z}\in\thyp(\ell)}
  D\mathfrak{z}^{-1}(z) \overline{D\mathfrak{z}^{-1}} (z)
  f\left(\mathfrak{z}^{-1}(z),\overline{\mathfrak{z}^{-1}}(z)\right) \end{multline*}
  which means that acting on the diagonal $\gamma(z) :=
  f(z,\overline{z})_{|z\in\Dm\left(\thyp(\ell)\right)}$ the transfer operator reduces to the
  Perron-Frobenius operator $\mathfrak{P}_{\ell}\gamma$.

  To establish basic properties of the transfer operator introduce
  {\em branch operators} for a generic branch $\mathfrak{z}$.
  \begin{equation}\label{equ:27mp,1}
      {\cal P}_{\mathfrak{z},\ell} f(z,w) = D\mathfrak{z}^{-1}_{\ell}(z)D\mathfrak{z}^{-1}_{\ell}(w)
  f\left(\mathfrak{z}^{-1}_{\ell}(z),\mathfrak{z}^{-1}_{\ell}(w)\right)
  .
  \end{equation}
  Because of uniformly bounded distortion,
  cf. Theorem~\ref{theo:23mp,1}, we get an estimate
  \begin{equation}\label{equ:27mp,2}
  \| {\cal P}_{\mathfrak{z},\ell} \|\leq K_{\text{norm}}
  \left|\Dm(\mathfrak{z}_{\ell})\right| \end{equation}
  for all $\ell$.

  \begin{lem}\label{lem:24mp,1}
  For every generic branch $\mathfrak{z}$ of $\thyp$ and $\ell \geq
  \ell(\mathfrak{z})$, the branch operator ${\cal P}_{\mathfrak{z},\ell}$ is
  compact.
  \end{lem}
  \begin{proof}
  Let $X_{\mathfrak{z},\ell}$ means the space of functions $f\in X$ restricted to
  the domain of $\mathfrak{z}_{\ell}^{-1}(\overline{\vhyp})$, still with the
  $\sup$-norm. Then, by formula~\ref{equ:27mp,1}, operator ${\cal
    P}_{\mathfrak{z},\ell}$ can be represented as the composition of a
  continuous operator on $X_{\mathfrak{z},\ell}$ and the restriction
  operator from $X$ to $X_{\mathfrak{z},\ell}$. Since
  $\mathfrak{z}_{\ell}^{-1}\left(\overline{\vhyp}\right)$ is a compact subset of
  $\vhyp$ by the last claim of Theorem~\ref{theo:23mp,1},  the restriction
  operator is compact by Cauchy's integral formula and Ascoli-Arzela's
  theorem.
  \end{proof}

  \begin{lem}\label{lem:27mp,2}
    For some $\ell_0<\infty$ and every $\ell\geq\ell_0$ the series in~\ref{defi:24mp,2}
    converges in operator norm and
    ${\cal P}_{\ell}$ is a compact operator. Furthermore,
\[ \sup \left\{ \|{\cal P}_{\ell}^n\| :\: n\geq 0,\, \ell\geq\ell_0 \right\}
< \infty .\]
  \end{lem}
  \begin{proof}
  The series satisfies Cauchy's condition in operator norm by
  estimate~(\ref{equ:27mp,2}). The compactness of the limit then
  follows from Lemma~\ref{lem:24mp,1}.

  As to the additional claim, observe that for any $n>1$ operator
  ${\cal P}_{\ell}^n$ is given by a formula analogous to
  Definition~\ref{defi:24mp,2} except that the summation extends over
  $\mathfrak{z}\in\thyp^n$. Since estimate~(\ref{equ:27mp,2}) is only
  based on bounded distortion, it extends to branches of $\thyp^n$.
  \end{proof}

 \begin{lem}\label{lem:27mp,3}
    \[ \lim_{\ell\rightarrow\infty} \|{\cal P}_{\ell} - {\cal
      P}_{\infty}\| = 0 .\]
 \end{lem}
 \begin{proof}
  Observe first that for every
  generic branch $\mathfrak{z}$ branch operators ${\cal
    P}_{\mathfrak{z},\ell}$ converge in operator norm to ${\cal
    P}_{\mathfrak{z},\infty}$, cf. Proposition~\ref{prop:11jp,1}. By
  the uniform tightness of $\thyp$ and estimate~(\ref{equ:27mp,2}) for
  every $\epsilon>0$ there is a finite set of branches
  $\mathfrak{B}(\epsilon)$ such that for every $\ell$ sufficiently
  large
  \[ \bigl\| \left( \sum_{\mathfrak{z}\in\mathfrak{B}} {\cal
    P}_{\mathfrak{z},\ell} \right) - {\cal P}_{\ell} \bigr\| < \epsilon .\]
  The claim of the lemma now follows by a $3\epsilon$
  argument.
  \end{proof}

  \paragraph{Fixed points of transfer operators.}
  Let us consider the $D_X$ which consists of all $f\in X$ such that
  $\gamma_f(z) := f(z,\overline{z})_{|z\in \Dm\left(\thyp(\ell)\right)}$ is a density when restricted to
  $z\in\Dm\left(\thyp(\ell)\right)$.

  \subparagraph{An identity principle for two complex variables.}
  \begin{fact}\label{fa:30zp,1}
  Suppose that $U$ is a domain in $\CC$ and $F :\:
  U\times U\rightarrow\CC$ is holomorphic. If $F(z,\overline{z})=0$ for
  all $z$ in an open subset of $U$, then $F$ vanishes identically.
  \end{fact}
  This Fact is a particular case of Theorem 7, p. 36, in~\cite{bochner}.

  \begin{prop}\label{prop:27mp,1}
  For every $\ell$ large enough ${\cal P}_{\ell}$ has a unique fixed
  point $f^{\infty}_{\ell} \in D_X$. Additionally,
  $f^{\infty}_{\ell}(z,\overline{z}) = g^{\infty}_{\ell}$,
  cf. Fact~\ref{fa:24mp,1}. Moreover,
  \[ \lim_{\ell\rightarrow\infty} \| f^{\infty}_{\ell} -
  f_{\infty}^{\infty} \|_X = 0 .\]
  \end{prop}

  \paragraph{Proof of Proposition~\ref{prop:27mp,1}.}
  Let $f\in D_X$. By Fact~\ref{fa:24mp,1} functions ${\cal P}_{\ell}^n
  \gamma_f$ converge to $g_{\ell}^{\infty}$ in $L_1\bigl(\Dm\left(\thyp(\ell)\right)\bigr)$.
  By the compactness  and uniform bound of Lemma~\ref{lem:27mp,2},
  sequence ${\cal P}_{\ell}^n f$ is contained in a compact subset of
  $X$. Take any two convergent subsequences ${\cal P}_{\ell}^{n_p}
  f$. Since their limits are the same on the diagonal
  $(z,\overline{z}) :\: z\in\Dm\left(\thyp(\ell)\right)$  they are the same in
  $\vhyp\times\vhyp$ by Fact~\ref{fa:30zp,1}. Hence, the entire sequence ${\cal P}_{\ell}^n f$ converges to a
  fixed point $f_{\ell}^{\infty}$ of ${\cal P}_{\ell}$. Since the transfer operator
  preserves $D_X$, then $f_{\ell}^{\infty}\in D_X$. The same argument based on the
  identity principle shows that the limit is independent of $f$ and
  hence unique for each $\ell$. Since the initial $f$ can be constant
  on $\vhyp$, it also follows that the set $\{ f_{\ell}^{\infty} :\:
  \ell\geq\ell_0\}$ is bounded by some $K_1$.

  It remains to show the convergence of $f_{\ell}^{\infty}$ to
  $f_{\infty}^{\infty}$. Since ${\cal P}_{\infty}$ is compact, the set
  $\{ {\cal P}_{\infty} f_{\ell}^{\infty} :\: \ell\geq\ell_0 \}$ is
  pre-compact in $X$.
  Observe that
  \[ \lim_{\ell\rightarrow\infty} \|{\cal P}_{\infty} f_{\ell}^{\infty} - f_{\ell}^{\infty} \|_X \leq
  \lim_{\ell\rightarrow\infty} \|{\cal P}_{\infty}-{\cal P}_{\ell}\|
  K_1 = 0\]
  by Lemma~\ref{lem:27mp,3}.

  Let $\hat{f}^{\infty}$ be the limit for any convergent subsequence $\ell_p$ of
  ${\cal P}_{\infty} f_{\ell}^{\infty}$. Then, by what has just been
  observed, $\hat{f}^{\infty} = \lim_{p\rightarrow\infty}
  f_{\ell_p}^{\infty}$. Then for any $p$
  \[ {\cal P}_{\infty} \hat{f}^{\infty} - \hat{f}_{\infty} = {\cal
    P}_{\infty} (\hat{f}^{\infty} - f_{\ell_p}^{\infty}) + ({\cal
    P}_{\infty}-{\cal P}_{\ell_p}) f_{\ell_p}^{\infty} +
  (f_{\ell_p}^{\infty} - \hat{f}^{\infty}) .\]
  Since every term on the right-hand side tends to $0$ as
  $p\rightarrow\infty$, $\hat{f}^{\infty}$ is a fixed point of ${\cal
    P}_{\infty}$ and $\hat{f}^{\infty} = f_{\infty}^{\infty}$ by the
  uniqueness of the fixed point.

  Proposition~\ref{prop:27mp,1} has been proved.

  \subsection{Invariant measures for original mappings $T_{\ell}$.}
  The general construction for passing from invariant measures for induced maps
  $\thyp(\ell)$ to measures for $T_{\ell}$ is well known. Let
  $r_{\thyp}(\mathfrak{z})$ denote the return time for branch
  $\mathfrak{z}$ of $\thyp$, i.e. the number of iterates of $T$ which
  compose to $\mathfrak{z}$.

  If $\mu_{\text{hyp},\ell}$ is an invariant measure
  for $\thyp(\ell)$, which is piecewise equal to $T_{\ell}^j$, then
  \[ \mu_{\ell} := \sum_{\mathfrak{z}\in \thyp}
  \sum_{j=0}^{r_{\thyp}(\mathfrak{z})-1}
  \left( T^j_{\ell|\Dm(\mathfrak{z}_{\ell})}\right)_*
    \mu_{\text{hyp},\ell}  \]
  is immediately seen to be invariant under the push-forward by $T_{\ell}$.

    We will work in the spaces
  $L_p := L_p(\CC,\Leb_2,\RR)$.

  \begin{defi}\label{defi:28mp,1}
    Given a set $B$ of branches of $\thyp$ the {\em propagation operator}
    \[ \hat{\mathfrak{P}}_{B,\ell} :\: L_{\infty}
    \rightarrow L_1 \]
    is defined by
    \[ \hat{\mathfrak{P}}_{B,\ell} g(u) = \sum_{\mathfrak{z}\in B}
    \sum_{j=0}^{r_{\thyp}(\mathfrak{z})-1} |DT_{\ell}^{-j}(u)|^2
    (g\cdot\chi_{Dm(\mathfrak{z}_{\ell})})\left(T_{\ell}^{-j}(u)\right)
    .\]
    When $B$ is not specified, it is assumed to be the set of all
    branches of $\thyp$.
  \end{defi}

  The convergence of this sum will addressed later.

  Define {\em simple propagation operators}

  \begin{equation}\label{equ:6na,1} \hat{\mathfrak{P}}_{\mathfrak{z},j,\ell} g(u) :=
   |DT_{\ell}^{-j}(u)|^2
   (g\cdot\chi_{Dm(\mathfrak{z}_{\ell})})\left(T_{\ell}^{-j}(u)\right)
   \end{equation}
   where $\mathfrak{z}$ is any branch of $\thyp$ and
   $0\leq j<r_{\thyp}(\mathfrak{z})$.


  \begin{lem}\label{lem:4np,1}
    For any branch $\mathfrak{z}$ of $\thyp$, $0\leq
    j<r(\mathfrak{z})$ and $\ell$ sufficiently large,
    \[ \lim_{\ell\rightarrow\infty}
    \hat{\mathfrak{P}}_{\mathfrak{z},j,\ell} g_{\ell}^{\infty}  =
    \hat{\mathfrak{P}}_{\mathfrak{z},j,\ell} g_{\infty}^{\infty} \]
    in $L_1$, cf. Fact~\ref{fa:24mp,1}.
  \end{lem}
  \begin{proof}
By Proposition~\ref{prop:27mp,1} densities $g_{\ell}^{\infty}(u)$ extend
to real analytic functions
\[ \hat{g}_{\ell}(u) := f_{\ell}^{\infty}(u,\overline{u}) \] which
converge uniformly on $\vhyp$. Replacing $g_{\ell}^{\infty}$ with
$\hat{g}_{\ell}$ in the claim of the Lemma does not
change its meaning, since the characteristic functions of
$\Dm(\mathfrak{z}_{\ell})$ force the restriction to appropriate domains.
    In connection with formula~(\ref{equ:6na,1}) write
    \begin{multline*} |DT_{\ell}^{-j}|^2
      \left(\hat{g}_{\ell}\cdot\chi_{\Dm(\mathfrak{z}_{\ell})}\right)\circ T^{-j}_{\ell} -
    |DT^{-j}_{\infty}|^2\left(\hat{g}_{\infty}\cdot
    \chi_{\Dm(\mathfrak{z}_{\infty})}\right)\circ T^{-j}_{\infty} =\\ \left(
    |DT^{-j}_{\ell}|^2\hat{g}_{\ell}\circ T^{-j}_{\ell}-|DT^{-j}_{\infty}|^2
    |\hat{g}_{\infty}\circ T^{-j}_{\infty}\right)
    \chi_{T_{\infty}^{-j}\left(\Dm(\mathfrak{z}_{\infty})\right)} +\\
    \left(\chi_{\Dm(\mathfrak{z}_{\ell})}\circ
    T^{-j}_{\ell}-\chi_{\Dm(\mathfrak{z}_{\infty})}\circ T_{\ell}^{-j}\right)
    \hat{g}_{\ell}\circ T^{-j}_{\ell} |DT^{-j}_{\ell}|^2 .
    \end{multline*}

    To see that the first term goes to $0$ in $L_1$, observe
    $\hat{g}_{\ell}, T^{-j}_{\ell} \rightarrow \hat{g}_{\infty},
    T^{-j}_{\infty}$ uniformly on compact subsets of
    $\vhyp$. Next, $\mathfrak{z}_{\ell}$ maps onto
    $\bigcup_{h\notin\mathfrak{B}} \Dm(h_{\ell})$,
    cf. Theorem~\ref{theo:23mp,1}. For any fixed $h$,
     $\chi_{\Dm(h_{\ell}\circ\mathfrak{z}_{\ell})}$ converges to
    $\chi_{\Dm(h_{\infty}\circ\mathfrak{z}_{\infty})}$ in $L_1$ by
    Lemma~\ref{lem:26ma,1}. Then the sums over all $h$ also converge
    by the dominated convergence theorem.

    After changing variables by $T^{-j}_{\ell}$, we estimate the
    $L_1$-norm of the second
    term as follows:
    \begin{multline*} \left\|  \left(\chi_{\Dm(\mathfrak{z}_{\ell})}\circ
    T^{-j}_{\ell}-\chi_{\Dm(\mathfrak{z}_{\infty})}\circ T_{\ell}^{-j}\right)
    \hat{g}_{\ell}\circ T^{-j}_{\ell} |DT^{-j}_{\ell}|^2
    \right\|_1 =\\ \left\| (\chi_{\Dm(\mathfrak{z}_{\ell})} -
    \chi_{Dm(\mathfrak{z}_{\infty})}) \hat{g}_{\ell}\right\|_1 .
    \end{multline*}

    Since $\hat{g}_{\ell}$ are uniformly bounded,
    cf. Proposition~\ref{prop:27mp,1}, the second term goes to $0$ as
    $\ell\rightarrow\infty$ by Fact~\ref{fa:4np,1}.
  \end{proof}

  \begin{theo}\label{theo:28mp,1}
    Mappings $T_{\ell}$ have invariant densities
    $\gamma_{\ell} \in L_1\left(\Leb_2)\right)$, each supported on the
    corresponding $\Dm(T_{\ell})$. The convergence
    $\lim_{\ell\rightarrow\infty} \|\gamma_{\ell}-\gamma_{\infty}\|_1
    = 0$.

    Additionally, for some $R_{\text{analytic}}>0$ and all $\ell$ sufficiently large
    $\gamma_{\ell}$ extend to holomorphic
    functions of two complex variables on
    \[ D(x_{0,\infty},R_{\text{analytic}})\times
    D(x_{0,\infty},R_{\text{analytic}})\] which
    converge uniformly on this set.
  \end{theo}

  \paragraph{Proof of $L_1$ convergence.}
  Densities $\gamma_{\ell}$ are given by $\gamma_{\ell} :=
  \hat{\mathfrak{P}}_{\ell} g_{\ell}^{\infty}$. We will now address
  the convergence of the propagation operator.

  All inverse branches in formula~(\ref{equ:6na,1}) have uniformly bounded
  distortion since they map into the domains of branches of
  $\thyp(\ell)$ which are all contained in $F_{\text{hyp}}$. So, for
  some $K_{\text{propag}}$ independent of $\mathfrak{z},\ell,j$
  \[ \| \hat{\mathfrak{P}}_{\mathfrak{z},j,\ell}  \|_1 \leq
  K_{\text{propag}} \|\chi_{\Dm(\mathfrak{z}_{\ell})}\|_1 .\]
  If we sum this up over all $j :\: 0\leq j<r_{\thyp}(\mathfrak{z})$
  we get a factor $r_{\thyp}(\mathfrak{z})$ and if we further sum up
  over all branches $\mathfrak{z}$ in some set $B$, then since the domains are
  disjoint, we get
  \[ \|\hat{\mathfrak{P}}_{B,\ell}\|_1 \leq K_{\text{propag}}
  \int_{\bigcup_{\mathfrak{z}\in B} \Dm(\mathfrak{z}_{\ell})}
  r_{\thyp,\ell}(u)\, d\Leb_2(u) .\]

  Let $\epsilon>0$ be arbitrary.
  By uniform tightness in Theorem~\ref{theo:23mp,1},
  cf. Definition~\ref{defi:3jp,1}, there is a set of branches
  $B(\epsilon)$ including all but finitely many branches of $\thyp$
  such that
  \begin{equation}\label{equ:6np,1}
    \|\hat{\mathfrak{P}}_{B(\epsilon),\ell}\|_1 \leq
    K_{\text{propag}} \epsilon
    \end{equation} uniformly for all
  $\ell\geq\ell(\epsilon)$.  Then we estimate
  \begin{multline*}
    \limsup_{\ell\rightarrow\infty} \left\| \hat{\mathfrak{P}}_{\ell} g_{\ell}^{\infty} -
  \hat{\mathfrak{P}}_{\infty} g_{\infty}^{\infty}\right\|_1 \leq
  \limsup_{\ell\rightarrow\infty} \bigl\|\sum_{\mathfrak{z}\notin
    B(\epsilon)}\sum_{j=0}^{r(\mathfrak{z})-1}\left(
    \hat{\mathfrak{P}}_{\mathfrak{z},j,\ell} g_{\ell}^{\infty} -
    \hat{\mathfrak{P}}_{\mathfrak{z},j,\infty}
    g_{\infty}^{\infty}\right)\bigr\|_1 +\\
    \limsup_{\ell\rightarrow\infty} \left\|
    \hat{\mathfrak{P}}_{B(\epsilon),\ell} g_{\ell}^{\infty} -
    \hat{\mathfrak{P}}_{\infty} g_{\infty}^{\infty}\right\|_1 \leq 0 +
    K_{\text{propag}}\epsilon\left(\|g_{\ell}^{\infty}\|_{\infty} +
    \|g_{\infty}^{\infty}\|_{\infty}\right)
    \end{multline*}
    where in the final estimate we used Lemma~\ref{lem:4np,1} and
    inequality~(\ref{equ:6np,1}). Since $\epsilon$ was arbitrary and
    $\|g_{\ell}^{\infty}\|_{\infty}$ are uniformly bounded,
    cf. Proposition~\ref{prop:27mp,1}, it follows that
    \[ \lim_{\ell\rightarrow\infty} \left\| \hat{\mathfrak{P}}_{\ell} g_{\ell}^{\infty} -
    \hat{\mathfrak{P}}_{\infty} g_{\infty}^{\infty}\right\|_1 = 0 \]
    which is the first claim of Theorem~\ref{theo:28mp,1}.

  To prove that claim about $R_{\text{analytic}}$ start with the
  observation that $\thyp(\mathfrak{B})$ is the first return map to
  union of the domains of branches not in $\mathfrak{B}$. Then the
  formula of Definition~\ref{defi:28mp,1} implies that
  $\gamma_{\ell}=g_{\ell}^{\infty}$ on the domain of such branches,
  since the  only possibility to get something other than $0$ for
  $u\in \Dm\bigl(\mathfrak{x}_{\ell} :\: \mathfrak{x}\notin\mathfrak{B}\bigr)$ is
  when $\mathfrak{z}=\mathfrak{x}$ and $j=0$. On the other hand, the
  set $\mathfrak{B}$ was finite and disjoint from
  $D(x_{0,\ell},R_{\text{analytic}})$ for some
  $R_{\text{analytic}}>0$. Then the $\gamma_{\ell} =
  g_{\ell}^{\infty}$ continue analytically to $f_{\ell}^{\infty}$ and
  the second claim of Theorem~\ref{theo:28mp,1} follows from
  Proposition~\ref{prop:27mp,1}.

 \subsection{Geometric properties of the boundary of $\Omega_{\ell}$.}
 Recall the arc $\mathfrak{w}_{\ell}$ which joins $x_{\pm,\ell}$ to
 $x_{0,\ell}$ and is invariant under $G_{\ell}$. Define
 $\hat{\HH}_{\pm,\ell} := \HH_{\pm}\setminus
 \mathfrak{w}_{\ell}$. We can also take $\hat{\HH}_{\pm,\infty} =
 \hat{\HH}_{\pm}$.
 Then $\hat{\HH}_{\pm,\ell}$ are swapped by the
 action of the principal inverse branch $\mathbf{G}_{\ell}^{-1}$:
 \[ \mathbf{G}_{\ell}^{-1}(\hat{\HH}_{\sigma,\ell}) \subset
 \hat{\HH}_{-\sigma,\ell} ,\, \sigma=\pm .\]

 \paragraph{Fundamental segments in the border of $\Omega_{\ell}$.}
Recall the point $y_{\ell} := \mathbf{G}_{\ell}^{-1}(\tau_{\ell}
x_{0,\ell})$. The first segment of the boundary of $\Omega_{+,\ell}$
is composed of
two arcs in the form $G_{\pm,\ell}^{-1}[y_{\ell},0)$ where
  $G_{\pm,\ell}^{-1}$ denotes the inverse branch which maps
  $\CC\setminus [0,+\infty)$ into $\HH_{\pm}$ while sending
    $(y_{\ell},0)$ into the border of $\Omega_{+,\ell}$.
Then the rest of the boundary consists images of these two arcs by
$\mathbf{G}_{\ell}^{-n}$ for $n>0$.
\begin{defi}\label{defi:24xp,1}
We will call all arcs in the form
$\mathbf{G}_{\ell}^{-n}\left(G_{\pm,\ell}^{-1}[y_{\ell},0)\right)$, $n=0,1,\cdots$
  {\em fundamental segments} of $\partial\Omega_{\ell}$ of {\em order}
  $n$.
\end{defi}

\begin{lem}\label{lem:31na,1}
There exists finite constants $K_{\text{arc}},\ell_0$ and for every $n\geq
0$ and $\ell_0\leq\ell\leq\infty$ if
 $\mathfrak{u}$ is a fundamental segment
of the boundary of $\Omega_{\ell}$ with endpoints
$u_1$ and $u_2$, then
\[ \diam(\mathfrak{u})   \leq K_{\text{arc}} |u_1 - u_2| \; .\]
\end{lem}

\begin{proof}
  For any $\ell=\infty$ arc of orders $2$ and $3$ are compactly contained in
  $\hat{\HH}_{\infty}$. This persists for large $\ell$ by Proposition~\ref{prop:11jp,1} and Lemma~\ref{lem:17ha,3}. For those arcs the estimate holds.
  Arcs of larger orders are obtained
  by taking inverse branches $\mathbf{G}_{\ell}^{-1}$ with uniformly
  bounded distortion.
\end{proof}

\begin{lem}\label{lem:27va,1}
There are a constant $K_{\ref{lem:27va,1}}$ and an integer
$\ell_{\ref{lem:27va,1}}$ with the following property. Let
$\ell\geq\ell_{\ref{lem:27va,1}}$ and $v$ belong
to a fundamental segment of order $n$ with endpoints $u_1,u_2$
in the boundary of
$\Omega_{\sigma_1,\ell} \cap \HH_{\sigma_2}$, where
$\sigma_1,\sigma_2=\pm$. Then there is $\hat{v} \in
\partial\Omega_{-\sigma_1,\ell}\cap\HH_{\sigma_2}$ which belongs to a
fundamental segment of order $n+1$ and
\[ |\hat{v}-v| \leq K_{\ref{lem:27va,1}} |u_1-u_2| .\]
\end{lem}
\begin{proof}
Without loss of generality $n\geq 2$ and let $v_0 =
G_{\ell}^{n-2}(v)$. By Proposition~\ref{prop:11jp,1} for $\ell$
sufficiently large the fundamental segment of order $2$ which contains
$v_0$ and a point $\hat{v}_0$ which belongs to the
fundamental segment of order $3$ can be enclosed in a disk of
fixed hyperbolic diameter in $\hat{\HH}_{(-1)^n\sigma_2}$. This
configuration is then mapped by $\mathbf{G}_{\ell}^{-n+2}$ with
bounded distortion, which yields the claim of the Lemma.
\end{proof}

\paragraph{Action of $\mathbf{G}_{\ell}^{-2}$ near fixed points.}
\begin{lem}\label{lem:28na,2}
  Choose $\ell$ and consider point $u+x_{0,\ell}$ in the boundary
  of $\Omega_{\ell}$. Then, for every $\epsilon>0$ and
  integer $k$, possibly negative,  there is
  $r(\epsilon,k)>0$ and for every
  $r: 0<r\leq r(\epsilon,k)$ there is
  $\ell(r,\epsilon,k)<\infty$ so that for every $\ell(r,\epsilon,k)\leq\ell\leq\infty$
  the estimate
  \[
  \left| \mathbf{G}_{\ell}^{-2k}(u+x_{0,\ell}) - x_{0,\ell} -
  u(1+ka|u|^2) \right| < \epsilon |u|^3 \]
  holds wherever $r \leq |u| \leq r(\epsilon,k)$,
  where
  \[ a := -\frac{1}{3}SG_{\infty}(x_{0,\infty}) = \frac{1}{6}D^3\mathbf{G}^{-2}_{\infty}(x_{0,\infty})
  > 0 .\]
\end{lem}
\begin{proof}
We will first consider the case of $k=\pm 1$.
We have the expansion
\[ \mathbf{G}_{\ell}^{-2k}(x_{0,\ell}+u) - x_{0,\ell} = u +
ka u^3 + \psi_{\ell,k}(u) + u^4 O_{\ell}(1) \]
where $\left| O_{\ell}(1) \right|$ is bounded for all $0<|u|<r_0$
and $\ell\geq\ell_0$ while $\lim_{\ell\rightarrow\infty}
\psi_{\ell,k}(u) = 0$ for all $0<|u|<r_0$.

Since $u+x_{0,,\infty}$ is in the boundary of $\Omega_{\ell}$, by
Lemma~\ref{lem:18ha,1} for any $\eta>0$ there is $r(\eta)>0$ and if
$r<|u|<r(\delta)$ as well as $\ell\geq\hat{\ell}(\eta,r)$, then
$|\arg u^2-\pi|<\eta$. Consequently,
\begin{equation}\label{equ:30vp,1}
  \left| u^2 + |u|^2\right| < 2\sin\frac{\eta}{2}|u|^2 < \eta .
\end{equation}
when $\eta$ is small enough.

Thus,
\[ \left| \mathbf{G}_{\ell}^{\mp 2}(u+x_{0,\ell}) - x_{0,\ell} -
  u(1 \pm a|u|^2) \right| \leq \eta |u|^3 + |\psi_{\ell,k}(u)| +
  |u|^4 |O_{\ell}(1)| .\]

This leads to fixing $r(\epsilon,k)$ so that for $|u|\leq r(\epsilon,k)$
the last term is less than
$\frac{\epsilon}{4}|u|^3$ and addition to $r(\epsilon,k) \leq r(
\epsilon/4)$ in the previous estimate. Then,
\[ \left| \mathbf{G}_{\ell}^{\mp 2}(u+x_{0,\ell}) - x_{0,\ell} -
  u(1 \pm a|u|^2) \right| \leq \frac{\epsilon}{2} |u|^3 + |\psi_{\ell,k}(u)| .\]

Now as soon as $r$ has been specified in the claim of the Lemma and $|u|>r$,
$|\psi_{\ell,k}(u)|$ can be made less than $\frac{\epsilon}{4}
|u|^3$ by choosing $\ell(r,\epsilon,k)$ suitably large. Additionally, we need
$\ell(r,\epsilon,k) \geq \hat{\ell}(\epsilon/4,r)$ which was needed to
secure estimate~(\ref{equ:30vp,1}).

The general case follows by induction with respect to $k$. To fix attention,
we will describe the inductive step for $k>0$. Let $u_k =
\mathbf{G}_{\ell}^{-2k}(u+x_{0,\infty}) - x_{0,\infty}$. Then
\begin{equation}\label{equ:31vp,1}
  \left| \mathbf{G}_{\ell}^{-2}(u_k-x_{0,\ell}) - x_{0,\ell} -
  u_k(1+a|u_k|^2) \right| < \epsilon_1 |u_k|^3  .
  \end{equation}
  By the inductive step
  \[ \left| u_k - u(1+ka|u|^2) \right| < \epsilon_2 |u|^3 . \]
  In particular, $u_k$ and $u$ differ only by $O(|u|^3)$ and
  $u_k|u_k|^2 = u|u|^2 + O(|u|^5)$.
  Furthermore,
  \begin{multline*}
    \left| u_k(1+a|u_k|^2) - u(1+(k+1)a|u|^2) \right| = \\
    \left| u_k(1+a|u_k|^2) - a u_k |u_k|^2 - ku(1+ka|u|^2) + O(|u|^5) \right|
    \leq \epsilon_2 |u|^3 .
    \end{multline*}
  From formula~(\ref{equ:31vp,1}), we get
  \[  \left| \mathbf{G}_{\ell}^{-2}(u_k+x_{0,\ell}) - x_{0,\ell} - u(1+(k+1)a|u|^2) \right| \leq \epsilon_1 |u_k|^3 + \epsilon_2|u|^2 +O(|u|^5) .\]
  Given $\epsilon>0$, if we take $\epsilon_1=\epsilon_2=\epsilon/3$ and $u$ small enough, we get the claim.
\end{proof}

\begin{lem}\label{lem:2xp,1}
Consider a fundamental segment $\mathfrak{u}$ in the boundary of
$\Omega_{\ell}$ with endpoints $v$ and
$\mathbf{G}^{-2}_{\ell}(v)$. Let $r>0$. There are constants $K_{\ref{lem:2xp,1}}$ and $\ell(r)$ such
that if $\ell\geq\ell(r)$ and $\mathfrak{u}$ intersects
$C(x_{0,\ell},r)$, then $|v-\mathbf{G}_{\ell}^{-2}(v)| < K_{\ref{lem:2xp,1}}r^3$.
\end{lem}
\begin{proof}
  In Lemma~\ref{lem:28na,2} set $\epsilon=a$ and $k=1$. Let
  $\rho:=|v-x_{0,\ell}|$. To use Lemma~\ref{lem:28na,2} we need to
  have $\rho$ suitably small. This is true provided that the order of
  $\mathfrak{u}$ is sufficiently large and $\ell$ is large as well, by
  Proposition~\ref{prop:11jp,1}. The estimate of the Lemma can be
  easily met for finitely many orders based on the same Proposition.

  Then for $\ell\geq\ell(\rho)$ Lemma~\ref{lem:28na,2}  will implies
  that
  \begin{equation}\label{equ:2xp,1}
    |v-\mathbf{G}^{-2}_{\ell}(v)| \leq a\rho^3 .
\end{equation}
    By
  Lemma~\ref{lem:31na,1}
  \[ \rho(1-aK_{\text{arc}}\rho^2) \leq r \leq
  \rho(1+aK_{\text{arc}}\rho^2) .\] Again, for $\rho <
  \sqrt{\frac{1}{2aK_{\text{arc}}}}$, this reduces to
    $\frac{r}{2}\leq\rho\leq 2r$. Hence, we can replace the condition
    $\ell\geq\ell(\rho)$ with $\ell\geq\ell(\frac{r}{2})$ and rewrite
    estimate~(\ref{equ:2xp,1}) as
    \[  |v-\mathbf{G}^{-2}_{\ell}(v)| \leq 8 a r^3 .\]
    \end{proof}

\paragraph{Sections of $\partial\Omega_{\ell}$ by circles.}
Fix $\sigma_1,\sigma_2=\pm$ and $r>0$. Then we will write
\[ X_{\sigma_1,\sigma_2,\ell}(r) := \left\{
v\in\partial\Omega_{\sigma_1,\ell}\cap\HH_{\sigma_2} :\:
|v-x_{0,\ell}|=r \right\} .\]

\begin{lem}\label{lem:31np,2}
There exist $r_{\ref{lem:31np,2}}>0$ and an integer constant
$N_{\ref{lem:31np,2}}<\infty$ with the following property. For every $r :\:
0<r<r_{\ref{lem:31np,2}}$ there is $\ell(r)<\infty$ and if
$\ell(r)\leq\ell\leq\infty$, $\sigma_1,\sigma_2=\pm$, then
$X_{\sigma_1,\sigma_2,\ell}(r)$ is contained in some
$N_{\ref{lem:31np,2}}$ consecutive fundamental segments in the boundary of
  $\Omega_{\sigma_1,\ell}\cap\HH_{\sigma_2}$.
\end{lem}
\begin{proof}
  Let $v_0$ be the point of $X_{\sigma_1,\sigma_2,\ell}(r)$ which
  is furthest from $x_{\sigma_2,\ell}$ in the ordering of the
  arc. Suppose that $k_0$ is an integer chosen so that the fundamental
  segment which contains $\mathbf{G}_{\ell}^{-2k_0}(v_0)$ also intersects
  $X_{\sigma_1,\sigma_2,\ell}(r)$ and contains its point $v_1$ which
  is closest to $x_{\sigma_2,\ell}$ in the ordering of the arc.

  Use Lemma~\ref{lem:28na,2} with
  $u=v+0-x_{0,\infty}$ and $\epsilon:=a$.
  For $r$ sufficiently small and
  $\ell$ large depending on $r,k_0$, we obtain
  \[ |\mathbf{G}_{\ell}^{-2k_0}(v_0)-x_{0,\ell}| \leq
  r\left( 1-(k_0-2)ar^2 \right) . \]

  Then the diameter of the fundamental segment which
  contains $\mathbf{G}_{\ell}^{-2k_0}(v_0)$ is at least
  $(k_0-2)ar^3$. On other hand, the diameter of any
  fundamental segment which intersects $C(x_{0,\ell},r)$ for
  $\ell\geq\ell(r)$ is bounded by $K_{~\ref{lem:2xp,1}} K_{\text{arc}}
  r^3$ in view of
  Lemmas~\ref{lem:2xp,1} and~\ref{lem:31na,1}. Then $k_0 \leq
  a^{-1}K_{~\ref{lem:2xp,1}} K_{\text{arc}}+2$
  and  $N_{\ref{lem:31np,2}}$ is that bound increased by
  $1$.
\end{proof}

Results on sections are summarized by the following Proposition.
\begin{prop}\label{prop:3xp,1}
  There exist $R_{\ref{prop:3xp,1}}>0$, $Q_{\ref{prop:3xp,1}}<\infty$
  and an integer constant
$M_{\ref{prop:3xp,1}}$ with the following property. For every $r :\:
0<r<R_{\ref{prop:3xp,1}}$ there is $\ell(r)<\infty$ and if
$\ell\geq\ell(r)$, then the set $X_{\sigma,\ell}(r) =
X_{+,\sigma,\ell}(r)\cup X_{-,\sigma,\ell}(r)$ is
\begin{itemize}
  \item
    covered by fundamental segments in
    $\partial\Omega_{\ell}\cap\HH_{\sigma}$ with orders that vary by
    no more than $M_{\ref{prop:3xp,1}}$.
\item
    contained in a Euclidean disk of radius $Q_{\ref{prop:3xp,1}} r^3$,
\end{itemize}
\end{prop}
\subparagraph{Proof of Proposition~\ref{prop:3xp,1}.}
Let $v \in X_{-,\sigma,\ell}(r)$ for definiteness belong to a
fundamental segment of order $n$. By Lemma~\ref{lem:27va,1} we find
$\hat{v}$ in a fundamental segment of order $n+1$ in
$\partial\Omega_{+,\ell}$ with
\[ |v-\hat{v}| \leq K_{\ref{lem:2xp,1}} K_{\ref{lem:27va,1}} r^3 \]
  provided that $\ell\geq\ell(r)$ by Lemma~\ref{lem:2xp,1}.

Again, by making $r$ small we can ensure that
$|\hat{v}-x_{0,\ell}|>\frac{r}{2}$. Then we use
Lemma~\ref{lem:28na,1} with $\epsilon=a$. What we get is that when $k>
\frac{K_{\ref{lem:2xp,1}} K_{\ref{lem:27va,1}}}{2a}+1$, then
  $|\mathbf{G}_{\ell}^{k}(\hat{v})-x_{0,\ell}| > r$ while
  $|\mathbf{G}_{\ell}^{-k}(\hat{v})-x_{0,\ell}| < r$ provided $r$ is
  small enough depending on $k$ and $\ell$
  is large enough depending on $k,r$. In any case, a fundamental
  segment in $\partial \Omega_{+,\ell}$ with order between $n-k$ and
  $n+k+2$ intersects $C(x_{0,\ell},r)$. In other words,
  $X_{+,\sigma,\ell}(r)$ intersects a fundamental segment whose order
  differs from $n$ by no more than $k+2$.

  Now the first claim of Proposition~\ref{prop:3xp,1} follows from
  Lemma~\ref{lem:31np,2}.

  The second claim is derived by Lemma~\ref{lem:2xp,1}, since
  $X_{\sigma,\ell}(r)$ can be connected by a bounded number of
  fundamental segments and the interval from $v$ to $\hat{v}$.

  \subsection{K\"{o}nig's coordinate.}
For $\ell<\infty$ point $x_{+,\ell}$ is an attracting point for
$\mathbf{G}_{\ell}^{-2}$ and the basin of attraction contains the
entire $\HH_+$.

\begin{defi}\label{defi:19np,1}
The {\em K\"{o}nig coordinate} $\mathfrak{k}_{\ell,\pm}$ is a univalent
map from $H_{\pm}$ into $\CC$ given by
\[ \mathfrak{k}_{\pm,\ell}(u) = \lim_{n\rightarrow\infty}\bigl[
  \left(DG^2_{\ell}(x_{\pm,\ell})\right)^{n}
  \left(\mathbf{G}_{\ell}^{-2n}(u) - x_{\pm,\ell}\right) \bigr] .\]
\end{defi}

It is worth noting that since $x_{\pm,\ell}$ form a cycle under
$G_{\ell}$, we get
\[ DG^2_{\ell}(x_{\pm,\ell}) = DG_{\ell}(x_{+,\ell})
DG_{\ell}(x_{-,\ell}) = \left| DG_{\ell}(x_{\pm,\ell})\right |^2 \]
since $DG_{\ell}(x_{+,\ell}) = \overline{DG_{\ell}(x_{-,\ell})}$.

We get the functional equation for $u \in \HH_{\pm}$
\[ \mathfrak{k}_{\mp,\ell} \circ \mathbf{G}^{-1}_{\ell}(u) =
\left(DG_{\ell}(x_{\pm,\ell})\right)^{-1}
\mathfrak{k}_{\pm,\ell}(u)\; . \]

Our interest is in
the behavior of K\"{o}nig's coordinate in
$D(x_{0,\infty},R_{\text{analytic}})\setminus\Omega_{\ell}$. Recall
the arc $\mathfrak{w}_{\ell}$ which joins $x_{\pm,\ell}$ to
$x_{0,\ell}$ and is the common boundary component of
$\Omega_{\pm,\ell}$ and invariant under $G_{\ell}$. It is convenient
to restrict the domain of $\mathfrak{k}_{\pm,\ell}$ to
$\HH_{\pm}\setminus\mathfrak{w}_{\ell}$ and then we take the logarithm
$\log \mathfrak{k}_{\pm,\ell}$ which will map into some horizontal
strip of with $2\pi$.

Set
\begin{equation}\label{equ:2qp,1}
  t_{\ell} := \frac{\log
    \left|DG_{\ell}(x_{\pm,\ell})\right|^2}{\log\tau_{\ell}^2} .
\end{equation}
Then functions $\psi_{\ell} = \log H_{\ell},\, t_{\ell}^{-1}\log \mathfrak{k}_{\pm,\ell}$
satisfy the same functional equation
\begin{equation}\label{equ:4qa,1}
  \psi_{\ell} \circ \mathbf{G}_{\ell}^{-2}(u) = \psi_{\ell}(u) -
  \log\tau^2_{\ell} .
\end{equation}

\paragraph{Repelling Fatou coordinate.}
When $\ell=\infty$ the K\"{o}nig coordinate is replaced
with the exponential of the repelling Fatou coordinate denoted by
$\mathfrak{k}_{\pm,\infty}$.  The functional
equation is $\log\mathfrak{k}_{\pm,\infty}\mathbf{G}^{-2}_{\infty} = 1 +
\log \mathfrak{k}_{\pm,\infty}$. In that case we put
\[ \psi_{\infty}:=-\log\tau_{\infty}^2\log\mathfrak{k}_{\pm,\infty} \]
and the equation~(\ref{equ:4qa,1})
will be satisfied. We will speak a {\em generalized} K\"{o}nig
coordinate to include this case.

\paragraph{Estimates of the variation of the generalized K\"{o}nig coordinate.}
We will write $d_{\pm,\ell}$ for the hyperbolic metric of
$\hat{\HH}_{\pm,\ell}$.
\begin{lem}\label{lem:28na,1}

  For every $R>0$ there are $\ell_{\ref{lem:28na,1}}<\infty$ and $K_{\ref{lem:28na,1}}(R)<\infty$ for which the
  following statement holds true for every $\ell\geq\ell_{\ref{lem:28na,1}}$.

  Fix a fundamental segment
  in the boundary
  of $\Omega_{\ell}\cap\HH_{\pm}$ and denote its
  endpoints by $u_1,u_2$. Let $\Delta$ be a disk which contains that
  fundamental segment and whose hyperbolic diameter is less than
  $R\cdot d_{\pm,\ell}(u_1,u_2)$,

  Then, whenever $z_1,z_2\in\Delta$,
   \[ t_{\ell}^{-1}\bigl|
\log\mathfrak{k}_{\pm,\ell}(z_1)-\log\mathfrak{k}_{\pm,\ell}(z_2)\bigr|
< K_{\ref{lem:28na,1}}(R) .\]
\end{lem}
\begin{proof}
  $\ell_{\ref{lem:28na,1}}$ should be chosen so that for every
  $\ell\geq\ell_{\ref{lem:28na,1}}$ the hyperbolic diameter of the
  fundamental arcs of order $2,3$ in the boundary of $\Omega_{\ell}$
  is bounded by some $K$.  Then the same bound holds for all orders,
  since $\mathbf{G}_{\ell}^{-2}$ is a hyperbolic contraction.

  Thus, the hyperbolic diameter of $\Delta$ is bounded by $KR$
  and so the distortion of
  $\log \mathfrak{k}_{\pm,\ell}$ is bounded on $\Delta$ in terms of $R$.
  Let $u_1$ and $u_2$ be the endpoints of the fundamental segment.
  Then,
  \[ \bigl|\log\mathfrak{k}_{\pm,\ell}\left(z_1\right) -
  \log\mathfrak{k}_{\pm,\ell}\left(z_2\right) \bigr| \leq K_1(R)
  \left|\log\mathfrak{k}_{\pm,\ell}(u_1)-\log\mathfrak{k}_{\pm,\ell}(u_2)\right|
  = K_1(R)t_{\ell}\log\tau_{\ell}^2 \]
where the last equality follows from formula~(\ref{equ:2qp,1}).
\end{proof}

\paragraph{The filler map.}
$\mathfrak{k}_{\pm,\ell}$ is defined up to a multiplicative
constant. Let us choose it so that
$t_{\ell}^{-1}\log\mathfrak{k}_{\pm,\ell}$ and $\log H_{\ell}$ are
equal at an endpoint of the arc of order $2$ in the boundary of
$\Omega_{\ell}$. Then we get:

\begin{coro}\label{coro:8xp,1}
  There exists $K_{\ref{coro:8xp,1}}$ such that every
    $\ell\geq\ell_{\ref{lem:28na,1}}$
      at every point $u$ in the boundary of $\Omega_{\ell}$,
      \[ \left| t_{\ell}^{-1}\log\mathfrak{k}_{\pm,\ell}(u) - \log
      H_{\ell}(u)\right| \leq K_{\ref{coro:8xp,1}} .\]
\end{coro}

This follows from Lemma~\ref{lem:28na,1} since its enough to establish
the estimate $u$ in arcs of order $2$ and $3$.

\begin{lem}\label{lem:31np,3}
  There are $R_{\ref{lem:31np,3}}>0$ and $K_{\ref{lem:31np,3}}$ and for every
  $r :\: 0<r\leq R_{\ref{lem:31np,3}}$
one can choose $\ell_{\ref{lem:31np,3}}(r)<\infty$ with the following property.

For every $\ell :\: \ell_{\ref{lem:31np,3}}(r) \leq\ell\leq\infty$ there is an arc
which contains the set $X_{\sigma,\ell}(r)$, $\sigma=\pm$,
cf. Proposition~\ref{prop:3xp,1}, so that for every two points on this
arc $t_{\ell}^{-1}\log\mathfrak{k}_{\sigma,\ell}$
differ by no more than $K_{\ref{lem:31np,3}}$.
\end{lem}
\begin{proof}
By Proposition~\ref{prop:3xp,1} the convex hull of
$X_{\sigma,\ell}(r)$ on the circle $C(x_{0,\ell},r)$ is contained in a
Euclidean disk $\hat{\Delta}$  of radius $2Q_{\ref{prop:3xp,1}} r^3$. On the other
hand, if $u\in X_{\sigma,\ell}(r)$ then
$|u-\mathbf{G}_{\ell}^{-2}(u)|\geq \frac{a}{2} r^3$ provided that $r$ is
sufficiently small and $\ell$ large enough depending on $r$ - just
refer to Lemma~\ref{lem:28na,2} with $\epsilon=\frac{a}{2}$. Additionally, the
distance from $u$ to $\RR$ is bigger than $r/2$ under the same
conditions on $r,\ell$.

Hence, $t_{\ell}^{-1}\log\mathfrak{k}_{\sigma,\ell}$ maps $\hat{\Delta}$
with uniformly bounded distortion and the claim follows as in
Lemma~\ref{lem:28na,1}.

\end{proof}

  \subsection{The drift integral.}
  Let us recall the fundamental annulus $A_{\ell}$, cf. Definition~\ref{defi:3hp,2}.
  The drift integral is
  \[ \vartheta(\ell) = -\frac{1}{\log\tau_{\ell}} \Re \int_{A_{\ell}}
  \log\frac{H_{\ell}(u)}{u} \gamma_{\ell}(u)\, d\Leb_2(u) ,\]
  cf. in~\cite{leswi:limit} Lemma 3.2, Definition~\ref{defi:3hp,2}.

  The function $\log\frac{H_{\ell}(u)}{u}$ is bounded in $A_{\ell}$
  except in neighborhoods of $x_{0,\ell}$. Its growth there can be
  controlled by the functional equation $H_{\ell}\circ G_{\ell} =
  \tau^{-2}_{\ell} H_{\ell}$. This shows that for $\ell$ finite the
  magnitude of that function exceeds $M$ on sets which are
  exponentially small in $M$ and hence the drift integral is
  well-defined. For $\ell=\infty$ this argument breaks down and the
  drift integral has undefined value, see in~\cite{leswi:measure},
  Proposition 3.5. Our goal is to prove the following result.

  \begin{theo}\label{theo:7np,1}
  There exists a finite limit
\[ \lim_{\ell\rightarrow\infty} \vartheta(\ell) = -\frac{1}{\log\tau_{\infty}}\lim_{r\rightarrow
  0^+} \Re \int_{A_{\infty}\setminus D(x_{0,\infty},r)}
\log\frac{H_{\infty}(u)}{u}\gamma_{\infty}(u)\, d\Leb_2(u) .\]
  \end{theo}

  For all $\ell\geq \ell(r)$ the complement of $D(x_{0,\infty},r)$
  meets the domains of only finitely many branches of $T_{\ell}$. For
  that reason $\log\frac{H_{\ell}(u)}{u} \cdot \chi_{A_{\ell}\setminus
    D(x_{0,\infty},r)}$ are    bounded    uniformly
    with respect to $\ell$ and converge to
    $\log\frac{H_{\infty}(u)}{u}$ pointwise. By the Lebesgue
    dominated convergence theorem and Theorem~\ref{theo:28mp,1}
\begin{multline}\label{equ:17np,1} \lim_{\ell\rightarrow\infty} -\frac{1}{\log\tau_{\ell}} \Re
\int_{A_{\ell}\setminus D(x_{0,\infty},r)}
\log\frac{H_{\ell}(u)}{u}\gamma_{\ell}(u)\, d\Leb_2(u) =\\
-\frac{1}{\log\tau_{\infty}} \int_{A_{\infty}\setminus D(x_{0,\infty},r)}
\log\frac{H_{\infty}(u)}{u}\gamma_{\infty}(u)\, d\Leb_2(u) .\end{multline}

\paragraph{Stokes' formula.}
We can assume
$r<R_{\text{analytic}}$, cf. Theorem~\ref{theo:28mp,1}, and hence all
$\gamma_{\ell}$ are analytic.
\begin{defi}\label{defi:10np,1}
Let $\psi_{\ell}(u)$ satisfy
\begin{equation}\label{equ:7np,2}
  \partial_{\overline{u}} \psi_{\ell}(u) = \gamma_{\ell}(u),
\end{equation}
normalized so that the linear part at $x_{0,\infty}$ is
$\gamma_{\ell}(x_{0,\infty})\overline{u-x_{0,\infty}}$.
\end{defi}
  Stokes' formula is $\int_D F(u)\gamma_{\ell}(u)\, d\Leb_2(u)
= \frac{1}{2i} \int_{\partial D} F(u)\psi_{\ell}(u)\, du$ for $F$ holomorphic
in $D$ and continuous to the closure.

\begin{defi}\label{defi:10xp,1}
  For every $\ell$
  including $\infty$ define $\Phi_{\ell}$ on some fixed
  neighborhood of $x_{0,\infty}$ by
  \[ \Phi_{\ell}(z) := \left\{ \begin{array}{ccc} \log \frac{H_{\ell}(z)}{z}
    &\text{if}& z\in\Omega_{\ell}\\ t_{\ell}^{-1}
    \log\mathfrak{k}_{\pm,\ell}(z) - \log z& \text{if} & z\notin
    \Omega_{\ell} \end{array}\right. \]
\end{defi}

Then, let us also define

\begin{multline}\label{equ:10xp,1}
\Theta_{\ell}(r) := \Re \left[ \frac{\iota}{2} \left(
  \int_{C(x_{0,\infty},r)} \Phi_{\ell}(u)\psi_{\ell}(u)\, du \right.\right.+\\
  \left.\left.\int_{\partial\Omega_{\ell}\cap D(x_{0,\ell},r)} \left(\log\frac{H_{\ell}(u)}{u} -
  \Phi_{\ell}(u)\right)\psi_{\ell}(u)\, du \right) \right] .
\end{multline}

For $\ell<\infty$ we claim that
\begin{equation}\label{equ:11xp,1}
  \Theta_{\ell}(r) = -\int_{D(x_{0,\infty},r)} \Re\Phi_{\ell}(u)\gamma_{\ell}(u)\,d\Leb_2(u) .
\end{equation}
Observe first that the singularities of $\Phi_{\ell}$ and
$\log H_{\ell}$ at $x_{\sigma,\ell}, \sigma=+,-,0$ are
logarithmic and therefore integrable. Then take
into account that $\Phi_{\ell}$ is discontinuous and hence the Stokes'
formula has to
be used separately on $\Omega_{\ell}\cap D(x_{0,\infty},r)$ and
$D(x_{0,\infty},r)\setminus\Omega_{\ell}$. The boundaries of those
sets can be complicated, but add up to $C(x_{0,\infty},r)$ and
  subtract along $\partial\Omega_{\ell} \cap D(x_{0,\infty},r)$ which
  corresponds to the second term in formula~(\ref{equ:10xp,1}).

For $\ell=\infty$ the convergence of $\Theta_{\infty}$ is not clear
and  will be shown later. Assuming it holds, in all cases including
$\ell=\infty$, we get for $0<\rho<r$
\begin{equation}\label{equ:11xp,2}
  \Theta_{\ell}(r)-\Theta_{\ell}(\rho) = \int_{\{ u :\: \rho<|u-x_{0,\infty}|<r\}}\Re\Phi_{\ell}(u)\gamma_{\ell}(u)\,d\Leb_2(u) .
\end{equation}

\begin{prop}\label{prop:10xp,1}
For every $\epsilon>0$ there is $r(\epsilon)>0$
and for every $0<r\leq r(\epsilon)$ there is $\ell_{\ref{prop:10xp,1}}(r)<\infty$ such that
\[ \forall \ell\;\;
\ell_{\ref{prop:10xp,1}}(r)\leq\ell\leq\infty\implies
\left| \Theta_{\ell}\left(r\right)\right| < \epsilon .\]
This includes the claim that $\Theta_{\infty}(r)$ is convergent.
\end{prop}
The proof of this Proposition will require some preparatory estimates.

\paragraph{Estimates on circles.}
\begin{lem}\label{lem:10na,1}
For every $r>0$ there is $\ell(r)<\infty$ so that for all
$\ell(r)\leq\ell\leq\infty$ and $u :\: |u-x_{0,\infty}|=r$
\[ \left|\psi_{\ell}(u) - \gamma_{\ell}(x_{0,\infty}) \overline{u-x_{0,\infty}}\right| \leq O_{\psi}(r^2) \]
where $O_{\psi}(r)$ is independent of $\ell$ and $\limsup_{r\rightarrow 0^+} r^{-2} O_{\psi}(r^2) < \infty$.
\end{lem}
\begin{proof}
  Change variables to $z := u-x_{0,\infty}$.
  By Definition~\ref{defi:10np,1}
  $\psi_{\ell}(z) = \gamma_{\ell}(x_{0,\infty})\overline{z} +
  \psi_{1,\ell}(z)$
  where the linear part of $\psi_{1,\ell}$ vanishes at $z=0$.
  By the analytic convergence claim of Theorem~\ref{theo:28mp,1}, it
  means that $|\psi_{1,\ell}(z)| \leq K_1 |z|^2$ for all $\ell$
  sufficiently large and $z$ in a fixed neighborhood of $0$.
  Hence,
  $|\psi_{\ell}(z) - \gamma_{\ell}(x_{0,\infty}) \overline{z}| \leq K_1 |z|^2$.
\end{proof}

\begin{lem}\label{lem:11na,1}
  There exist a function $o_{\text{Fatou}}(r^{-2}) :\: \lim_{r\rightarrow 0^+} r^2 o_{\text{Fatou}}(r) = 0$ and a positive constant $C_{\text{Fatou}}$ such that
  \begin{multline*} \forall\epsilon>0\; \exists r(\epsilon)>0\; \forall 0<r<r(\epsilon) \exists \ell(r) < \infty\; \forall \ell(r)\leq\ell\leq\infty\\ \log H_{\ell}(u) = -\frac{C_{\text{Fatou}}}{(u-x_{0,\infty})^2} + o_{\text{Fatou}}(r^{-2})
  \end{multline*}
  for $u :\: |u-x_{0,\infty}|=r,\, -\pi +\epsilon<\arg
  (u-x_{0,\infty})^2<\pi-\epsilon$. In particular, it holds for  $u$ in $\Omega_{\ell}$.
\end{lem}
\begin{proof}
  For $\ell=\infty$ recall Fact~\ref{fa:18ha,1} by which the arc of values of $u$ in the claim of the Lemma is a compact subset of $\Omega_{\infty}$. Then the claim of the Lemma
  follows from the form of the Fatou coordinate, with $C_{\text{Fatou}} = -\frac{3\log \tau_\infty^2}{D^3 G_{\infty}(x_{0,\ell})}$.  For $\ell$ finite, write
  \[ \log H_{\ell}(u) = \log H_{\infty} \left(\phi_{\pm,\infty}^{-1}\circ\phi_{\pm,\ell}(u)\right) .\]
  When $r$ has been fixed, the composition in parentheses goes to the identity uniformly on a
  neighborhood of the arc  $u :\: |u-x_{0,\infty}|=r,\, -\pi +\epsilon<\arg (u-x_{0,\infty})^2<\pi-\epsilon$ by Proposition~\ref{prop:11jp,1}.
  Hence, by choosing $\ell(r)$ large enough we can make $|H_{\infty}(u)-H_{\ell}(u)|$ smaller than some $o(r^{-2})$.
\end{proof}

\begin{lem}\label{lem:11na,2}
For some $K_{\text{arc}}<\infty$ and every $r>0$ there is
$\ell(r)<\infty$ such that for all $\ell :\:
\ell(r)\leq\ell\leq\infty$ and $u\in \Omega_{\ell} \cap
C(x_{0,\infty},r)$ the estimate $|\log H_{\ell}(u)| \leq
K_{\text{arc}} r^{-2}$ holds.
\end{lem}
\begin{proof}
Under $\phi^{-1}_{\pm,\infty}$ vertical lines are mapped to arcs which
tend to $x_{0,\infty}$ with tangents at angle $\pi/4$ with the real
line. Hence, for $L_1$ large enough and positive, the preimage of
$L_1+\iota\RR$ is in the domain of the repelling Fatou coordinate of
$G_{\infty}$. Its image by the repelling Fatou coordinate is contained
in some right half-plane $\Re z > L_2, L_2>0$. For $n>0$ the image of
$L_1+n\log\tau_{\infty}^2$ is contained in $\Re z > L_2 + n$. From the
asymptotics of the repelling Fatou coordinate, the preimage of
$L_1+n\log\tau_{\infty}$ by $H_{\infty}$ is contained in $D(x_{0,\infty},Kn^{-2})$.
Hence the desired estimate for $\ell=\infty$ follows on the set
$\phi_{\pm,\infty}^{-1}\left(\{ u :\: \Re u>
L_1+\log\tau_{\infty}^2\}\right)$. For $\ell$ large enough it is then
derived from Proposition~\ref{prop:11jp,1}. On the other hand,
when $\Re u$ is bounded, for $\ell$ sufficiently large the preimage
by $\phi_{\pm,\ell}$ is contained in the  wedge $|\arg
(u-x_{0,\infty})^2| < \frac{3}{4}\pi$ and Lemma~\ref{lem:11na,1}
applies with a stronger claim.
\end{proof}

\begin{lem}\label{lem:10xp,1}
For some $K_{\ref{lem:10xp,1}}<\infty$ and every $r>0$ there is
$\ell_{\ref{lem:10xp,1}}(r)<\infty$ such that for all $\ell :\:
\ell_{\ref{lem:10xp,1}}(r)\leq\ell\leq\infty$ and
$u\in C(x_{0,\infty},r)$ the estimate $|\log \Phi_{\ell}(u)| \leq
K_{\ref{lem:10xp,1}} r^{-2}$ holds.
\end{lem}
\begin{proof}
The term $\log u$ from Definition~\ref{lem:10xp,1} is bounded and can
be ignored. Now in view of Lemma~\ref{lem:11na,2} the
estimate needs to be established for $u$ outside of
$\Omega_{\ell}$. But then the difference between
$t_{\ell}^{-1}\mathfrak{k}_{\pm,\ell}(u)$ and $\log H_{\ell}(u)$ at a
point of $X_{\pm,\ell}(r)$, cf. Proposition~\ref{prop:3xp,1}, is
uniformly bounded by Lemma~\ref{lem:31np,3} and
Corollary~\ref{coro:8xp,1}.
\end{proof}

\begin{lem}\label{lem:11np,1}
In the setting of Proposition~\ref{prop:10xp,1}, for every
$\epsilon>0$, $0<r\leq r(\epsilon)$ and $\ell\geq\ell(r)$

\[ \left| \Re \left[ \frac{\iota}{2}\int_{C(x_{0,\infty},r)}
  \Phi_{\ell}(u)\psi_{\ell}(u)\,du \right] \right| <
  \frac{\epsilon}{2} .\]
\end{lem}
\begin{proof}
We pick an $\eta>0$ having in mind the statement
of Lemma~\ref{lem:11na,1} and then split the arc
$C(x_{0,\infty},r)$ into the sum of arcs $C_{\pm}$ which
are contained in the sector $u :\: |u-x_{0,\infty}|=r,\, -\pi
+\epsilon<\arg (u-x_{0,\infty})^2<\pi-\epsilon$ and in $\Omega_{\ell}$
and
$c_{\pm}$ which the rest.
The angular measure
of $c_{\pm}$ does not exceed
$\eta$. If we take into account that
$|\psi_{\ell}(u)|\leq K_1|u-x_{0,\infty}|$ by
Lemma~\ref{lem:10na,1} and combine with the estimate of
Lemma~\ref{lem:10xp,1},  both holding when $\ell\geq\ell(r)$, then for such $\ell$
\begin{equation}\label{equ:15np,3} \bigl|
  \int_{c_{\pm}} \Phi_{\ell}(u) \psi_{\ell}(u)\, du \bigr| \leq K_1 \eta .
\end{equation}
We want to have
$K_1 \eta(\epsilon) = \frac{\epsilon}{6}$
which sets a value $\eta(\epsilon)$.

Now we pass to estimating the integral along
$C_{\pm}$. We will rely on Lemma~\ref{lem:11na,1}
which requires $r < r(\eta(\epsilon)) := r(\epsilon)$.

Then, by Lemma~\ref{lem:11na,1},
\begin{multline}\label{equ:15np,1}
  \bigl| \Re \left[ \iota\int_{C_{\pm}} \log H_{\ell}(u) \psi_{\ell}(u)\, du \right] \bigr| = \\\bigl| \Re \left[ \iota\int_{C_{\pm}}
  \gamma_{\ell}(x_{0,\infty})C_{\text{Fatou}} \frac{\overline{z}}{z^2} \, dz \right] \bigr| + C_{\text{Fatou}} r^{-1} O_{\psi}(r^2) + K_2 r^2 o_{\text{Fatou}}(r^{-2}) .\end{multline}
The residual terms in estimate~(\ref{equ:15np,1}) tend to $0$ as $r\rightarrow 0^+$ and by $r(\epsilon)$ sufficiently small, we can ensure that they add up to less than $\frac{\epsilon}{3}$. The main term is evaluated directly
\begin{equation}\label{equ:15np,2}
  \Re \left[ \iota \int_{C_{\pm}} \frac{\overline{z}}{z^2} \, dz \right]
  = 2\Re \left[ \frac{1}{2\iota} \exp(-2\iota\theta) |_{\theta_1}^{\theta_2} \right]
\end{equation}
where $z=r\exp(\iota\theta)$. $\theta_1$ and $\theta_2$
are in the form
$\pm\left(\frac{\pi}{2} - \frac{\eta}{2}\right)$.  Inserting $\pm\frac{\pi}{2}$ for
$\theta_1,\theta_2$ results in a purely real difference and hence zero contribution
to the real part of the main integral. What remains has absolute value
bounded by $\eta$. So, this time by possibly decreasing
$\eta$ we get less than $\frac{\epsilon}{6}$. This, together
with estimates~(\ref{equ:15np,3},\ref{equ:15np,1}), yields the claim of
the Lemma.
\end{proof}

\paragraph{Length of the boundary arcs.}
Let us write $w(\sigma,s,\ell)$ for
$\partial{\Omega}_{\sigma,\ell}\cap \HH_{s}$ where
$\sigma,s$ can be any combination of $+,-$. For $r>|x_{+,\ell}-x_{0,\ell}|$ we will write
$w_r(\sigma,s,\ell)$ for the smallest connected subarc of
$w(\sigma,s,\ell)$ which touches $x_{s,\ell}$ and contains
$w(\sigma,s,\ell)\cap D(x_{0,\ell},r)$.

\begin{lem}\label{lem:3qp,1}
For every $\varepsilon>0$ there exist $\ell(\varepsilon)<\infty$ and $r(\varepsilon)>0$
so that for every $\ell\geq\ell(\varepsilon)$ and $\sigma,s=\pm$,
we get $r(\varepsilon)>|x_{\pm,\ell}-x_{0,\ell}|$ and
the Euclidean length $|w_{r(\varepsilon)}(\sigma,s,\ell)|<\varepsilon$.
\end{lem}
\begin{proof}
We start by observing that the length of the basic arc
$G_{\pm,\ell}^{-1}[y_{\ell},0)$ is uniformly bounded for all $\ell$
  sufficiently large. That arc is the preimage under
  $\phi_{\sigma,\ell}$ of the horizontal ray $x+\iota\pi :\:
  -\infty<x<\log |y_{\ell}|$. For $\ell=\infty$ it is an analytic arc
  of finite length. As $\ell\rightarrow\infty$
  $\phi^{-1}_{\sigma,\ell}$ converge uniformly to
  $\phi^{-1}_{\sigma,\infty}$ together with the derivatives, by Cauchy
  estimates.

  Then $w(\sigma,s,\ell)$ is formed by taking images under the inverse
  map $\mathbf{G}_{\ell}^{-1}$. These mappings all have uniformly
  bounded distortion for $\ell$ large enough and hence we can estimate
  the length by taking the sum of absolute values of the derivatives
  $D_z\mathbf{G}_{\ell}^{-n}\left(G_{\pm,\ell}^{-1}(y_{\ell})\right)$.

  The requisite estimates are provided by Lemma~\ref{lem:30hp,2}.
  Point $z$ in the Lemma will be chosen as
  \begin{equation}\label{equ:2fa,1}
    z(\sigma,s,\ell):=\mathbf{G}_{\ell}^{-n(\sigma,s)}\left(G_{s',\ell}^{-1}(y_{\ell})\right)
  \end{equation}
  where $s'=\pm$ and is equal to $s$ if and only if $n(\sigma,s)$ is
  even. This will do for $n(\sigma,s)$ and $\ell$ large enough, since $z$ needs
  to be close enough to $x_{s,\ell}$ and then the condition on
  the argument of $z-x_{s,\ell}$ is also satisfied by
  Lemma~\ref{lem:18ha,1}.

  Then Lemma~\ref{lem:30hp,2} specifies $k(z(\sigma,s,\ell),\ell)$
  which will be written as $k(\sigma,s,\ell)$.
  First look at the estimate for
  $\bigl|D_z\mathbf{G}_{\ell}^{-1}\left(z(\sigma,s,\ell)\right)\bigr|$
  for $k\geq k(\sigma,s,\ell)$. Recall that
  $\lim_{\ell\rightarrow\infty} \rho_{\ell} = 0$. Take
  $\eta=\frac{1}{4}$ while $r$ in Lemma~\ref{lem:30hp,2} can be fixed
  since $z(\sigma,s,\ell)$ is given by formula~(\ref{equ:2fa,1}).
  For $\rho_{\ell}$ small enough, $(1+\rho_{\ell})^{-\frac{1}{8}} <
  1-\frac{\rho_{\ell}}{7}$             and hence the sum of those derivatives is
  bounded by $9L\sqrt{\rho_{\ell}}$ which can be made
  arbitrarily small by taking $\ell$ large enough.

  For $\hat{k}<k(\sigma,s,\ell)$ the sum of the derivatives of iterates between $\hat{k}$ and
  $k(\sigma,s,\ell)$ is bounded by $\frac{L'}{\sqrt{\hat{k}}}$. It
  remains to show that as $r\rightarrow 0$ in the statement of the
  present Lemma,  $\mathbf{G}_{\ell}^{-2k}z(\sigma,s,\ell) \subset
  D(x_{0,\ell},r)$ implies $k\geq \hat{k}(r)$ and $\hat{k}(r)$ can be
  made as large as needed by making $r$ small. This is indeed so,
  since the hyperbolic distance between $z(\sigma,s,\ell)$ and
  $\mathbf{G}_{\ell}^{-2}\left(z(\sigma,s,\ell)\right)$ is fixed and
  then shrunk by iterates. Finally, the condition
  $r(\varepsilon)>|x_{0,\infty}-x_{\pm,\ell}|$ can be satisfied by
  again specifying $\ell(\varepsilon)$ sufficiently large.
 \end{proof}

\paragraph{Proof of Proposition~\ref{prop:10xp,1}.}
In view of Lemma~\ref{lem:11np,1} it remains to estimate
\[ \int_{\partial\Omega_{\ell}\cap D(x_{0,\infty},r)} \left( \log\frac{H_{\ell}(u)}{u} -
\Phi_{\ell}(u)\right)\, du .\]
The integrand is uniformly bounded for
all $\ell$ sufficiently large by Corollary~\ref{coro:8xp,1}.
The length of $\partial\Omega_{\ell}\cap D(x_{0,\infty},r)$ can be
made arbitrarily small for all $\ell$ large enough including $\infty$
by Lemma~\ref{lem:3qp,1} by making $r$ small.

Proposition~\ref{prop:10xp,1} has been established.

\paragraph{Integral of $\Phi_{\ell}$ outside of $A_{\ell}$.}
\begin{prop}\label{prop:11xp,1}
For every $\epsilon>0$ there is $r(\epsilon)>0$
and for every $0<r\leq r(\epsilon)$ there is $\ell_{\ref{prop:11xp,1}}(r)<\infty$ such that
\[ \forall \ell\;\;
\ell_{\ref{prop:10xp,1}}(r)\leq\ell\leq\infty\implies
\left| \int_{D(x_{0,\infty},r)\setminus A_{\ell}}
\left|\Re\Phi_{\ell}(u)\right|\gamma_{\ell}(u)\,d\Leb_2(u)\right| < \epsilon .\]
\end{prop}

\paragraph{Proposition~\ref{prop:11xp,1} for the complement of
  $\Omega_{\ell}$.}
Let $W_{\pm,\ell} : = \HH_{\pm} \setminus \overline{\Omega}_{\ell}$.
For $r$ small and $\ell$ large given $r$, $D(x_{0,\infty},r)\setminus A_{\ell}$
contains two sets
\[ W_{\pm,\ell}(r) := D(x_{\infty,0},r) \cap W_{\pm,\ell} .\]
  We
  will prove the estimate of Proposition~\ref{prop:11xp,1} first for
  these sets.

Let us consider the case of $\ell=\infty$ first. Under the repelling
Fatou coordinate $W_{\pm,\infty}$ is a strip of bounded horizontal width.
Thus the measure of $W_k = \left\{ u\in W_{\pm,\infty}(r) :\: k-1 \leq
|t_{\infty}^{-1}\mathfrak{k}_{\pm,\infty}(u)| \leq k\right\}$ is $O(k^{-3})$. Thus,
\[ \left|\int_{W_{\pm,\infty}(r)} \Phi_{\infty}(u)\gamma_{\infty}(u)\,d\Leb_2(u)\right| \leq K \sum_{k\geq k(r)} k^{-2} \]
where $k(r)$ is the smallest $k$ for which $W_k$ is non-empty. Since
$\lim_{r\rightarrow 0^+} k(r) = \infty$ this can be made less than $\epsilon/2$ by
taking $r(\epsilon)$ small.

For $\ell$ finite we observe first that since $\log H_{\ell}$ is real
on $\partial{\Omega}_{\ell}$, then by Corollary~\ref{coro:8xp,1}, the
imaginary part of $t_{\ell}^{-1} \log\mathfrak{k}_{\pm,\ell}$ is
bounded on $\partial\Omega_{\ell}$, as well as on the arc of
$C(x_{0,\infty},r)$ which joins the components of
$\partial\Omega_{\ell}$ by Lemma~\ref{lem:31np,3}. By the maximum
principle for harmonic functions the imaginary part is thus bounded on
$W_{\pm,\ell}(r)$ for $\ell$ large enough depending on $r$.

The real part, on the other hand, by the functional equation is just
the exit time from $W_{\pm,\ell}(r)$ under $G_{\ell}$ up to
constants.
In the notations of
section~\ref{sec:4qa,1},
\[ \left| \Re t_{\ell}^{-1}\log\mathfrak{k}_{\pm,\ell}(u) \right| \leq K_2
\mathfrak{E}_{\text{par},\ell}(\tau_{\ell}^{-1}u) + K_3 .\]
The integral of this over $u \in W_{\pm,\ell}(r) :\:
|t_{\ell}^{-1}\Re\mathfrak{k}_{\pm,\ell}(u)| \geq k(r)$ tends to $0$
with $k(r)\rightarrow\infty$ by
Proposition~\ref{prop:21ma,1} and since $\lim_{r\rightarrow 0^+} k(r)
= \infty$ as in the case of $\ell=\infty$, the proof is finished.

\paragraph{Proposition~\ref{prop:11xp,1} for $\tau_{\ell}^{-1}\Omega_{\ell}$.}
The rest of $D(x_{0,\infty},r)\setminus A_{\ell}$ is the set
$\tau_{\ell}^{-1}\Omega_{+,\ell} \cap D(x_{0,\infty},r)$.
$\log H_{\ell}$ is defined in this sector. This time we will write
$w_{\ell}(r) := \tau^{-1}_{\ell} \Omega_{+,\ell} \cap D(x_{0,\infty},r)$.

The imaginary part of $\log H_{\ell}$ is bounded by $\pi$ on
$w_{\ell}(r)$ and the integral of the real part of
$\log H_{\ell} = \log\tau_{\ell}^2 +
\phi_{-,\ell}$ over the set
\[ \mathfrak{Q}(\lambda,\ell) := \left\{u\in\Omega_{-,\ell} :\: \Re H_{\ell}(u) <
\lambda, |\Im H_{\ell}(u)|<\pi\right\} \]
tend to $0$ uniformly with respect to
$\ell$, cf. Lemma~\ref{lem:12ma,1}. Since for every $\lambda$ there
is $r(\lambda)>0$ such that for all $\ell$
$w_{\ell}\left(r(\lambda)\right) \subset \mathfrak{Q}(\lambda,\ell)$,
the integral can be made arbitrarily small for all $\ell$ by making $r$ small
enough.

Proposition~\ref{prop:11xp,1} has been proved.

\paragraph{Proof of Theorem~\ref{theo:7np,1}.}
\subparagraph{Convergence of the right-hand side.}
From formula~(\ref{equ:11xp,2}) and
Propositions~\ref{prop:10xp,1},~\ref{prop:11xp,1} in the case of
$\ell=\infty$ we conclude that for every $\epsilon>0$ there is
$r_{\infty}(\epsilon)>0$ such that whenever
$0<\rho<r<r_{\infty}(\epsilon)$,  then

\[ \left| \int_{\{ u:\: \rho<|u-x_{0,\infty}|<r\}\cap A_{\infty}}
\Re\frac{H_{\infty}(u)}{u}\gamma_{\infty}(u)\,d\Leb_2(u) \right| < \epsilon .\]

Hence, the limit of right-hand side of the formula in Theorem~\ref{theo:7np,1}
exists and will be denoted with $\vartheta_{\infty}$.

\subparagraph{Convergence for finite $\ell$.}
For finite $\ell$, the same Propositions and
formula~(\ref{equ:11xp,1}) it follows that for every $\epsilon>0$
there are $r(\epsilon)>0$, then for every $0<r\leq r(\epsilon)$ there is
$\ell(r)$ such that for all
$\ell\geq\ell(r)$

\begin{equation}\label{equ:14xp,1}
  \left|\vartheta_{\ell} + \frac{1}{\log\tau_{\ell}} \int_{A_{\ell}\setminus
  D(x_{0,\infty},r)} \Re
\frac{H_{\ell}(u)}{u}\gamma_{\ell}(u)\,d\Leb_2(u)\right| <\epsilon .
\end{equation}

\subparagraph{The link between finite and infinite $\ell$.}
Now take $\rho :\: r(\epsilon)\geq\rho>0$ such that
\begin{equation}\label{equ:14xp,2}
  \left| \vartheta_{\infty} + \frac{1}{\log\tau_{\infty}} \int_{A_{\infty}\setminus D(x_{0,\infty},\rho)}
\Re\frac{H_{\infty}(u)}{u}\gamma_{\infty}(u)\,d\Leb_2(u) \right| <
\epsilon .
\end{equation}

By estimate~(\ref{equ:17np,1}) there is $\hat{\ell}(\rho)<\infty$
such that if $\ell\geq\hat{\ell}(\rho)$, then
\begin{multline*}
 \left| - \frac{1}{\log\tau_{\ell}} \int_{ A_{\ell}\setminus D(x_{0,\infty},\rho)}
 \Re\frac{H_{\ell}(u)}{u}\gamma_{\ell}(u)\,d\Leb_2(u) +\right. \\
 \left.\frac{1}{\log\tau_{\infty}} \int_{ A_{\infty}\setminus D(x_{0,\infty},\rho)}
\Re\frac{H_{\infty}(u)}{u}\gamma_{\infty}(u)\,d\Leb_2(u) \right| < \epsilon .
\end{multline*}

When $\ell\geq\max\left(\ell(\rho),\hat{\ell}(\rho)\right)$, from the
estimate above and~(\ref{equ:14xp,1}) we conclude that
\[ \left |\vartheta_{\ell}  +  \frac{1}{\log\tau_{\infty}} \int_{ A_{\infty}\setminus D(x_{0,\infty},\rho)}
\Re\frac{H_{\infty}(u)}{u}\gamma_{\infty}(u)\,d\Leb_2(u) \right| <
2\epsilon . \]
When~(\ref{equ:14xp,2}) is taken into account, we get
$|\vartheta_{\ell}-\vartheta_{\infty}| <3\epsilon$ which ends the
proof.

\subsection{Main Theorems.}
Theorem~\ref{m1} follows from the first claim of
Theorem~\ref{theo:28mp,1}. The convergence claim in Theorem~\ref{m2}
follows from the convergence in Theorem~\ref{theo:28mp,1} and the
drift formula from Theorem~\ref{theo:7np,1}.

\end{document}